\documentclass[]{siamart190516}
\usepackage{amsmath,amssymb}
\usepackage{cite}
\usepackage{color}
\usepackage{enumerate}  
\usepackage{multicol}
\usepackage{ulem}
\usepackage{url}

\newcommand{\mhalf}{M^{\frac{1}{2}}}
\newcommand{\mhalfh}{M^{\frac{1}{2}}_h}

\def \R{\mathbb{R}}

\def \b{\beta}

\def \e{\varepsilon}

\def \W{\Omega}
\def \phi{\varphi}

\def \1{\mathbbm{1}}

\def \<{\left<}
\def \>{\right>}

\def \mA{\mathcal{A}}
\def \mB{\mathcal{B}}
\def \mC{\mathcal{C}}
\def \mD{\mathcal{D}}
\def \mM{\mathcal{M}}
\def \mQ{\mathcal{Q}}
\def \mR{\mathcal{R}}
\def \mT{\mathcal{T}}
\def \Th{\mathcal{T}_h}

\def \EIx{\mathcal{E}_{x,h}^I}
\def \EIv{\mathcal{E}_{v,h}^I}

\def \grad{\nabla}
\def \lss{\lesssim}

\def \wto{\rightharpoonup}
\def \mSh{\mathcal{S}_h^\beta}
\def \mPh{\mathcal{P}_h}
\def \h1hg{H_h^1(\W_x)}
\def \Cwr{C_{\omega,r}}

\DeclareMathOperator\erf{erf}
\def \dx[#1]{\ensuremath{\operatorname{d}\!{#1}}}
\def \wto{\rightharpoonup}

\def \lavg{\{\!\!\{}
\def \ravg{\}\!\!\}}
\def \ljmp{[\![}
\def \rjmp{]\!]}

\newtheorem{defn}[theorem]{Definition}
\newtheorem{remark}[theorem]{Remark}
\newtheorem{prob}[theorem]{Problem}

\newtheorem{assumption}[theorem]{Assumption}
\numberwithin{equation}{section}

\title{Asymptotic Preserving Discontinuous Galerkin Methods for a Linear Boltzmann Semiconductor Model\thanks{This material is based upon work supported by the U.S. Department of Energy, Office of Science, Office of Advanced Scientific Computing Research, as part of their Applied Mathematics Research Program. The work was performed at the Oak Ridge National Laboratory, which is managed by UT-Battelle, LLC under Contract No. De-AC05-00OR22725. The United States Government retains and the publisher, by accepting the article for publication, acknowledges that the United States Government retains a non-exclusive, paid-up, irrevocable, world-wide license to publish or reproduce the published form of this manuscript, or allow others to do so, for the United States Government purposes. The Department of Energy will provide public access to these results of federally sponsored research in accordance with the DOE Public Access Plan (http://energy.gov/downloads/doe-public-access-plan).
}}

\author{Victor P.\ DeCaria\thanks{Mathematics in Computation Section, Computer Science and Mathematics Division, Oak Ridge National Laboratory, Oak Ridge, TN 37831, USA(\email{vpdecaria@gmail.com}).}
\and Cory D.\ Hauck\thanks{Mathematics in Computation Section, Computer Science and Mathematics Division, Oak Ridge National Laboratory, Oak Ridge, TN 37831, USA and 
Mathematics Department, University of Tennessee, Knoxville, TN 37996, USA (\email{hauckc@ornl.gov}).} 
\and Stefan R.\ Schnake\thanks{Mathematics in Computation Section, Computer Science and Mathematics Division, Oak Ridge National Laboratory, Oak Ridge, TN 37831, USA  (\email{schnakesr@ornl.gov}).}}

\begin{document}

\maketitle

\begin{abstract}
    A key property of the linear Boltzmann semiconductor model is that as the collision frequency tends to infinity, the phase space density $f = f(x,v,t)$ converges to an isotropic function $M(v)\rho(x,t)$, called the drift-diffusion limit, where $M$ is a Maxwellian and the physical density $\rho$ satisfies a second-order parabolic PDE known as the drift-diffusion equation.  Numerical approximations that mirror this property are said to be asymptotic preserving.  In this paper we build two discontinuous Galerkin methods to the semiconductor model: one with the standard upwinding flux and the other with a $\e$-scaled Lax-Friedrichs flux, where 1/$\e$ is the scale of the collision frequency.  We show that these schemes are uniformly stable in $\e$ and are asymptotic preserving.  In particular, we discuss what properties the discrete Maxwellian must satisfy in order for the schemes to converge in $\e$ to an accurate $h$-approximation of the drift diffusion limit.  Discrete versions of the drift-diffusion equation and error estimates in several norms with respect to $\e$ and the spacial resolution are also included.
\end{abstract} 



\section{Introduction}

Kinetic equations are an established tool for modeling charged-particle transport in semiconductors, particularly in non-equilibrium settings \cite{markowich2012semiconductor,poupaud1994mathematical,jungel2009transport}.  However, numerical simulations of such equations are known to be challenging, due to the size of the space on which they are defined (in general three position, three momentum variables, plus time) and the multiscale nature of the equations.  With regards to the latter, it is well-known that for large collision frequencies and long-time scales, the kinetic solution is well-approximated by a drift-diffusion equation which depends on space and time only.   Under reasonable conditions, this limit was established rigorously for the case of an applied electric field in \cite{Pou1991}.  The case of a self-consistent field was later treated in \cite{abdallah2004diffusion, masmoudi2007diffusion}. 

Because of the drift-diffusion approximation, solving a kinetic model of charge transport in collisional regimes may be unnecessarily expensive; to ameliorate this cost, methods which leverage the drift-diffusion approximation, either via domain decomposition \cite{klar1998asymptotic} or acceleration \cite{VDC2021,laiu2020fast} are sometimes used.  At a minimum, it is important that a discretization of the kinetic equation recover a stable and consistent discretization of the drift-diffusion limit as the collision frequency becomes infinitely large;  this is the so-called asymptotic preserving (AP) property \cite{jin1999efficient,jin2010asymptotic}.  While standard finite volume or finite-difference methods that that rely on upwinding to discretize advection terms are not asymptotic preserving, there are specialized spatial discretizations \cite{schmeiser1998convergence} and operator splitting techniques \cite{klar1999numerical,jin2000discretization, jin2010asymptotic} that are.

A different approach for capturing the numerical drift-diffusion limit is to use discontinuous Galerkin (DG) methods.  These methods have been developed both for kinetic semiconductor equations \cite{cheng2007discontinuous,cheng2008discontinuous,cheng2011brief, cheng2009discontinuous,majorana2019simulation} and for the drift-diffusion equations \cite{liu2016analysis, chen2020steady}.  While not yet rigorously established in the literature, it is reasonable to assume that DG methods will recover the numerical drift-diffusion limit.  Such a conjecture rests on a similar body of work for kinetic equations of radiation transport.  In that setting, collisional dynamics over long time scales lead to a standard diffusion equation \cite{bensoussan1979boundary,larsen1974asymptotic,habetler1975uniform}.   The asymptotic preserving properties of DG methods for transport equations were first established in \cite{larsen1974asymptotic} for one-dimensional (slab) geometries and later extended to the general multi-dimensional setting in \cite{adams2001discontinuous}.  In \cite{guermond2010asymptotic}, the work in \cite{adams2001discontinuous} was re-established using a rigorous functional analysis framework. The work presented here follows in the spirit of that framework.

In the current paper, we rigorously prove the numerical drift-diffusion limit for a DG method applied to a linear kinetic semiconductor equation.  In particular, the collision operator approximates very complicated material interactions with a simple relaxation model and the electric field is not self-consistent, but rather assumed to be given.  The DG method relies on a reformulation of the kinetic equation in terms of a weighted distribution function. While such a reformulation is likely not necessary, partially due to the numerical results in \cite{laiu2020fast}, it does make some stability results easier to prove.  Such results are challenging because, unlike the radiation transport case, the advection operators and collision operator of the kinetic semiconductor equation are 
stable in $L^2$ spaces with two different weightings.
Even at the analytical level, this mismatch poses significant challenges \cite{masmoudi2007diffusion}.  Even so, we expect that the analysis presented here can be leveraged for ``more standard'' implementations.

Beyond linearity, there are several other assumptions made in the analysis.  Some of these are technical, but others are quite important.  Among these, the most important is a zero-inflow boundary condition which precludes the development of a boundary layer.   We also assume that the initial data is well-prepared in the sense that it is consistent with the state of local thermal equilibrium.  Removing these three assumptions---linearity, zero inflow, and well-prepared initial data---will be important steps in future work.  In addition, uniform error estimates independent of the collision frequency, along the lines of \cite{sheng2021uniform} for the radiation transport case, should be considered. However, the analysis here is already fairly involved and requires more work than the radiation transport case. The main novelty of this work is the rigorous analysis of the numerical diffusion limit of a kinetic equation.  In contrast to  \cite{guermond2010asymptotic} whose work and novelty closely resembles and inspires the work here, the kinetic equation in the current setting is time-dependent, involves advection in both the physical and velocity variables, is defined over an unbounded velocity domain, and has a collision operator with a kernel (the local thermal equilibrium) that is not contained in a standard finite element space.  Additionally, this work provides several lemmas concerning stability and control of projecting discontinuous Galerkin finite element functions onto a continuous Galekrin finite element space.  These technical results will aid in the current and future numerical analysis of AP discontinuous Galerkin schemes.

The remainder of the paper is organized as follows.  In \Cref{sect:prelim}, we introduce the relevant equations, preliminary notation, assumptions used to construct a discrete Maxwellian, and the numerical method for solving the kinetic semiconductor model given in \eqref{eqn:sc} below.  We characterize the collision frequency by an asymptotic parameter $\e >0$ which is inversely proportional to the mean-free-path between collisions, and in \Cref{sect:a_priori_estimates}, we develop stability and pre-compactness estimates that allow us to take the $\e$-limit to 0.  Additionally, we give several technical results which will aid in the general analysis of discrete drift-diffusion limits.  In \Cref{sect:dde-limit}, we show the numerical density $\rho_h^\e$ of the kinetic model converges to the solution of a discretized version $\rho_h^0$ of the drift diffusion system, given in \eqref{eqn:dde-cont-system} below.  In \Cref{sect:error_est}, we show error estimates for $\|\rho_h^\e-\rho_h^0\|$ in $\e$ and $h$ as well as error estimates for $\|\rho_h^0-\rho^0\|$ in $h$.  This allows us to build estimates for $\|\rho_h^\e-\rho_0\|$ in $\e$ and $h$.
\section{Background, Preliminaries, and Assumptions}\label{sect:prelim}

Given $\e>0$, a Lipschitz spatial domain $\Omega_x \subset \mathbb R^3$, and data $f_0$ prescribed on $\Omega_x$, let $f_\e(x,v,t)$ be the solution of the following kinetic semiconductor model
\begin{subequations}
\label{eqn:sc}
\begin{alignat}{3}
\e\frac{\partial f^\e}{\partial t} + v\cdot\grad_x f^\e + E(x,t)\cdot\grad_vf^\e - \frac{1}{\e}Q(f^\e)= 0,&\qquad  && (x,v) \in \W_x\times\R^3 ,\, t >0;
\label{eqn:sc_eq}\\
f^\e(x,v,t) = f_{-}(x,v,t),&\qquad  &&(x,v)\in\partial\W_- ,\,t > 0; \label{eqn:sc_eq_inflow} \\
f^\e(x,v,0) = f_0(x,v),&\qquad  && (x,v) \in \W_x\times\R^3, \label{eqn:sc_eq_initial}
\end{alignat}
\end{subequations}
where $E\in W^{1,\infty}([0,T];L^{\infty}(\W_x))$ is a given electric field, $f_{-}$ is the inflow data, and 
\begin{equation}
\partial\W_- = \{(x,v)\in \partial\W_x\times \R^3:v\cdot n_{x}(x) < 0 \},
\end{equation}
with $n_{x}(x_0)$ being the normal to $\Omega_x$ at the point $x_0$, is the inflow component of the boundary.  Additionally the collision operator $Q$ is defined by
\begin{equation}
\label{eqn:Q_rho_M}
Q(f^\e)=\omega(M\rho^\e-f^\e)
\end{equation}
where
\begin{equation}\label{eqn:rho_M}
\qquad 
\rho^\e(x,t) = \int_{\mathbb R^3} f^\e(x,v,t) \dx[v];
\qquad 
M(v) = (2\pi\theta)^{-3/2}e^{-|v|^2/2\theta}
\end{equation}
with $\theta >0$ the (lattice) temperature; and $\omega\in L^{\infty}(\W_x)$ with $0<\omega_{\min}\leq \omega$ on $\W_x$ is the (scaled) collision frequency.

\begin{defn} The function $\rho^\e$ defined in \eqref{eqn:rho_M} is the number density and
\begin{align}
    J^\e = \frac{1}{\e} \int_{\mathbb R^3} v f^\e(x,v,t) \dx[v]
\end{align}
is the current density.
\end{defn}

It has been shown in \cite{Pou1991} that if the inflow and initial data are isotropic, that is, $f_{-}(x,v,t) = m(x)M(v)$ and $f_0(x,v)=\rho_0(x,t)M(v)$, then as $\e\to 0$, $f^\e$ converges to $M(v)\rho^0(x)$, where $\rho^0$ solves the drift-diffusion equation:
\begin{subequations}\label{eqn:dde-cont}
\begin{alignat}{3}
    \frac{\partial \rho^0}{\partial t}+\nabla_x \cdot \left(\frac{1}{\omega}(-\theta\grad_x\rho^0+E\rho^0)\right) = 0&, \quad && x \in \W_x,\, t > 0;\\
    \rho^0(x,t) = m(x)&, \qquad && x \in \partial\W_x,\, t > 0;\\
    \rho^0(x,0) = \rho_0(x)&, \qquad && x \in \W_x.
\end{alignat}
\end{subequations}
Let $J^0=\frac{1}{\omega}(-\theta\grad_x\rho^0+E\rho^0)$.  Then the pair $\{\rho^0,J^0\}$ solves the equivalent first-order system: 
\begin{subequations}
\label{eqn:dde-cont-system}
\begin{alignat}{3}
    \frac{\partial \rho^0}{\partial t}+\nabla_x \cdot J^0 = 0&,\qquad && x \in \W_x,\, t > 0;\\
    \omega J^0 + \theta\grad_x\rho^0-E\rho^0 = 0&,\qquad && x \in \W_x,\, t > 0; \\
    \rho^0(x,t) = m(x) &, \qquad && x \in \partial\W_x,\, t > 0; \\
    \rho^0(x,0) = \rho_0(x) &, \qquad && x\in \W_x.
\end{alignat}
\end{subequations}

\subsection{Notation}
\label{subsec:notation}

Given a measureable open set $D \subset \R^3$, let $L^2(D)$ and $W^{k,p}(D)$ be the standard Lebesgue and Sobolev spaces of functions on $D$ and let $H^k(D) := W^{k,2}(D)$.  When $D$ is a volume (a three-dimensional manifold) in $\R^3$, we use  $(\cdot,\cdot)_D$ to denote the standard $L^2$ inner product with respect to the Lebesgue measure $\dx[x]$.  If $D$ is a surface (a two-dimensional manifold) in $\R^3$, we use $\left<\cdot,\cdot\right>_{D}$ to denote the $L^2$ inner product with respect to the Lebesgue measure on the surface.  These inner products can be extended to vector valued functions in a natural way by use of the Euclidean inner product.

To discretize \eqref{eqn:sc}, we  first restrict the domain in $v$.  Given $L >0$, let $\W_v = [-L,L]^3$ and define $\W = \W_x\times\W_v$.  Given a mesh parameters $h_x>0$, let $\mathcal{T}_{x,h}:=\mathcal{T}_{x,h_x}$ be a mesh on $\W_x$ constructed from open polyhedral cells $K$ of maximum diameter $h_x$, and let $\EIx$ be the interior skeleton of $\mathcal{T}_{x,h}$, i.e., the set of edges $e \subset \partial K \not \subset \partial \W_x$.  Similarly, given $h_v>0$, let  $\mathcal{T}_{v,h}:=\mathcal{T}_{v,h_v}$ and $\EIv$ be a mesh and interior skeleton for $\W_v$ respectively.  We assume that $\mathcal{T}_{x,h}$ is quasi-uniform and shape regular.  The conditions of $\mathcal{T}_{v,h}$ are given in \Cref{subsect:discret_Max}.

Given an edge $e\in\EIx=\partial K^+\cap\partial K^-$ for some $K^+,K^-\in\mathcal{T}_{x,h}$, let $z \in L^2(\W_x)$ and $\tau\in [L^2(\W_x)]^3$ be scalar and vector-valued functions, respectively, each with well-defined traces on $K^+$ and $K^-$.  For such functions, we define the average and jump methods 
\begin{subequations}
\begin{alignat}{3}
\lavg z\ravg &= \frac{1}{2}\left(z\big|_{K^+} + z\big|_{K^-}\right),&\qquad \ljmp z\rjmp &= z\big|_{K^+}{n_x^+} + z\big|_{K^-}n_x^- ,\\
    \lavg\tau\ravg &= \frac{1}{2}(\tau|_{K^+}+\tau|_{K^-}),&\qquad  \ljmp\tau\rjmp &= \tau|_{K^+}\cdot n_x^+ + \tau|_{K^-}\cdot n_x^- ,
\end{alignat}
\end{subequations}
where $n_x^\pm$ are the unit normal vectors pointing outward from $K^\pm$, respectively.  These definitions can be modified to average and jumps in the $v$-direction in a natural way with unit normal vector $n_v$.

For ease of presentation we will use $a\lss b$ to denote $a\leq Cb$ where $C>0$ is a constant independent of $h_x$ and $\e$.  
The constant additionally depends on the data $\omega$ and $E$ (see \cref{ass:data}), the final time $T$,  $\W_x$, $L$, $h_x$-independent mesh parameters of $\mathcal{T}_{x,h}$, and the discrete Maxwellian discussed in \Cref{subsect:discret_Max} which depends on $h_v$, $L$, and $\theta$. 

Given integers $k_x \geq 0$ and $k_v \geq 0$, let 
\begin{align}
    V_{x,h} &= \{z\in L^2(\W_x):z\big|_T \in \mathbb{Q}_{k_x}(T)\ \forall K\in\mathcal{T}_{x,h} \}
    \\
    V_{v,h} &= \{z\in L^2(\W_v):z\big|_T \in \mathbb{Q}_{k_v}(T)\ \forall K\in\mathcal{T}_{v,h} \},
\end{align}
where $\mathbb{Q}_{k}(T)$ is the set of all polynomials on $K$ with $k$ being the maximum degree in any variable, and let
\begin{equation}
    V_h=V_{x,h}\otimes V_{v,h}
\end{equation}
be the tensor DG discrete space. For purposes of this paper, we assume $k_x\geq 0$ and $k_v\geq 1$.%
\footnote{The assumption on $k_v$ is to enable the construction of the discrete Maxwellian;  see  \Cref{subsect:discret_Max}.}
For any function $z_h\in V_h$, let $\grad_x z_h \subset V_h$ and $\grad_v z_h \subset V_h$ denote piece-wise gradients defined on $K$ for all $T\in\Th$.%
\footnote{The discrete gradient ignores the jumps in $z_h$ across the boundary, but agrees with the standard definition of gradient for continuous functions.} 

For the discretization in $x$, some addition notation is needed.  Let
\begin{equation}
V_{x,h}^0 =\{q_h\in V_{x,h}:q_h\big|_{\partial\W_x} = 0\}
\end{equation}  
be the space of DG functions with vanishing trace, and let
\begin{equation}
    S_{x,h}=V_{x,h}\cap C^0(\overline{\W_x})
    \quad\text{and}\quad
    S_{x,h}^0 =V_{x,h}^0\cap C^0(\overline{\W_x})
\end{equation}
be the continuous finite element analogues to the DG spaces $V_{x,h}$ and $V_{x,h}^0$, respectively.   

Below we consider two discretizations parameterized by an integer $\beta \in \{0,1\}$ that determines the type of numerical flux used and, consequently, the space that the discrete drift-diffusion limit will live; see \Cref{rmk:method}.  Let $\mSh$ be an
$L^2$-orthogonal projection operator from $L^2(\W_x)$ onto $V_{x,h}^0$ if $\beta=1$ and from $L^2(\W_x)$ onto $S_{x,h}^0$ if $\beta=0$.   Moreover, let $(S_{x,h}^0)^*$ and $(V_{x,h}^0)^*$ denote the topological dual of $S_{x,h}^0$ and $V_{x,h}^0$ respectively and let $\wto$ represent convergence in the weak topology.  Additionally, for $\beta \in \{0,1\}$ define the discrete dual-norm $H_{h,\beta}^{-1}(\W_x)$ by
\begin{equation}
\label{eqn:dual_norms}
\|z_h\|_{H_{h,0}^{-1}(\W_x)} = \sup_{\substack{q_h\in S_{x,h}^0\\ q_h\neq 0}}\frac{(z_h,q_h)_{\W_x}}{\|\grad_x q_h\|_{L^2(\W_x)}}
\quad \text{or}\quad 
\|z_h\|_{H_{h,1}^{-1}(\W_x)} = \sup_{\substack{q_h\in V_{x,h}^0\\ q_h\neq 0}}\frac{(z_h,q_h)_{\W_x}}{\| q_h\|_{H_h^1(\W_x)}}
\end{equation}
where $\|\cdot\|_{H_h^1(\W_x)}$ is a discrete $H^1$ norm:
\begin{equation}
\|q_h\|_{H_h^1(\W_x)}^2 = \|\grad_x q_h\|_{L^2(\W_x)}^2 + \frac{1}{h_x}\|\ljmp q_h\rjmp\|_{L^2(\EIx)}^2 + \frac{1}{h_x}\|q_h\|_{L^2(\partial\W_x)}^2.
\end{equation}
A discrete Poincar\'e-Friedrichs inequality \cite[Theorem 10.6.12]{Bre2008} yields
\begin{align}\label{eqn:pf_ineq}
 \|q_h\|_{L^2(\W_x)} \lss \|q_h\|_{H_h^1(\W_x)}.
\end{align}

Given a Banach space $X$, $1\leq p\leq \infty$, and the final time $T$, we let $L_T^p(X):= L^p([0,T];X)$, $C^0([0,T];X)$, and $H^1([0,T];X)$ be the standard $L^p/C^0/H^1$ spaces of Banach-space valued functions with Bochner integration.

Finally we will often write $\frac{\partial}{\partial t}$ as $\partial_t$ in order to keep the spacing consistent in longer estimates.  Both will be used interchangeably.  

\subsection{Alternate form of the PDE}\label{subsect:alternate_pde}

It is easy to show that the collision operator $Q$, defined in \eqref{eqn:Q_rho_M}, is semi-coercive in the weighted norm $\|M^{-\frac{1}{2}} (\cdot) \|_{L^2(\W)}$. Indeed, testing $Q$ by $M^{-1} f^\e$ gives
\begin{align}
 -( M^{-1} f^\e, Q  (f^\e))_\W
&= \| \omega^{\frac{1}{2}} M^{-\frac{1}{2}} (f^\e - M\rho^ \e) \|_{L^2(\W)}^2.
\end{align}
This structure is critical to achieving the drift-diffusion limit.  However since standard discretizations of \eqref{eqn:sc} do not allow test functions with an $M^{-1}$ weight, we instead rewrite \eqref{eqn:sc} in terms of the weighted distribution  $g^\e=M^{-\frac{1}{2}} f^\e$:
\begin{subequations}
\label{eqn:alternate}
\begin{alignat}{3}
\e \frac{\partial g^\e}{\partial t}+ v\cdot\grad_x g^\e + E(x,t)\cdot\grad_v g^\e 
- \frac{\omega}{\e}\left( \mhalf \rho^\e - g^\e \right) =  \frac{1}{2\theta}E(x,t)\cdot v g^\e,&\qquad  && (x,v) \in \W_x\times\R^3 ,\, t >0;
\label{eqn:alternate_eq} \\
g^\e(x,v,t) = f_{-}(x,v,t)/M^\frac{1}{2}(v),&\qquad  &&(x,v)\in\partial\W_- ,\,t > 0; \\
g^\e(x,v,0) = f_0(x,v)/M^\frac{1}{2}(v),&\qquad  && (x,v) \in \W_x\times\R^3,
\end{alignat}
\end{subequations}
where, in terms of $g^\e$, $\rho^\e = (\mhalf,g^\e)_{\R^3}$.
Since $\|g^\e\|_{L^2(\W)} = \|M^{-\frac{1}{2}}f^\e\|_{L^2(\W)}$, the weighted collision operator
\begin{equation}
   M^{-\frac12} Q (M^{\frac12} g^\e) = \omega \left( \mhalf \rho^\e - g^\e \right) 
\end{equation}
will be $L^2$-coercive and symmetric as a function of $g^\e$.  We refer to the function $\mhalf g^\e$ as the weighted equilibrium.   The cost of this additional structure is the electric field term on the right-hand side of \eqref{eqn:alternate_eq}.

\subsection{Construction of Discrete Maxwellian}
\label{subsect:discret_Max}

In order to recover the proper drift-diffusion limit, we need to construct a suitable discrete Maxwellian on the bounded domain $\Omega_v$.  This is done via an approximation of the square root of the one-dimensional Maxwellian.  Assume that $\mT_{v,h}$ is a tensor product mesh, i.e., $\mT_{v,h} = \mT_{v,h}^1\otimes \cdots \otimes \mT_{v,h}^3$, and let $M_{h,i}^{\frac{1}{2}}$ be a continuous, strictly positive, piecewise-polynomial approximation of the  one-dimensional root-Maxwellian over $\mT_{v,h}^i$:
\begin{equation}\label{eqn:maxwell_approx}
M_{h,i}^{\frac{1}{2}}(v_i)  \approx M_i^{\frac{1}{2}}(v_i) :=\left(\frac{1}{\sqrt{2\pi\theta}}e^{\frac{-v_i^2}{2\theta}}\right)^{1/2}, \qquad i=1,\dots,3,
\end{equation}
with the following properties:
\begin{assumption}
\label{ass-discrete-root-Max}
For each $i=1,\dots,3$, the function $M_{h,i}^{\frac{1}{2}}$ satisfies the following properties:
\begin{multicols}{2}
\begin{enumerate}[a.]
    \item  $(M_{h,i}^{\frac{1}{2}},M_{h,i}^{\frac{1}{2}})_{[-L,L]} = 1$, \label{ass-mass}
    \item  $M_{h,i}^{\frac{1}{2}}(L) = M_{h,i}^{\frac{1}{2}}(-L)$, \label{ass-symm} 
    \item $(\partial_v M_{h,i}^{\frac{1}{2}},\partial_v  M_{h,i}^{\frac{1}{2}})_{[-L,L]} = \frac{1}{4\theta}$, \label{ass-energy}
    \item   $(\partial_v M_{h,i}^{\frac{1}{2}},M_{h,i}^{\frac{1}{2}})_{[-L,L]} = 0$. \label{ass-momen}
\end{enumerate}
\end{multicols} 
\end{assumption}

\begin{defn}
\label{defn:root-Maxwellian-and-discrete-velocity}
The discrete root-Maxwellian $\mhalfh \in V_{v,h}\cap C^0(\overline{\W_v})$ is
\begin{equation}
    \mhalfh(v) = \prod_{i=1}^3 M_{h,i}^{\frac{1}{2}}(v_i),
\end{equation}
and the discrete velocity $v_h$ is 
\begin{equation}
\label{eqn:v_h}
    v_h = -2\theta\frac{\grad_v\mhalfh}{\mhalfh} = -2\theta \grad_v \log(\mhalfh) .
\end{equation}
\end{defn}
Since $\mhalfh>0$ on $\overline{\W_v}$, it follows that $v_h\in L^{\infty}(\W_v)$.  Moreover $v_h\mhalfh\in V_{v,h}$, even though $v_h\notin V_{v,h}$.

\begin{remark} \label{rmk:method}
In defining the discrete root-Maxwellian:
\begin{enumerate}
    \item The continuity requirement on $M_{h,i}^{\frac{1}{2}}$ is the reason for assumption $k_v \geq 1$ in \Cref{subsec:notation}.
    \item Assumption \ref{ass-discrete-root-Max}.\ref{ass-momen} is not independent, but rather is implied by \Cref{ass-discrete-root-Max}.\ref{ass-symm}.  
    \item If Assumption \ref{ass-discrete-root-Max}.\ref{ass-energy} is not satisfied, then the numerical discretization below will still converge in the $\e$-limit to a discretization of a drift-diffusion system, but  $\theta$ from \eqref{eqn:dde-cont-system} will be $\theta_{h_v}$, see \eqref{eqn:theta_h}, instead of the proper temperature.  This can be seen by substituting $\theta_{h_v}$ for $\theta$ in \Cref{sect:a_priori_estimates,sect:dde-limit}. 
\end{enumerate}
\end{remark}

\begin{remark} \label{rmk:interp_maxwell}
The existence of such a discrete Maxwellian satisfying every assumption given is not discussed. Rather, to create a discrete Maxwellian that satisfies every assumption but \Cref{ass-discrete-root-Max}.\ref{ass-energy}, take $M_{h,i}^{\frac{1}{2}}$ to be the Lagrange piecewise linear nodal interpolant of $M_i^{\frac{1}{2}}$ and scale it to have an $L^2$ norm of 1.  With mild restrictions on $L$ and $h_v$ based on $\theta$, we show in \Cref{lem:M_bnd}, given in the appendix, that $M_{h,i}^{1/2}$ is an $\mathcal{O}(h_v^2)$ approximation to $M_i^{1/2}$ in the $L^2$-norm and an $\mathcal{O}(h_v)$ approximation in the $H^1$-norm.\footnote{We are neglecting the errors due to the finite velocity domain.  See \Cref{lem:M_bnd} for the full estimate.}  While the discrete temperature 
\begin{equation}\label{eqn:theta_h}
    \theta_{h_v} := \tfrac{1}{4}(\partial_v M_{h,i}^{\frac{1}{2}},\partial_v  M_{h,i}^{\frac{1}{2}})_{[-L,L]}^{-1}
\end{equation}
is not exactly $\theta$, is it readily seen from \Cref{lem:M_bnd} that $\theta_{h_v}$ is an $\mathcal{O}(h_v)$ approximation to $\theta$.
\end{remark}

\subsection{The Numerical Method}

We now give our numerical method for \eqref{eqn:alternate}.  
\begin{prob}
\label{prob:discr}Find $g_h^\e\in H^1([0,T];V_h)$  such that
\begin{subequations}
\label{eqn:alt_forms}
\begin{align}
\e\left(\frac{\partial g_h^\e}{\partial t},z_h \right)_\W + \mathcal{A}(g_h^\e,z_h) + \mathcal{B}(g_h^\e,z_h) + \mathcal{D}(g_h^\e,z_h) - \frac{1}{\e}\mathcal{Q}(g_h^\e,z_h) &= \mathcal{C}(g_h^\e,z_h) + \mR(z_h), \label{eqn:alt_forms_eq}\\
g_h^\e(0):=g_{0,h}   \label{eqn:alt_forms_ic}
\end{align}
\end{subequations}
for all $z_h\in V_h$ and a.e\ $0<t\leq T$ where
\begin{subequations}
\label{eqn:forms}
\begin{align}
    \mathcal{A}(w_h,z_h) &= -(v_hw_h,\grad_x z_h)_\W + \left< v_h\lavg w_h\ravg+\e^\beta\frac{|v_h\cdot{n_x}|}{2}\ljmp w_h\rjmp,\ljmp z_h\rjmp\right>_{\EIx\times\W_v} \nonumber \\
    &\quad+ \left<v_hw_h,{n_x}z_h\right>_{\partial\W_+},\label{eqn:A}\\
    \mathcal{B}(w_h,z_h) &= -(Ew_h,\grad_v z_h)_\W + \left< E\lavg w_h\ravg+\frac{|E\cdot{n_v}|}{2}\ljmp w_h\rjmp,\ljmp z_h\rjmp\right>_{\W_x\times\EIv},\label{eqn:B}\\
    \mathcal{D}(w_h,z_h) &= \left< E\mhalfh P(w_h),{n_v} z_h\right>_{\W_x\times\partial\W_v} \label{eqn:D}\\
    P(w_h) &= (\mhalfh, w_h)_{\W_v}\\
    \mathcal{Q}(w_h,z_h) &= \left(\omega (\mhalfh P(w_h)-w_h),z_h\right)_\W,\label{eqn:Q}\\
    \mathcal{C}(w_h,z_h) &= \frac{1}{2\theta}(E\cdot v_h w_h, z_h)_\W,\label{eqn:C}\\
    \mR(z_h) &= -\left<v_h g_{-\!,h},{n_x}z_h\right>_{\partial\W_-},\label{eqn:R}
\end{align}
\end{subequations}
and the functions $g_{-\!,h}\in V_h$ and $g_{0,h}\in V_h$ are the discrete inflow and initial data respectively.
\end{prob}

\begin{defn} The discrete number density $\rho_h^\e$ and current density $J_h^\e$ are given by
\begin{align}
    \rho_h^\e &= P(g_h^\e) = (\mhalfh, g_h^\e)_{\W_v},  \label{eqn:rho_h}\\
    J_h^\e &= \frac{1}{\e}P(v_h g_h^\e) 
    = \frac{1}{\e} (\mhalfh v_h, g_h^\e)_{\W_v}.\label{eqn:J_h}
\end{align}
\end{defn}

\begin{remark}
In \Cref{prob:discr},
\begin{enumerate}
    \item The bilinear form $\mA$ and functional $\mR$ are the result of the discretizations of the operator  $v\cdot\grad_x$ with $v$ replaced by $v_h$.  The parameter $\beta$ is a switch between a standard upwinding and a scaled upwinding flux.  If $\b=0$, the flux $\widehat{v_hg}$ is the standard upwinding flux, namely
    \[
        \widehat{v_hg} = 
        \begin{cases}
            v_h\lavg g\ravg + \tfrac{|v_h\cdot{n_x}|}{2}\ljmp g\rjmp &\text{ on } \EIx\times\W_v \\
            v_hg &\text{ on } \partial\W_+ \\
            v_h g_{-\!,h} &\text{ on } \partial\W_-
        \end{cases}
    \]
    In this case, $\rho_h^\e$ will converge to a continuous finite element function as $\e \to 0$; see Section \ref{sect:dde-limit}.  As a result, locking will occur when 
    $k_x = 0$, i.e. when $V_{x,h}$ is comprised of piecewise constant functions in $x$.  If $\b=1$, the flux $\widehat{v_hg}$ instead contains an $\e$-scaled Lax-Friedrichs penalty in the jump.  This modification yields a LDG-like discretization of the drift-diffusion equations; see Section \ref{sect:dde-limit}.
    \item The bilinear forms $\mB$ and $\mD$ are constructed using standard upwind fluxes for the Vlasov operator $E\cdot\grad_v$ but, due to the velocity boundary domain restriction, we weakly impose the boundary condition $g_h^\e=\mhalfh\rho_h^\e$ on $\W_x\times\partial\W_v$.  This provides two benefits.  First, the density $g_h^\e$ will not lose or gain mass out of the velocity boundary. Second, this boundary condition keeps the restriction of $v$ to the bounded domain $\Omega_v$ from polluting the discrete drift-diffusion limit.
    \item The bilinear form $\mQ$ is a standard discretization of the collision operator.  
    \item The bilinear form $\mC$ is a standard discretization of the term $E\cdot v g$, with $v$ replaced by $v_h$.
\end{enumerate}
\end{remark}

\subsection{Other Assumptions}\label{subsec:other_assumpts}

Here we collect any other assumptions used in the analysis of \Cref{prob:discr}.

\begin{assumption}\label{ass:data}
The collision frequency $\omega$ and the electric field $E$ in \eqref{eqn:sc} are specified as $\omega\in L^\infty(\W_x)$ with $0<\omega_{\min}\leq\omega$ on $\W_x$ and $E\in W^{1,\infty}([0,T];L^{\infty}(\W_x))$.  
\end{assumption}

\begin{assumption}\label{ass:bc}
The inflow data in \eqref{eqn:sc} is isotropic, that is, $f_0(x,v)=\rho_0(x)M(v)$.  Additionally we require $g_{0,h}$ in \Cref{prob:discr} to be discretely isotropic, that is, $g_{0,h}=\rho_{0,h}\mhalfh$ where $\rho_{0,h}\in V_{x,h}$ is defined by
\[
(\rho_{0,h},q_h)_{\W_x} = (\rho_0,q_h)_{\W_x}\quad\forall q_h\in V_{x,h}.
\]
\end{assumption}

\Cref{ass:bc} is a common assumption made in the study of the drift-diffusion limit on the continuous PDE \eqref{eqn:sc}; see \cite{masmoudi2007diffusion,Pou1991}.

It follows from \Cref{ass:bc}, \Cref{ass-discrete-root-Max}.\ref{ass-mass}, and \Cref{ass-discrete-root-Max}.\ref{ass-momen} that 
\begin{align}
    \rho_h^\e |_{t=0} &= (\mhalfh, g_{0,h})_{\W_v} = (\mhalfh, \mhalfh)_{\W_v}\rho_{0,h} = \rho_{0,h}, \\
    J_h^\e |_{t=0} &= \label{eqn:init_J} \frac{1}{\e}(v_h\mhalfh,g_{0,h})_{\W_v} = \frac{1}{\e}(v_h\mhalfh,\mhalfh)_{\W_v} \rho_{0,h} = 0.
\end{align}

\begin{assumption}
\label{ass:zero_incoming}
The continuous incoming data $f_{-}$ in \eqref{eqn:sc_eq_inflow}, the discrete incoming data $g_{-,h}$ in \Cref{prob:discr} and, consequently, the functional $\mathcal{R}$ in \Cref{prob:discr} are identically zero.
\end{assumption}
Assumptions like \Cref{ass:zero_incoming} are commonly made in the analysis of diffusion limits to avoid handling complications due to boundary layers. While it is not expected that the equilibrium boundary condition will induce a boundary layer, the case of non-zero incoming data will be treated elsewhere.

\begin{assumption}\label{ass:inv_laplace_bnd}
The inverse Laplacian $S:L^2(\W_x)\to H_0^1(\W_x)$ defined by
\begin{equation}
    (\grad Sq,\grad w)_{\W_x} = (q,w)_{\W_x} \quad\forall w\in H_0^1(\W),
\end{equation}
is a bounded linear operator from $L^2(\W_x)\to H^2(\W_x)\cap H_0^1(\W_x)$, that is,
\begin{equation}
    \|Sq\|_{H^2(\W_x)} \lss \|q\|_{L^2(\W_x)}
\end{equation}
for any $q\in L^2(\W_x)$.
\end{assumption}

\Cref{ass:inv_laplace_bnd} is needed for the proof of stability of the $L^2$ projection $\mSh$ in the $H_h^1(\W_x)$-norm (see \Cref{ass:proj-stable}).  If $\W_x$ is convex, then \Cref{ass:inv_laplace_bnd} is automatically satisfied \cite[Section 3.2]{grisvard2011elliptic}.

\begin{assumption}\label{ass:eps_h_relation}
There is a constant $C>0$ such that $\frac{\e}{h_x}<C$, that is, $h_x$ cannot go to zero faster than $\e$.
\end{assumption}

Assumption \ref{ass:eps_h_relation} is to improve the readability of the paper.  There are several estimates that include bounds of $1+\frac{\e}{h_x}$ and these bounds will not improve any rates of convergence regardless of the choice of $\e$ and $h_x$.  Thus with Assumption \ref{ass:eps_h_relation} we have the bounds $1+\frac{\e}{h_x}\lss 1$ and $1+\sqrt{\frac{\e}{h_x}}\lss 1$.

\section{A Priori Estimates}\label{sect:a_priori_estimates}

In this section we develop space-time stability estimates for $g_h^\e$, the number density $\rho_h^\e$, and the current density $J_h^\e$ in \Cref{prob:discr}.

\subsection{Preliminary Estimates and Identities}

In this subsection we list an inverse and trace inequality, derived from the standard estimates in \cite{Riv2008}, technical interpolation and projections estimates, and a useful integration by parts identity.

\begin{lemma}[Trace Inequality]\label{lem:trace}
For any $z_h\in V_{h}$ and $q_h\in V_{x,h}$ we have
\begin{align}
    \|\ljmp z_h\rjmp\|_{L^2(\EIx\times\W_v)}^2 + \|z_h\|_{L^2(\partial\W_x\times\W_v)}^2 &\leq \frac{C}{h_x}\|z_h\|_{L^2(\W)}^2, \label{eqn:trace_xv}\\
        \|\ljmp q_h\rjmp\|_{L^2(\EIx)}^2 + \|q_h\|_{L^2(\partial\W_x)}^2 &\leq \frac{C}{h_x}\|q_h\|_{L^2(\W_x)}^2, \label{eqn:trace} \\
    \|\ljmp z_h\rjmp\|_{L^2(\W_x\times\EIv)}^2 + \|z_h\|_{L^2(\W_x\times\partial\W_v)}^2 &\leq \frac{C}{h_v}\|z_h\|_{L^2(\W)}^2. \label{eqn:trace_vx}
\end{align}
Here $C>0$ is an $\e$, $h_x$, and $h_v$-independent constant.  $C$ depends on the polynomial degree of $V_{h}$ and other $h_x$ and $h_v$-independent mesh parameters of $\mathcal{T}_{x,h}$ and $\mathcal{T}_{v,h}$.
\end{lemma}

\begin{lemma}[Inverse Inequality]\label{lem:inverse}
For any $q_h\in V_{x,h}$ we have
\begin{align}\label{eqn:inverse}
    \|\grad q_h\|_{L^2(\W_x)} \leq C h_x^{-1}\|q_h\|_{L^2(\W_x)}
\end{align}
where $C>0$ is some $\e$ and $h_x$-independent constant that depends on the the polynomial degree of $V_{x,h}$. 
\end{lemma}

\begin{lemma}[Integration by Parts]  \label{lem:parts}
For any $q_h\in V_{x,h}$ and $\tau_h\in [V_{x,h}]^3$, there holds
\begin{align}\label{eqn:int_by_parts}
\begin{split}
    (q_h,\grad_x \cdot\tau_h)_{\W_x} &= -(\grad_x q_h,\tau_h)_{\W_x} + \left<\lavg q_h\ravg,\ljmp \tau_h\rjmp\right>_{\EIx} 
    + \left<\ljmp q_h\rjmp,\lavg \tau_h\ravg\right>_{\EIx} + \left<q_h{n_x},\tau_h\right>_{\partial\W_x}.
\end{split}
\end{align}
\end{lemma}

\subsection{Technical Estimates for Drift-Diffusion Analysis} \label{subsect:new_identities}

In this subsection, we present several technical results which are useful both in the analysis in the drift-diffusion limit to \Cref{prob:discr}, and for the future analysis of similar problems.

The first result is an error estimate of an interpolant from $V_{x,h}\to S_{x,h}^0$ if $\beta=0$ and $V_{x,h}\to V_{x,h}^0$ if $\beta=1$.  The result allows us to control a DG function's distance to $S_{x,h}^0$, in the $H_h^1$-norm, by the function's interior jumps and boundary data.  The interpolant onto the conforming finite element space, $I_h^0$, is the KP interpolant from \cite[Theorem 2.2]{Kar2003}.  The construction of $I_h^1$ is achieved in a similar manner and proved below.  

\begin{lemma}[Conforming Interpolant]\label{ass:interp}
There is an interpolant $I_h^\beta:V_{x,h}\to S_{x,h}^0$ if $\beta=0$ and $I_h^\beta:V_{x,h}\to V_{x,h}^0$ if $\beta=1$ such that for any $q_h\in V_{x,h}$ we have
\begin{subequations}\label{eqn:FEM_interp}
\begin{align}
    \|q_h-I_h^0q_h\|_{H_h^1(\W_x)}^2 &\lss \frac{1}{h_x}\|q_h\|_{L^2(\partial\W_x)}^2 + \frac{1}{h_x}\|\ljmp q_h\rjmp \|_{L^2(\EIx)}^2, \label{eqn:FEM_interp_0}\\
    \|q_h-I_h^1q_h\|_{H_h^1(\W_x)}^2 &\lss \frac{1}{h_x}\|q_h\|_{L^2(\partial\W_x)}^2. \label{eqn:FEM_interp_1}
\end{align}
\end{subequations}
\end{lemma}

\begin{proof}
The existence of $I_h^0$ and \eqref{eqn:FEM_interp_0} is given in \cite[Theorem 2.2]{Kar2003}.  The construction of $I_h^1$ and proof of \eqref{eqn:FEM_interp_1} uses ingredients and notation from Theorem 2.2 of \cite{Kar2003}.  For $K\in\mathcal{T}_{x,h}$, let $\mathcal{N}_k:=\{x_K^i,i=1\ldots,m\}$ be the Lagrange nodes of $K$ with $\{\phi_K^i,i=1,\ldots,m\}$ be the associated local Lagrange basis functions with $\phi_K^i(x_K^j)=\delta_{ij}$.  Let $\mathcal{N}:=\cup_{K\in\mathcal{T}_{x,h}} \mathcal{N}_K$.  We partition $\mathcal{N}=\mathcal{N}_{\rm{i}}\cup\mathcal{N}_{\rm{b}}$ into the interior nodes $\mathcal{N}_{\rm{i}}$ and the boundary nodes $\mathcal{N}_{\rm{b}}$ with $\mathcal{N}_{\rm i}\cap\mathcal{N}_{\rm b}=\emptyset$.  Given the representation of $q_h\in V_{x,h}$ in the basis $\phi_K^i$, define $I_h^1q_h\in V_{x,h}^0$ by
\begin{equation}\label{eqn:Ih1_def}
    I_hq_h = \sum_{K\in\mathcal{T}_{x,h}}\sum_{i=1}^m
    \begin{cases}
        \alpha_K^i\phi_K^i &\text{ if }x_K^j\in \mathcal{N}_{\rm{i}} \\
        0\phi_K^i &\text{ if } x_K^i\in\mathcal{N}_{\rm{b}}
    \end{cases}
    \quad\text{ where }\quad
    q_h = \sum_{K\in\mathcal{T}_{x,h}} \sum_{i=1}^m \alpha_K^i\phi_K^i.
\end{equation}
From \eqref{eqn:Ih1_def} it is easy to see that $q_h-I_h^1q_h$ is zero at the interior nodes.  Via scaling arguments ($\W_x\subset\R^3$), we have
\begin{subequations} \label{eqn:pfIh1:1}
\begin{align}
    \|\grad\phi_K^i\|_{L^2(K)}^2 &\lss h_x, \\
    \frac{1}{h_x} \sum_{e\in\EIx} \|\ljmp\phi_K^i\rjmp\|_{L^2(e\cap\partial K)}^2 &\lss h_x, \\
    \frac{1}{h_x}\|\phi_K^i\|_{L^2(\partial\W_x\cap\partial K)}^2 &\lss h_x.
\end{align}
\end{subequations}
Thus by \eqref{eqn:pfIh1:1} we have 
\begin{align} \label{eqn:pfIh1:2}
    \|q_h-I_h^1q_h\|_{H_h^1(\W_x)}^2 \lss h_x\sum_{K\in\mathcal{T}_{x,h}}\sum_{i:x_K^i\in\mathcal{N}_{\rm b}} \|\alpha_K^i\|^2 = h_x\sum_{\nu\in\mathcal{N}_{\rm b}}\sum_{x_K^i=\nu} \|\alpha_K^i\|^2.
\end{align}
where the last equality is just re-indexing of the double sum.  From (2.21) of \cite{Kar2003} and the quasi-uniformity of $\mathcal{T}_{x,h}$, we have for all $\nu\in\mathcal{N}_{\rm b}$:
\begin{align} \label{eqn:pfIh1:3}
    \sum_{x_K^i=\nu} \|\alpha_K^i\|^2 \lss h_x^{-1}\|q_h\|_{L^2(\partial\W\cap\partial K)}^2.
\end{align}
Therefore \eqref{eqn:FEM_interp_1} follows from \eqref{eqn:pfIh1:2} and \eqref{eqn:pfIh1:3}.  The proof is complete
\end{proof}

Additionally, we give an $H_h^1$ stability estimate for the $L^2$ projection from $V_{x,h}$ to the spaces $S_{x,h}^0$ and $V_{x,h}^0$.  This interpolant, $\mSh$ is vital to the analysis of evolution problems and this estimate is needed to extend the results of the interpolant $I_h^\beta$ to $S_{x,h}^\beta$; see \eqref{eqn:rho-and-J-stab:7}.  

\begin{lemma} \label{ass:proj-stable}
For $\beta\in\{0,1\}$, the projection $\mSh$ is stable on $V_{x,h}$ with respect to the $H_h^{1}(\W_x)$-norm, that is,
\begin{align}\label{eqn:l2-proj-stab}
    \|\mSh q_h\|_{H_h^1(\W_x)} \lss \|q_h\|_{H_h^1(\W_x)} \quad \forall q_h\in V_{x,h}.
\end{align}
\end{lemma}

We will prove \Cref{ass:proj-stable} by first showing an equivalent result. Given $\gamma > 0$, Define the DG discrete Laplacian energy on $V_{x,h}$ via the symmetric bilinear form
\begin{align}
\begin{split}
    (q_h,z_h)_E &:= (\grad q_h,\grad z_h)_{\W_x} - \left< \lavg \grad q_h\ravg,\ljmp z_h\rjmp\right>_{\EIx} -\left< \ljmp q_h\rjmp,\lavg \grad z_h\ravg\right>_{\EIx} \\
    &\quad-\left< \grad q_h,z_h n_x \right>_{\partial\W_x} - \left< q_hn_x,\grad z_h \right>_{\partial\W_x} + \frac{\gamma}{h_x}\left<\ljmp q_h\rjmp,\ljmp z_h\rjmp\right>_{\EIx} + \frac{\gamma}{h_x}\left<q_h,z_h\right>_{\partial\W_x}.
\end{split}
\end{align}
Standard DG elliptic theory shows that there exists a $\gamma_* > 0$, independent of $h$, such that $(\cdot,\cdot)_E$ is an inner product on $V_{x,h}$ for all $\gamma>\gamma_*$ \cite{Riv2008}.  We fix some $\gamma > \gamma_*$ and therefore $(\cdot,\cdot)_E$ induces a norm $\|\cdot\|_E$ on $V_{x,h}$.  
Moreover, by use of the trace inequality, \Cref{lem:trace}, we have
\begin{equation} \label{eqn:energy-h1h-bnd}
    \|q_h\|_{\h1hg} \lss \|q_h\|_E \lss \|q_h\|_{\h1hg}
\end{equation}
for all $q_h\in V_{x,h}$.  From \eqref{eqn:energy-h1h-bnd}, the following lemma immediately implies \Cref{ass:proj-stable}. 

\begin{lemma} \label{lem:proj-stable-energy}
The $L^2$-projection $\mSh$ is stable on $V_{x,h}$ with respect to $\|\cdot\|_E$, that is, there is a constant $C>0$, independent of $h_x$, such that
\begin{align}\label{eqn:l2-proj-stab-energy:0}
    \|\mSh q_h\|_{E} \leq C \|q_h\|_E \quad \forall q_h\in V_{x,h}.
\end{align}
\end{lemma}

\begin{proof}
We will focus on the case $\beta = 0$. The proof is written such that the $\beta=1$ case is shown by substituting $S_{x,h}^0$ with $V_{x,h}^0$.  Let $\{\psi_i\}_{i=1}^M\subset V_{x,h}$ be an orthonormal eigenbasis in $L^2(\W_x)$ with associated eigenvalues $\lambda_i>0$, in increasing order, for the following eigenvalue problem: find $\psi\in V_{x,h}$ and $\lambda\in\R$ such that  
\begin{equation} \label{eqn:l2-proj-stab-energy:1}
    (\psi,q_h)_E = \lambda(\psi,q_h)_{\W_x}  \quad\forall q_h\in V_{x,h}.
\end{equation}

Similarly, let $\{\phi_j\}_{j=1}^N\subset S_{x,h}^0$ be an orthonormal eigenbasis in $L^2(\W_x)$ with associated eigenvalues $\mu_j>0$, in increasing order, for the following eigenvalue problem on $S_{x,h}^0$: Find $\phi\in S_{x,h}^0$ and $\mu\in\R$ such that
\begin{equation} \label{eqn:l2-proj-stab-energy:2}
   (\phi,w_h)_E = \mu(\phi,w_h)_{\W_x}  \quad\forall w_h\in S_{x,h}^0.
\end{equation}

We also define the solution operators $T_h:V_{x,h}\to V_{x,h}$, and $T_h^0:S_{x,h}^0\to S_{x,h}^0$ by
\begin{align}
    (T_hq_h,z_h)_{E} &= (q_h,z_h)_{\W_x} \quad\forall z_h\in V_{x,h}, \label{eqn:l2-proj-stab-energy:3a}\\
    (T_h^0w_h,s_h)_E &= (w_h,s_h)_{\W_x} \quad\forall s_h\in S_{x,h}^0. \label{eqn:l2-proj-stab-energy:3b}
\end{align}

We note that while written using the inner product $(\cdot,\cdot)_E$, \eqref{eqn:l2-proj-stab-energy:3b} is the standard continuous Galerkin finite element method for the Poisson problem and \eqref{eqn:l2-proj-stab-energy:2} is its respective eigenvalue problem.  Note $\psi_i$ and $\phi_j$ are eigenvectors of $T_h$ and $T_h^0$ with associated eigenvalues $\lambda_i^{-1}$ and $\mu_j^{-1}$ respectively.  We also recall the inverse Laplacian $S:L^2(\W_x)\to H^2(\W_x)\cap H_0^1(\W_x)$ which is given in \Cref{ass:inv_laplace_bnd}.  Standard continuous and discontinuous Galerkin theory \cite{Bre2008,Riv2008} and \Cref{ass:inv_laplace_bnd} yield the following estimates:
\begin{align}
    \|T_hq_h-Sq_h\|_{E} &\lss h_x \|Sq_h\|_{H^2(\W)} \lss h_x\|q_h\|_{L^2(\W_x)} \quad\forall q_h\in V_{x,h}, \label{eqn:l2-proj-stab-energy:5a}\\
    \|T_h^0w_h-Sw_h\|_{E} &\lss h_x\|Sw_h\|_{H^2(\W)} \lss h_x\|w_h\|_{L^2(\W_x)} \quad\forall w_h\in S_{x,h}^0. \label{eqn:l2-proj-stab-energy:5b}
\end{align}
Therefore by \eqref{eqn:l2-proj-stab-energy:5a}-\eqref{eqn:l2-proj-stab-energy:5b} we have
\begin{equation}
    \|T_hw_h-T_h^0w_h\|_{E} \lss h_x\|w_h\|_{L^2(\W_x)} \quad\forall w_h\in S_{x,h}^0. \label{eqn:l2-proj-stab-energy:6}
\end{equation}

Given $q_h\in V_{x,h}$, let $\alpha\in\R^M$ be the coefficients of $q_h$ w.r.t the basis $\{\psi_i\}$ given by $
\alpha_i = (q_h,\psi_i)_{\W_x}$.  Similarly, given $w_h\in S_{x,h}^0$, let $\xi\in\R^N$ be the coefficients of $w_h$ w.r.t the basis $\{\phi_j\}$ given by $
\xi_j = (w_h,\phi_j)_{\W_x}$.  Due to the eigenbasis decomposition we have
\begin{subequations}
\begin{align}
    q_h &= \sum_{i=1}^M \alpha_i\psi_i,& \qquad w_h &= \sum_{j=1}^N \xi_j\phi_j, \label{eqn:l2-proj-stab-energy:7a}\\
    \|q_h\|_{L^2(\W_x)}^2 &= \sum_{i=1}^M \alpha_i^2 = |\alpha|^2,&\|w_h\|_{L^2(\W_x)}^2 &= \sum_{j=1}^N \xi_j^2 = |\xi|^2, \label{eqn:l2-proj-stab-energy:7b}\\
    \|q_h\|_{E}^2 &= \sum_{i=1}^M \lambda_i\alpha_i^2=:|\alpha|_E^2, 
    & \| w_h\|_{E}^2 &= \sum_{i=1}^M \mu_j\xi_j^2=:|\xi|_E^2. \label{eqn:l2-proj-stab-energy:7c}
\end{align}
\end{subequations}

Through the decompositions in \eqref{eqn:l2-proj-stab-energy:7a}, we also have
\begin{align}\label{eqn:l2-proj-stab-energy:8}
    T_hq_h &= \sum_{i=1}^M \frac{\alpha_i}{\lambda_i}\psi_i,& T_h^0w_h &= \sum_{j=1}^N \frac{\xi_j}{\mu_j}\phi_j.
\end{align}
Additionally, we define the operator $A_h^0:S_{x,h}^0\to S_{x,h}^0$ by 
\begin{align}\label{eqn:l2-proj-stab-energy:9}
   A_h^0w_h &= \sum_{j=1}^N \xi_j\mu_j^{1/2}\phi_j.
\end{align}
The fact that $\{\phi_j\}$ is an orthonormal set in $L^2(\W_x)$ and \eqref{eqn:l2-proj-stab-energy:7c} yield
\begin{align}\label{eqn:l2-proj-stab-energy:10}
   \|A_h^0w_h\|_{L^2(\W_x)} = \| w_h\|_{E}.
\end{align}
Also, $\phi_j$ is an eigenvector of $A_h^0$ with associated eigenvalue $\mu_j^{1/2}$.

Note that \eqref{eqn:l2-proj-stab-energy:0} is equivalent to uniformly bounding
\begin{equation}\label{eqn:l2-proj-stab-energy:11}
    \sup_{q_h\in V_{x,h}\setminus\{0\}} \frac{\|\mSh q_h\|_E}{\|q_h\|_E}= \sup_{\substack{q_h\in V_{x,h}\setminus\{0\}\\w_h\in S_{x,h}^0\setminus\{0\}}} \frac{(\mSh q_h,w_h)_{E}}{\|q_h\|_{E}\| w_h\|_{E}}
\end{equation}
in $h_x$; thus we seek to bound $(\mSh q_h,w_h)_{E}$.  Using \eqref{eqn:l2-proj-stab-energy:2}, \eqref{eqn:l2-proj-stab-energy:7a}, \eqref{eqn:l2-proj-stab-energy:7c}, and \eqref{eqn:l2-proj-stab-energy:10} we have
\begin{align} \label{eqn:l2-proj-stab-energy:12}
\begin{split}
    (\mSh q_h,w_h)_{E} &= \sum_{ij}\alpha_i(\mSh \psi_i,\phi_j)_{E}\xi_j \\
    &= \sum_{ij}\alpha_i (\mSh \psi_i,\phi_j)_{\W_x}\mu_j\xi_j = \sum_{ij}\alpha_i (\psi_i,\phi_j)_{\W_x}\mu_j\xi_j \\
    &= \sum_{ij} \alpha_i\lambda_i^{1/2}(\lambda_i^{-1/2}\psi_i,\mu_j^{1/2}\phi_j)_{\W_x}\xi_j\mu_j^{1/2}  \\
    &= \sum_{ij} \alpha_i\lambda_i^{1/2}(\lambda_i^{-1/2}\psi_i,A_h^0\phi_j)_{\W_x}\xi_j\mu_j^{1/2}  \\
    &\leq |\alpha|_E\|\overline{C}\|_2|\xi|_E = \|\overline{C}\|_2\|q_h\|_{E}\| w_h\|_{E}.
\end{split}
\end{align}
Here $\overline{C}\in\R^{M\times N}$ with $\overline{C}_{ij} = (\lambda_i^{-1/2}\psi_i,A_h^0\phi_j)_{\W_x}$, and
\begin{equation}\label{eqn:l2-proj-stab-energy:13}
    \|\overline{C}\|_2 = \max_{\xi\in\R^{N}\setminus\{0\}}\frac{\xi^T\overline{C}^T\overline{C}\xi}{|\xi|^2}.
\end{equation}
To bound $\|\overline{C}\|_2$, we use the decomposition of \eqref{eqn:l2-proj-stab-energy:7a} along with \eqref{eqn:l2-proj-stab-energy:8} and \eqref{eqn:l2-proj-stab-energy:3a} to compute the following:
\begin{align} \label{eqn:l2-proj-stab-energy:14}
\begin{split} 
    \xi^T\bar{C}^T\bar{C}\xi &= \sum_{jl=1}^N 
 \sum_{i=1}^M\xi_j(A_h^0\phi_j,\lambda_i^{-1/2}\psi_i)_{\W_x} (\lambda_i^{-1/2}\psi_i,A_h^0\phi_l)_{\W_x}\xi_l \\
    &= \sum_{jl=1}^N \xi_j \left(A_h^0\phi_j,\sum_{i=1}^M\frac{(A_h^0\phi_l,\psi_i)_{\W_x}}{\lambda_i}\psi_i\right)_{\W_x} \xi_l \\
    &= \sum_{jl=1}^N \xi_j (A_h^0\phi_j,T_hA_h^0\phi_l)_{\W_x} \xi_l = \sum_{jl=1}^N \xi_j (T_hA_h^0\phi_j,T_hA_h^0\phi_l)_E \xi_l \\
    &= \left(T_hA_h^0 \sum_{j=1}^N \xi_j\phi_j, T_hA_h^0\sum_{l=1}^N \xi_l\phi_l\right)_E \\
    &= (T_hA_h^0w_h,T_hA_h^0w_h)_E. 
\end{split}
\end{align}
Therefore \eqref{eqn:l2-proj-stab-energy:7b}, \eqref{eqn:l2-proj-stab-energy:13}, and \eqref{eqn:l2-proj-stab-energy:14} yield
\begin{equation}\label{eqn:l2-proj-stab-energy:15}
 \|\overline{C}\|_2 = \max_{w_h\in S_{x,h}^0\setminus\{0\}}\frac{\|T_hA_h^0w_h\|_E^2}{\|w_h\|_{L^2(\W_x)}^2}.
\end{equation}
Therefore the desired estimate \eqref{eqn:l2-proj-stab-energy:0} holds provided we can show
\begin{equation}\label{eqn:l2-proj-stab-energy:16}
    \|T_hA_h^0w_h\|_E \lss \|w_h\|_{L^2(\W_x)}
\end{equation}
for all $w_h\in S_{x,h}^0$.  

Let $w_h\in S_{x,h}^0$ with decomposition given in \eqref{eqn:l2-proj-stab-energy:7a}.  To show \eqref{eqn:l2-proj-stab-energy:16}, we will add and subtract $T_h^0A_h^0w_h$ where $T_h^0$ is given in \eqref{eqn:l2-proj-stab-energy:3b} and apply the triangle inequality to obtain
\begin{equation} \label{eqn:l2-proj-stab-energy:17}
    \|T_hA_h^0w_h\|_E \leq \|(T_h-T_h^0)A_h^0w_h\|_E + \|T_h^0A_h^0w_h\|_E.
\end{equation}
To bound $\|(T_h-T_h^0)A_h^0w_h\|_E$, we use \eqref{eqn:l2-proj-stab-energy:6}, \eqref{eqn:l2-proj-stab-energy:10}, \eqref{eqn:energy-h1h-bnd}, and the inverse inequality \eqref{eqn:inverse} which gives
\begin{equation} \label{eqn:l2-proj-stab-energy:18}
\begin{split}
    \|(T_h-T_h^0)A_h^0w_h\|_E \lss h_x\|A_h^0w_h\|_{L^2(\W_x)} = h_x\| w_h\|_{E} \lss h_x\|w_h\|_{H_h^1(\W)} \lss \|w_h\|_{L^2(\W_x)}.
\end{split}
\end{equation}
Direct computation of $T_h^0A_h^0w_h$ gives us
\begin{equation} \label{eqn:l2-proj-stab-energy:19}
    T_h^0A_h^0w_h = \sum_{j=1}^N \xi_jT_h^0A_h^0\phi_j = \sum_{j=1}^N \xi_j \mu_j^{1/2}T_h^0\phi_j = \sum_{j=1}^N \xi_j \mu_j^{-1/2}\phi_j
\end{equation}
Additionally, since $\phi_j$ is an eigenvector to the problem given in \eqref{eqn:l2-proj-stab-energy:2}, then 
\begin{equation} \label{eqn:l2-proj-stab-energy:20}
    (\phi_j,\phi_l)_{E} = \delta_{jl}\mu_j, \quad\text{ and }\quad \|\phi_j\|_E^2 = \mu_j,
\end{equation}
where $\delta_{jl}$ is the Kronecker delta.  Then using \eqref{eqn:l2-proj-stab-energy:19}, \eqref{eqn:l2-proj-stab-energy:20}, and \eqref{eqn:l2-proj-stab-energy:7b} we obtain
\begin{equation} \label{eqn:l2-proj-stab-energy:21}
    \|T_h^0A_h^0w_h\|_{E}^2 = \sum_{j=1}^N \frac{\xi_j^2}{\mu_j}\|\phi_j\|_{E}^2 = \sum_{j=1}^N \xi_j^2 = \|w_h\|_{L^2(\W_x)}^2.
\end{equation}
Therefore \eqref{eqn:l2-proj-stab-energy:17}, \eqref{eqn:l2-proj-stab-energy:18}, and \eqref{eqn:l2-proj-stab-energy:21} imply \eqref{eqn:l2-proj-stab-energy:16}.  The proof is complete.
\end{proof}

\subsection{Initial Estimates}
We first focus on estimates for $g_h^\e$ that will lead to estimates for $\rho_h^\e$ and $J_h^\e$. 

\begin{lemma}\label{lem:stability_identity_and_estimates}
The bilinear forms defined in Problem \ref{prob:discr} satisfy the following bounds
\begin{align}\label{eqn:collision_L2}
    -\mathcal{Q}(g_h^\e,g_h^\e) &\geq \omega_{\min}\left\|g_h^\e - \mhalfh \rho_h^\e\right\|_{L^2(\W)}^2,  \\ 
\label{eqn:ps-garding}
\begin{split}
     \mathcal{C}(g_h^\e,g_h^\e) &\leq 
    \frac{1}{2}\left(2C_1 + \frac{\omega_{\min}}{\varepsilon} \right) 
    \left \|g_h^\e -  \mhalfh \rho_h^\e \right\|_{L^2(\W)}^2 + \frac{\e}{2\omega_{\min}}{C_3^2}\|g_h^\e\|_{L^2(\W)}^2,
\end{split}  \\ 
\label{eqn:Apos}
    \mathcal{A}(g_h^\e,g_h^\e) 
        &=\left< \frac{|v_h\cdot{n_x}|}{2}\ljmp g_h^\e \rjmp ,\ljmp g_h^\e\rjmp \right>_{\partial \W_x\times\W_v}
        + {\e^\b}\left< \frac{|v_h\cdot{n_x}|}{2} \ljmp g_h^\e \rjmp,\ljmp g_h^\e\rjmp \right>_{\EIx\times\W_v} ,
    \\ 
    \label{eqn:Bpos}
    \begin{split}
       \mB(g_h^\e,g_h^\e) + \mD(g_h^\e,g_h^\e) &\geq \left<\frac{|E\cdot{n_v}|}{2}\ljmp g_h^\e\rjmp ,\ljmp g_h^\e\rjmp \right>_{\W_x\times\EIv} - 
       \frac{C_2}{2h_v}\left \|g_h^\e-\mhalfh\rho_h^\e \right\|_{L^2(\W)}^2    
    \end{split}
\end{align}
where
\begin{subequations}\label{eqn:constant_def}
\begin{align}
    C_1 &:= \frac{\|E\cdot v_h\|_{L_T^{\infty}(L^{\infty}(\Omega))}}{2\theta},& C_2&:=C_T\|E\|_{L_T^\infty(L^\infty(\W_x))}, \\
    C_3& := \frac{3}{2\theta}\|E\|_{L_T^\infty(L^\infty(\W_x))}.
\end{align}
\end{subequations}
are constants independent of $\e$, $h_x$, and $h_v$; and $C_T>0$ is a constant from the trace inequality \eqref{eqn:trace},
\end{lemma}
\begin{proof}
For \eqref{eqn:collision_L2}, the definition of $\rho_h^\e$, along with Assumption \ref{ass-discrete-root-Max}.\ref{ass-mass}, implies that
\begin{align}\label{eqn:energy-est-1}
\begin{split}
    (\mhalfh\rho_h^\e,\mhalfh\rho_h^\e)_{\W_v} = |\rho_h^\e|^2 = (\mhalfh\rho_h^\e,g_h^\e)_{\W_v}.
\end{split}
\end{align}
Therefore by expansion we see
\begin{align}\label{eqn:energy-est-2}
\begin{split}
    \|g_h^\e - \mhalfh\rho_h^\e\|_{L^2(\W_v)}^2
    = (g_h^\e,g_h^\e-\mhalfh\rho_h^\e)_{\W_v}.
\end{split}
\end{align}
Using \eqref{eqn:energy-est-2} gives
\begin{align*}
    -\mQ(g_h^\e,g_h^\e) &= \left(\omega,  (g_h^\e,g_h^\e-\mhalfh\rho_h^\e)_{\W_v}\right)_{\W_x} = \left(\omega, \|g_h^\e - \mhalfh\rho_h\|_{L^2(\W_v)}^2 \right)_{\W_x} \\
    &\geq \omega_{\min}\left\|g_h^\e - \mhalfh \rho_h^\e\right\|_{L^2(\W)}^2,
\end{align*}
which is \eqref{eqn:collision_L2}.

For \eqref{eqn:ps-garding}, by Assumption \ref{ass-discrete-root-Max}.\ref{ass-momen},
\begin{align}\label{eqn:ps-garding1}
\begin{split}
    \mC(\mhalfh \rho_h^\e, \mhalfh \rho_h^\e) &=\frac{1}{2\theta} (E\rho_h^\e,\rho_h^\e)_{\W_x}\cdot(v_h\mhalfh,\mhalfh)_{\W_v} \\
    &=-(E\rho_h^\e,\rho_h^\e)_{\W_v}\cdot(\grad_v \mhalfh,\mhalfh)_{\W_v} = 0.
\end{split}
\end{align}
Thus by \eqref{eqn:ps-garding1},
\begin{align}\label{eqn:c-split0}
\begin{split}
    \mathcal{C}(g_h^\e -  \mhalfh \rho_h^\e,g_h^\e-  \mhalfh \rho_h^\e) &= \mC(g_h^\e,g_h^\e) - 2\mathcal{C}(g_h^\e,\mhalfh \rho_h^\e) + \mathcal{C}(\mhalfh\rho_h^\e,\mhalfh\rho_h^\e) 
    \\
    &= \mC(g_h^\e,g_h^\e) - 2\mathcal{C}(g_h^\e,\mhalfh \rho_h^\e) + 2\mathcal{C}(\mhalfh\rho_h^\e,\mhalfh\rho_h^\e) 
    \\
    &= \mC(g_h^\e,g_h^\e) - 2\mathcal{C}(g_h^\e-\mhalfh\rho_h^\e,\mhalfh \rho_h^\e).
\end{split}
\end{align}
Rewriting \eqref{eqn:c-split0} yields
\begin{align}\label{eqn:c-split}
\begin{split}
    \mathcal{C}(g_h^\e,g_h^\e) &=  \mathcal{C}(g_h^\e -  \mhalfh \rho_h^\e,g_h^\e-  \mhalfh \rho_h^\e) + 2\mathcal{C}(g_h^\e-\mhalfh\rho_h^\e,\mhalfh \rho_h^\e)
    \\
    &:= I_1 + I_2.
\end{split}
\end{align}
Applying H\"older's inequality to \eqref{eqn:c-split} gives
\begin{equation}\label{eqn:c-split3}
    |I_1| \leq \frac{\|E\cdot v_h\|_{L^\infty(\W)}}{2\theta} \|g_h^\e -  \mhalfh \rho_h^\e\|_{L^2(\W)}^2
\end{equation}
Using H\"older's inequality, \eqref{eqn:v_h}, and \Cref{ass-discrete-root-Max}.\ref{ass-energy}, we bound $I_2$ as
\begin{align}\label{eqn:c-split4}
\begin{split}
    |I_2| &\leq \tfrac{1}{\theta}\|E\cdot v_h\mhalfh\rho_h^\e\|_{L^2(\W)}\|g_h^\e-\mhalfh\rho_h^\e\|_{L^2(\W)} \\
    &\leq 2\|E\|_{L^\infty(\W_x)}\|\rho_h^\e\|_{L^2(\W_x)}\| \tfrac{1}{2\theta}v_h\mhalfh\|_{L^2(\W_v)}\|g_h^\e-\mhalfh\rho_h^\e\|_{L^2(\W)} \\
    &\leq 2\|E\|_{L^\infty(\W_x)}\|\rho_h^\e\|_{L^2(\W_x)}\| \grad_v\mhalfh\|_{L^2(\W_v)}\|g_h^\e-\mhalfh\rho_h^\e\|_{L^2(\W)} \\
    &\leq \frac{3}{2\theta}\|E\|_{L^\infty(\W_x)}\|\rho_h^\e\|_{L^2(\W_x)}\|g_h^\e-\mhalfh\rho_h^\e\|_{L^2(\W)}.
\end{split}
\end{align}
Meanwhile, integrating \eqref{eqn:energy-est-1} over $\Omega_x$ and applying the Cauchy-Schwarz inequality gives
\begin{equation} \label{eqn:rho-cont-bound}
    \|\rho_h^\e\|_{L^2(\W_x)} = \|\mhalfh \rho_h^\e\|_{L^2(\W)}  \leq \|g_h^\e\|_{L^2(\W)}.
\end{equation}
Substituting \eqref{eqn:c-split3}-\eqref{eqn:rho-cont-bound} into \eqref{eqn:c-split} and invoking Young's inequality gives
\begin{align} \label{eqn:c-split2}
\begin{split}
    \mathcal{C}(g_h^\e,g_h^\e) &\leq C_1 \|g_h^\e -  \mhalfh \rho_h^\e\|_{L^2(\W)}^2 + C_3\|g_h^\e\|_{L^2(\W)} \|g_h^\e -  \mhalfh \rho_h^\e\|_{L^2(\W)}
    \\
    &\leq 
    \left(C_1 + \frac{\omega_{\min}}{2\varepsilon} \right) \|g_h^\e -  \mhalfh \rho_h^\e\|_{L^2(\W)}^2 + \frac{\e}{2\omega_{\min}}C_3^2\|g_h^\e\|_{L^2(\W)}^2.
\end{split}
\end{align}
which yields \eqref{eqn:ps-garding}.

For \eqref{eqn:Apos}, it follows from \Cref{lem:parts} (setting $q_h = g_h^\e$ and $\tau_h = v_h g_h^\e$) that
\begin{equation}
\begin{split}
\label{eqn:apply_parts}
    (v_hg_h^\e,\grad_x g_h^\e)_\W 
        &= \left< v_h\lavg g_h^\e\ravg,\ljmp g_h^\e\rjmp \right>_{\EIx\times\W_v} + \frac{1}{2}\left<v_hg_h^\e,{n_x}g_h^\e\right>_{\partial\W_x\times\W_v}. \\
\end{split}
\end{equation}
Direct substitution of this formula into the definition of $\mathcal{A}$ gives
\begin{align}\label{eqn:apos-1}
\begin{split}
     \mathcal{A}(g_h^\e,g_h^\e) 
     &= \e^\b\left<\frac{|v_h\cdot{n_x}|}{2}\ljmp g_h^\e\rjmp,\ljmp g_h^\e\rjmp\right>_{\EIx\times\W_v} 
     + \left< \frac{|v_h\cdot{n_x}|}{2} g_h^\e, g_h^\e\right>_{\partial \W_x\times\W_v},
\end{split}
\end{align}
which is \eqref{eqn:Apos}.

For \eqref{eqn:Bpos}, a formula similar to \eqref{eqn:apply_parts} gives
\begin{align}\label{eqn:bpos-1}
    \mB(g_h^\e,g_h^\e) = \left<\frac{|E\cdot{n_v}|}{2}\ljmp g_h^\e\rjmp,\ljmp g_h^\e\rjmp\right>_{\W_x\times\EIv} - \frac{1}{2}\left<Eg_h^\e,{n_v}g_h^\e\right>_{\W_x\times\partial\W_v}. 
\end{align}
Meanwhile, invoking the divergence theorem and Assumption \ref{ass-discrete-root-Max}.\ref{ass-momen} yields
\begin{equation}\label{eqn:bpos-1.5}
    \left< E\cdot n_v,(\mhalfh)^2 \right>_{\partial\W_v} = 2E\cdot(\grad_v \mhalfh,\mhalfh)_{\W_v} = 0.
\end{equation}
Therefore applying the polarization identity $ab=\frac{1}{2}(a^2+b^2-(a-b)^2)$ with $a:=g_h^\e$ and $b:=\mhalfh\rho_h$ and \eqref{eqn:bpos-1.5} gives the following bound on $\mD$:
\begin{align}\label{eqn:bpos-2}
\begin{split}
    \mD(g_h^\e,g_h^\e) &= \left<E\cdot{n_v},\mhalfh\rho_h^\e g_h^\e\right>_{\W_x\times\partial\W_v} \\
    &= \frac{1}{2}\left<E\cdot{n_v},(g_h^\e)^2+(\mhalfh\rho_h^\e)^2-(g_h^\e-\mhalfh\rho_h^\e)^2\right>_{\W_x\times\partial\W_v} \\
    &= \frac{1}{2}\left<E\cdot{n_v},(g_h^\e)^2\right>_{\W_x\times\partial\W_v} - \frac{1}{2}\left<E\cdot{n_v},(g_h^\e-\mhalfh\rho_h^\e)^2\right>_{\W_x\times\partial\W_v} \\
    &\geq \frac{1}{2}\left<Eg_h^\e,n_v g_h^\e\right>_{\W_x\times\partial\W_v} - \frac{\|E\|_{L_T^\infty(L^\infty(\W_x))}}{2}\|g_h^\e-\mhalfh\rho_h^\e\|_{L^2(\W_x\times\partial\W_v)}^2.
\end{split}
\end{align}
Applying the discrete trace estimate \eqref{eqn:trace_vx} to \eqref{eqn:bpos-2} gives
\begin{align}\label{eqn:bpos-3}
\begin{split}
    \mD(g_h^\e,g_h^\e) &\geq \frac{1}{2}\left<Eg_h^\e,{n_v}g_h^\e\right>_{\W_x\times\partial\W_v} - \frac{C_T\|E\|_{L_T^\infty(L^\infty(\W_x))}}{2h_v}\|g_h^\e-\mhalfh\rho_h^\e \|_{L^2(\W)}^2, \\ 
\end{split}
\end{align}
where the constant $C_T>0$ is independent of $\e$, $h_x$, and $h_v$.
Adding \eqref{eqn:bpos-1} and \eqref{eqn:bpos-3} yields
\begin{align}\label{eqn:bpos-4}
\begin{split}
    \mB(g_h^\e,g_h^\e) + \mD(g_h^\e,g_h^\e) &\geq \left<\frac{|E\cdot{n_v}|}{2}\ljmp g_h^\e\rjmp,\ljmp g_h^\e\rjmp\right>_{\W_x\times\EIv} \\
 &\quad- \frac{C_T\|E\|_{L_T^\infty(L^\infty(\W_x))}}{2h_v}\|g_h^\e-\mhalfh\rho_h^\e\|_{L^2(\W)}^2
\end{split}
\end{align}
which is \eqref{eqn:Bpos}.  The proof is complete.
\end{proof}


Using Lemma \ref{lem:stability_identity_and_estimates}, we derive a space-time energy estimate for $g_h^\e$ and $\mhalfh\rho_h^\e-g_h^\e$.  Define 
\begin{equation}
\label{eqn:eps_h}
\e_{h_v} := \frac{\omega_{\min}h_v}{4C_1h_v + 2C_2}.
\end{equation}

\begin{lemma}\label{lem:g-l2-stab}
Given $h_v>0$, $h_x>0$, and $\e \leq   \e_{h_v}$,
\begin{align}\label{eqn:g-l2-stab}
\begin{split}
    \|g_h^\e(T)\|_{L^2(\W)}^2 &+ \frac{\omega_{\min}}{2\e^2}\|g_h^\e - \mhalfh \rho_h^\e\|_{L_T^2(L^2(\W))}^2 \\
    & +\int_0^T \left[\e^{\b-1} \left<{|v_h\cdot{n_x}|}\ljmp g_h^\e\rjmp,\ljmp g_h^\e\rjmp\right>_{\EIx\times\W_v} 
     + \tfrac{1}{\e}\left< {|v_h\cdot{n_x}|} g_h^\e, g_h^\e\right>_{\partial \W_x\times\W_v}\right] \dx[t] \\
     &\leq \|g_{0,h}\|_{L^2(\W)}^2 \exp\left( \frac{C_3^2}{\omega_{\min}}T\right).
\end{split}
\end{align}
\end{lemma}
\begin{proof}
Letting $z_h = \frac{2}{\e}g_h^\e$ in \eqref{eqn:alt_forms_eq} yields the energy equation
\begin{align}\label{eqn:g-l2-stab-0}
\begin{split}
    \frac{\dx[]}{\dx[t]}\|g_h^\e\|_{L^2(\W)}^2 + \frac{2}{\e}\mA(g_h^\e,g_h^\e) + \frac{2}{\e}(\mB(g_h^\e,g_h^\e) + \mD(g_h^\e,g_h^\e)) - \frac{2}{\e^2}\mQ(g_h^\e,g_h^\e) = \frac{2}{\e}\mC(g_h^\e,g_h^\e)
\end{split}
\end{align}
Substituting the estimates in Lemma \ref{lem:stability_identity_and_estimates} yields
\begin{align} \label{eqn:g-l2-stab-1}
\begin{split}
    \frac{\dx[]}{\dx[t]}\|g_h^\e\|_{L^2(\W)}^2 &+ \e^{\b-1} \left<{|v_h\cdot{n_x}|}\ljmp g_h^\e\rjmp,\ljmp g_h^\e\rjmp\right>_{\EIx\times\W_v} 
     + \frac{1}{\e}\left< {|v_h\cdot{n_x}|} g_h^\e, g_h^\e\right>_{\partial \W_x\times\W_v} \\
     &+\frac{1}{\e}\left<|E\cdot{n_v}|\ljmp g_h^\e\rjmp ,\ljmp g_h^\e\rjmp \right>_{\W_x\times\EIv} - 
       \frac{C_2}{\e h_v} \|\mhalfh\rho_h^\e-g_h^\e \|_{L^2(\W)}^2 \\
     &+ \frac{2\omega_{\min}}{\e^2}\|\mhalfh\rho_h^\e-g_h^\e\|_{L^2(\W)}^2\\
    &\leq \frac{1}{\omega_{\min}}C_3^2\|g_h^\e\|_{L^2(\W)}^2 +  \left( \frac{2C_1}{\e}+\frac{\omega_{\min}}{\e^2}\right) \|g_h^\e -  \mhalfh \rho_h^\e \|_{L^2(\W)}^2 
\end{split}
\end{align}
Dropping the positive contribution $\frac{1}{\e}\left< {|E\cdot{n_v}|} \ljmp g_h^\e\rjmp,\ljmp g_h^\e\rjmp\right>_{\W_x\times\EIv}$ and collecting like terms gives 
\begin{align} \label{eqn:g-l2-stab-2}
\begin{split}
    \frac{\dx[]}{\dx[t]}\|g_h^\e\|_{L^2(\W)}^2 + & \e^{\b-1} \left<{|v_h\cdot{n_x}|}\ljmp g_h^\e\rjmp,\ljmp g_h^\e\rjmp\right>_{\EIx\times\W_v} 
     + \frac{1}{\e}\left< {|v_h\cdot{n_x}|} g_h^\e, g_h^\e\right>_{\partial \W_x\times\W_v}\\
     &+\frac{1}{\e^2}\left( \omega_{\min} - 2\e C_1 - \e\frac{C_2}{h_v}\right) \|\mhalfh\rho_h^\e-g_h^\e\|_{L^2(\W)}^2
    \leq \frac{C_3^2}{\omega_{\min}}\|g_h^\e\|_{L^2(\W)}^2.
\end{split}
\end{align}
Since $\e_{h_v} = \frac{\omega_{\min}h_v}{4C_1h_v + 2C_2}$ (see \eqref{eqn:eps_h}), it follows that for any $\e\leq \e_{h_v}$, 
\begin{equation}
    \omega_{\min} - 2\e C_1 - \e\frac{C_2}{h_v} \geq \frac{\omega_{\min}}{2}.
\end{equation}
Therefore
\begin{align} \label{eqn:g-l2-stab-3}
\begin{split}
    \frac{\dx[]}{\dx[t]}\|g_h^\e\|_{L^2(\W)}^2 &+ \e^{\b-1} \left<{|v_h\cdot{n_x}|}\ljmp g_h^\e\rjmp,\ljmp g_h^\e\rjmp\right>_{\EIx\times\W_v} 
     + \frac{1}{\e}\left< {|v_h\cdot{n_x}|} g_h^\e, g_h^\e\right>_{\partial \W_x\times\W_v} \\
     &+ \frac{\omega_{\min}}{2\e^2}\|\mhalfh\rho_h^\e-g_h^\e\|_{L^2(\W)}^2 \leq \frac{C_3^2}{\omega_{\min}}\|g_h^\e\|_{L^2(\W)}^2.
\end{split}
\end{align}
Applying Gr\"onwall's inequality to \eqref{eqn:g-l2-stab-3} yields \eqref{eqn:g-l2-stab}.  The proof is complete.
\end{proof}

With Lemma \ref{lem:g-l2-stab} in hand, we can obtain stability estimates for $\rho_h^\e$ and $J_h^\e$ as well as some projection estimates which will be useful in the next section.  We first list a technical lemma whose proof is provided in the appendix.

\begin{lemma}\label{lem:tech_lb}
Let
\begin{align}
\gamma_I(x) := \left(\frac{|v_h\cdot{n_x}(x)|}{2}\mhalfh,\mhalfh\right)_{\W_v} \quad \text{and} \quad 
\gamma_B(x) := \left(v_h\mhalfh,\mhalfh\right)_{\{v:v_h(v)\cdot{n_x}(x) > 0\}}
\end{align}
for $x\in\EIx$ and $x\in\partial\W$ respectively.  Then there exists $\gamma_*=\gamma_*(h_v)>0$ such that $\gamma_I>\gamma_*$ on $\EIx$ and $\gamma_B\cdot n>\gamma_*$ on $\partial\W_x$ for all $h_x$ and $\e$. 
\end{lemma}

\begin{lemma}\label{lem:rho-and-J-stab}
Recall the definitions of $\rho_h^\e$ and $J_h^\e$ from \eqref{eqn:rho_h} and \eqref{eqn:J_h}, respectively.  For all $h_x>0$ and every $\e\leq\e_{h_v}$, where $\e_{h_v}$ is defined in \eqref{eqn:eps_h}, the following space-time stability estimates hold:
\begin{align} \label{eqn:rho-and-J-stab-g}
    \|\rho_h^\e\|_{L^\infty_T(L^2{(\W_x))}} \leq \|g_h^\e\|_{L_T^\infty(L^2(\W))} &\lss \|g_{0,h}\|_{L^2(\W)}, \\
    \frac{1}{\e}\|\mhalfh\rho_h^\e-g_h^\e\|_{L_T^2(L^2(\W))}& \label{eqn:rho-and-J-stab-mp-g} \lss \|g_{0,h}\|_{L^2(\W)},\\
    \|J_h^\e\|_{L_T^2(L^2(\W_x))} &\lss  \label{eqn:rho-and-J-stab-J} \|g_{0,h}\|_{L^2(\W)}.
\end{align}
Moreover,
\begin{align}
\frac{\sqrt{\gamma_*}}{\sqrt{\e^{1-\beta}}}\|\ljmp\rho_h^\e\rjmp\|_{L_T^2(L^2(\EIx))} +
 \frac{\sqrt{\gamma_*}}{\sqrt{\e}}\|\rho_h^\e\|_{L_T^2(L^2(\partial\W_x))} &\lss \left(\sqrt{\frac{\e}{h_x}}+1\right)\|g_{0,h}\|_{L^2(\W)} \label{eqn:rho-and-J-stab-jmps}  \\
 \sqrt{\frac{h_x}{\e}}\|\rho_h^\e-\mSh\rho_h^\e\|_{L_T^2(L^2(\W_x))} + \sqrt{\frac{h_x}{\e}}\|\rho_h^\e-\mSh\rho_h^\e\|_{L_T^2(H_h^1(\W_x))} &\lss \left(\sqrt{\frac{\e}{h_x}}+1\right)\|g_{0,h}\|_{L^2(\W)} \label{eqn:rho-and-J-stab-proj}
\end{align}
where $\gamma_*$ is defined in Lemma \ref{lem:tech_lb}.
\end{lemma}

\begin{proof}
Estimates \eqref{eqn:rho-and-J-stab-g} and \eqref{eqn:rho-and-J-stab-mp-g} follow from \eqref{eqn:g-l2-stab}.  For \eqref{eqn:rho-and-J-stab-J}, the definition of $J_h^\e$ and Assumption \ref{ass-discrete-root-Max}.\ref{ass-momen} give
\begin{align}\label{eqn:rho-and-J-stab:2}
\begin{split}
    J_h^\e = \frac{1}{\e}(v_h\mhalfh,g_h^\e)_{\W_v} &= \frac{1}{\e}(v_h\mhalfh,g_h^\e-\mhalfh\rho_h^\e)_{\W_v} + \frac{1}{\e}(v_h\mhalfh,\mhalfh)_{\W_v}\rho_h^\e \\
    &= -\frac{2\theta}{\e}(\grad_v\mhalfh,g_h^\e-\mhalfh\rho_h^\e)_{\W_v}.  
\end{split}
\end{align}
Together \eqref{eqn:rho-and-J-stab-mp-g} and \eqref{eqn:rho-and-J-stab:2} imply that
\begin{align}
\begin{split}
    \|J_h^\e\|_{L_T^2(L^2(\W_x))} &\leq \frac{2\theta }{\e}\|\mhalfh\|_{H^1(\W_v)}\|\mhalfh\rho_h^\e-g_h^\e\|_{L_T^2(L^2(\W))} \lss \|g_{0,h}\|_{L^2(\W)}.
\end{split}
\end{align}
We now focus on \eqref{eqn:rho-and-J-stab-jmps}.  We will only prove the bound on the first term of \eqref{eqn:rho-and-J-stab-jmps} as the bound on the second term is similar.  Using the definition of $\gamma_I$; adding and subtracting $g_h^\e$; and using the trace inequality \eqref{eqn:trace_xv}, we obtain
\begin{align}\label{eqn:rho-and-J-stab:3}
\begin{split} 
    \sqrt{\gamma_*}\|\ljmp\rho_h^\e\rjmp\|_{L^2(\EIx)} &\leq \|\sqrt{\gamma_I}\ljmp\rho_h^\e\rjmp\|_{L^2(\EIx)} = \left\|\sqrt{\tfrac{|v_h\cdot{n_x}|}{2}}\ljmp \mhalfh\rho_h^\e\rjmp\right\|_{L^2(\EIx\times\W_v)} \\
    &\leq \left\|\sqrt{\tfrac{|v_h\cdot{n_x}|}{2}}\ljmp \mhalfh\rho_h^\e-g_h^\e\rjmp\right\|_{L^2(\EIx\times\W_v)} + \left\|\sqrt{\tfrac{|v_h\cdot{n_x}|}{2}}\ljmp g_h^\e\rjmp\right\|_{L^2(\EIx\times\W_v)} \\
    &\lss \frac{1}{\sqrt{h_x}}\|\mhalfh\rho_h^\e-g_h^\e\|_{L^2(\W)} + \left\|\sqrt{\tfrac{|v_h\cdot{n_x}|}{2}}\ljmp g_h^\e\rjmp\right\|_{L^2(\EIx\times\W_v)}.
\end{split}
\end{align}
Integrating \eqref{eqn:rho-and-J-stab:3} from $0$ to $T$ and using both \eqref{eqn:g-l2-stab} and \eqref{eqn:rho-and-J-stab-mp-g} yields
\begin{align}\label{eqn:rho-and-J-stab:4}
\begin{split} 
\sqrt{\gamma_*}\|\ljmp\rho_h^\e\rjmp\|_{L_T^2(L^2(\EIx))}
    &\lss \left(\frac{\e}{\sqrt{h_x}}+\sqrt{\e^{1-\beta}}\right)\|g_{0,h}\|_{L^2(\W)}.
\end{split}
\end{align}
If $\beta=1$, the proof is complete by noting that $\e\lss \sqrt{\e}$.  If $\beta=0$, then we can divide \eqref{eqn:rho-and-J-stab:4} by $\sqrt{\e}$ and arrive at \eqref{eqn:rho-and-J-stab-jmps}. 

For \eqref{eqn:rho-and-J-stab-proj}, we first focus on $\beta=0$.  Recall $I_h^\beta$ from \Cref{ass:interp}.  From  \eqref{eqn:FEM_interp_0} and \eqref{eqn:rho-and-J-stab-jmps} we have
\begin{align}\label{eqn:rho-and-J-stab:6}
\begin{split}
 \sqrt{\frac{h_x}{\e}}\|(\rho_h^\e-I_h\rho_h^\e)\|_{L_T^2(H_h^1(\W_x))} &\lss \frac{1}{\sqrt{\e}}\|\ljmp\rho_h^\e\rjmp\|_{L_T^2(L^2(\EIx))} + \frac{1}{\sqrt{\e}}\|\rho_h^\e\|_{L_T^2(L^2(\partial\W_x))} \\ 
 &\lss \left(\sqrt{\frac{\e}{h_x}}+1\right)\|g_{0,h}\|_{L^2(\W)}.
\end{split}
\end{align}
We now extend \eqref{eqn:rho-and-J-stab:6} to $\rho_h^\e-\mSh\rho_h^\e$ using the stability of $\mSh$.  By \Cref{ass:proj-stable} and the triangle inequality we have
\begin{align}\label{eqn:rho-and-J-stab:7}
\begin{split}
 \|\rho_h^\e-\mSh\rho_h^\e\|_{H_h^1(\W_x)} &\leq \|\rho_h^\e-I_h^\beta\rho_h^\e\|_{H_h^1(\W_x)} + \|I_h^\beta\rho_h^\e-\mSh\rho_h^\e\|_{H_h^1(\W_x)} \\
 &\leq \|\rho_h^\e-I_h^\beta\rho_h^\e\|_{H_h^1(\W_x)} + \|\mSh I_h^\beta\rho_h^\e-\mSh\rho_h^\e\|_{H_h^1(\W_x)} \\
 &\leq \|\rho_h^\e-I_h^\beta\rho_h^\e\|_{H_h^1(\W_x)} + \| \mSh(I_h^\beta\rho_h^\e-\rho_h^\e)\|_{H_h^1(\W_x)} \\
 &\lss \|\rho_h^\e-I_h^\beta\rho_h^\e\|_{H_h^1(\W_x)}.
\end{split}
\end{align}
Applying \eqref{eqn:rho-and-J-stab:6} to \eqref{eqn:rho-and-J-stab:7} yields 
\begin{align}\label{eqn:rho-and-J-stab:8}
 \sqrt{\frac{h_x}{\e}}\|(\rho_h^\e-\mSh\rho_h^\e)\|_{L_T^2(H_h^1(\W_x))}
 \lss \left(\sqrt{\frac{\e}{h_x}}+1\right)\|g_{0,h}\|_{L^2(\W)}.
\end{align}
We can extend \eqref{eqn:rho-and-J-stab:8} to the $L_T^2(L^2(\W_x))$ norm by \eqref{eqn:pf_ineq} and arrive at \eqref{eqn:rho-and-J-stab-proj}.  The case $\beta=1$ is similar.  The proof is complete.
\end{proof}

\subsection{Time Derivative Estimates}

In this subsection, we construct temporal estimates for $\partial_t \rho_h^\e$ and $\partial_t J_h^\e$ by determining the evolution equations for $\rho_h^\e$ and $J_h^\e$.  The evolution equations (see \Cref{lem:rho_eps_dde} and \Cref{lem:J_eps_dde}) are formed by choosing a particular type of test function in \cref{prob:discr}.  By adding and subtracting the discrete weighted equilibrium $\mhalfh\rho_h^\e$, we can write the evolution equations \eqref{eqn:rho_eps_dde} and \eqref{eqn:J_eps_dde} into the terms that will build the discretization of \eqref{eqn:dde-cont-system} and the remainder terms $\Theta_i$ where $\Theta_i$ is uniformly bounded in $\e$ when integrated over time.

We begin with the evolution equation for $\rho_h^\e$:
\begin{lemma}\label{lem:rho_eps_dde}
For any $\e>0$, $\rho_h^\e$ and $J_h^\e$ satisfy
\begin{align}\label{eqn:rho_eps_dde}
\begin{split}
    \left(\frac{\partial}{\partial t}\rho_h^\e,q_h\right)_{\W_x} &- (J_h^\e,\grad_x q_h)_{\W_x} + \left<\lavg J_h^\e\ravg,\ljmp q_h\rjmp\right>_{\EIx} +  \e^{\b-1}\left<\gamma_I \ljmp\rho_h^\e\rjmp,\ljmp q_h\rjmp\right>_{\EIx} \\
    &\quad+ \frac{1}{\e}\left<\gamma_B \rho_h^\e,{n_x}q_h\right>_{\partial\W_x} = \e^{\b}\Theta_1(\tilde{g}_h^\e,q_h) + \Theta_2(\tilde{g}_h^\e,q_h),
\end{split}
\end{align}
for all  $q_h\in V_{x,h}$, where $\tilde{g}_h^\e=g_h^\e-\mhalfh\rho_h^\e$  and 
\begin{align}
\Theta_1(\tilde{g}_h^\e,q_h) &=  -\frac{1}{\e}\left<\frac{|v_h\cdot{n_x}|}{2}\ljmp g_h^\e-\mhalfh\rho_h^\e\rjmp,\ljmp\mhalfh q_h\rjmp\right>_{\EIx\times\W_v}, \\
\Theta_2(\tilde{g}_h^\e,q_h) &= -\frac{1}{\e}\left<v_h(g_h^\e-\mhalfh\rho_h^\e),{n_x}\mhalfh q_h\right>_{\partial\W_+}.
\end{align}
Additionally, we have the following bounds:
\begin{align}\label{eqn:theta1_bnd}
|\Theta_1(\tilde{g}_h^\e,q_h)| &\lss \frac{1}{\e \sqrt{h_x}}\|g_h^\e-\mhalfh \rho_h^\e\|_{L^2(\W)}\|\ljmp q_h\rjmp\|_{L^2(\EIx)}, \\
|\Theta_2(\tilde{g}_h^\e,q_h)| &\lss \label{eqn:theta2_bnd} \frac{1}{\e \sqrt{h_x}}\|g_h^\e-\mhalfh\rho_h^\e\|_{L^2(\W)}\| q_h\|_{L^2(\partial\W_x)}.
\end{align}
\end{lemma}

\begin{proof}
 To show \eqref{eqn:rho_eps_dde},  we let $q_h\in V_{x,h}$ and choose $z_h=\mhalfh q_h\in V_h$ into \eqref{eqn:alt_forms_eq} and evaluate term by term.  First, the time derivative term reduces to
\begin{align}\label{eqn:rho_eps_dde:1}
    \e\left(\frac{\partial}{\partial t}g_h^\e,\mhalfh q_h\right)_\W = \e\left(\frac{\partial}{\partial t}(g_h^\e,\mhalfh)_{\W_v},q_h\right)_{\W_x} = \e\left(\frac{\partial}{\partial t}\rho_h^\e,q_h\right)_{\W_x}.
\end{align}
Next, using the definition of $J_h^\e$ in \eqref{eqn:J_h}, we compute
\begin{align}\label{eqn:rho_eps_dde:2}
\begin{split}
    \mA(g_h^\e,\mhalfh q_h) 
    &= -\left( (v_hg_h^\e,\mhalfh)_{\W_v},\grad_xq_h\right)_{\W_x} + \left<\lavg(v_hg_h^\e,\mhalfh)_{\W_v}\ravg,\ljmp q_h\rjmp\right>_{\EIx} \\
    &\quad+  \e^\b\left<\tfrac{1}{2}|v_h\cdot{n_x}|\ljmp g_h^\e\rjmp,\ljmp \mhalfh q_h\rjmp \right>_{\EIx\times\W_v} + \left<v_hg_h^\e,{n_x}\mhalfh q_h\right>_{\partial\W_+} \\
    &= -\e(J_h^\e,\grad q_h)_{\W_x} + \e\left<\lavg J_h^\e\ravg,\ljmp q_h\rjmp\right>_{\EIx} \\
     &\quad+  \e^\b\left<\tfrac{1}{2}|v_h\cdot{n_x}|\ljmp g_h^\e\rjmp,\ljmp \mhalfh q_h\rjmp \right>_{\EIx\times\W_v} + \left<v_hg_h^\e,{n_x}\mhalfh q_h\right>_{\partial\W_+}
\end{split}
\end{align}   
Adding and subtracting $\mhalfh\rho_h^\e$ from the last two terms of \eqref{eqn:rho_eps_dde:2} and using the definitions of  $\Theta_1$ and $\Theta_2$ gives 
\begin{align}\label{eqn:rho_eps_dde:2-1}
\begin{split}
    \mA(g_h^\e,\mhalfh q_h) 
    &= -\e(J_h^\e,\grad q_h)_{\W_x} + \e\left<\lavg J_h^\e\ravg,\ljmp q_h\rjmp\right>_{\EIx} 
    \\ &\quad
    + \e^\beta\left<\tfrac{1}{2}|v_h\cdot{n_x}|\ljmp \mhalfh\rho_h^\e\rjmp,\ljmp\mhalfh q_h\rjmp \right>_{\EIx\times\W_v} 
    + \left<v_h\mhalfh\rho_h^\e,{n_x}\mhalfh q_h\right>_{\partial\W_+} 
    \\ &\quad
    + \e^\b\left<\tfrac{1}{2}|v_h\cdot{n_x}|\ljmp g_h^\e-\mhalfh\rho_h^\e\rjmp,\ljmp \mhalfh q_h\rjmp \right>_{\EIx\times\W_v} 
    + \left<v_h(g_h^\e-\mhalfh\rho_h^\e),{n_x}\mhalfh q_h\right>_{\partial\W_+} \\
    &= -\e(J_h^\e,\grad q_h)_{\W_x} + \e\left<\lavg J_h^\e\ravg,\ljmp q_h\rjmp\right>_{\EIx} 
    + \e^\beta\left<\gamma_I \ljmp\rho_h^\e\rjmp,\ljmp q_h\rjmp\right>_{\EIx} +  \left<\gamma_B \rho_h^\e,{n_x}q_h\right>_{\partial\W_x} \\
    & \quad - \e^{\b+1}\Theta_1(\tilde{g}_h^\e,q_h) - \e \Theta_2(\tilde{g}_h^\e,q_h),
\end{split}
\end{align}
After division by $\e$, \eqref{eqn:rho_eps_dde:1} and \eqref{eqn:rho_eps_dde:2-1} recover \eqref{eqn:rho_eps_dde}.  Thus it remains to show that 
\begin{equation}
    \label{eqn:sum_BDQ_is_C}
    \mB(g_h^\e,\mhalfh q_h)+\mD(g_h^\e,\mhalfh q_h)+\mQ(g_h^\e,\mhalfh q_h) = \mC(g_h^\e,\mhalfh q_h).
\end{equation}  
For $\mB$, any edge integral in \eqref{eqn:B} vanishes because $\mhalfh q_h$ is continuous in $v$.  Thus by the definition of the discrete velocity $v_h$ in \eqref{eqn:v_h},
\begin{align}\label{eqn:rho_eps_dde:3}
\begin{split}
    \mB(g_h^\e,\mhalfh q_h) = -(Eg_h^\e,q_h\grad_v\mhalfh)_\W = \frac{1}{2\theta}(E\cdot v_hg_h^\e,\mhalfh q_h)_\W =  \mC(g_h^\e,\mhalfh q_h). 
\end{split}
\end{align}
For $\mD$, Assumption \ref{ass-discrete-root-Max}.\ref{ass-symm} implies that $\mD(g_h^\e,\mhalfh q_h) = 0$.  
For $\mQ$, because $\mhalfh q_h$ is isotropic, $\mQ(g_h^\e,\mhalfh q_h)=0$ as well.  Thus \eqref {eqn:sum_BDQ_is_C} holds and consequently so does \eqref{eqn:rho_eps_dde}.

We now prove the bounds on $\Theta_1$ and $\Theta_2$.  For $\Theta_1$, H\"older's inequality and the trace inequality \eqref{eqn:trace_xv} give
\begin{equation}
\begin{split}
\label{eqn:rho_eps_dde:4}
    |\Theta_1(\tilde{g}_h^\e,q_h)| 
    &\lss \frac{1}{\e}\|v_h \mhalfh\|_{L^\infty(\W_v)}  \|\ljmp g_h^\e-\mhalfh\rho_h^\e \rjmp\|_{L^2(\EIx)} \|\ljmp q_h\rjmp\|_{L^2(\EIx)} \\
    & \lss \frac{1}{\e \sqrt{h_x}}\|v_h \mhalfh\|_{L^\infty(\W_v)} \|g_h^\e-\mhalfh\rho_h^\e\|_{L^2(\W)}\|\ljmp q_h\rjmp\|_{L^2(\EIx)}
\end{split} 
\end{equation}
which is \eqref{eqn:theta1_bnd}.  A similar argument for $\Theta_2$ recovers \eqref{eqn:theta2_bnd}. The proof is complete.
\end{proof}

Using \eqref{eqn:rho_eps_dde}, we derive an $\e$-independent bound for $\partial_t\rho_h^\e$ in various norms.
\begin{lemma}\label{lem:dt_rho_bound}
For any $\e\leq\e_{h_v}$, where $\e_{h_v}$ is defined in \eqref{eqn:eps_h}, and all $h_x>0$,
\begin{align}\label{eqn:dt_rho_bound}
    \left\|\partial_t\rho_h^\e\right\|_{L^2(H_{h,\beta}^{-1}(\W_x))} &\lss \|g_{0,h}\|_{L^2(\W_x)}, \\
    \left\|\partial_t\mSh\rho_h^\e\right\|_{L_T^2(L^2(\W_x))} &\lss \frac{1}{h_x}\|g_{0,h}\|_{L^2(\W_x)}, \label{eqn:dt_rho_proj_bound}
\end{align}
where the $H_{h,\beta}^{-1}(\W_x)$ norm is defined in $\eqref{eqn:dual_norms}$.
\end{lemma}

\begin{proof}
We first focus on \eqref{eqn:dt_rho_bound} and $\beta=1$. Let $q_h\in V_{x,h}^0$ in $\eqref{eqn:rho_eps_dde}$, with $q_h\neq 0$. For such $q_h$, $\Theta_2(\tilde{g}_h^\e,q_h) = 0$; therefore
\begin{align}\label{eqn:dt_rho_bound:1}
\begin{split}
    \left(\partial_t\rho_h^\e,q_h\right) &= (J_h^\e,\grad_x q_h)_{\W_x} - \left<\lavg J_h^\e\ravg,\ljmp q_h\rjmp\right>_{\EIx} - \left<\gamma_I \ljmp\rho_h^\e\rjmp,\ljmp q_h\rjmp\right>_{\EIx} + \Theta_1(g_h^\e,q_h).
\end{split}
\end{align}
We now bound the remaining terms on the right-hand side of \eqref{eqn:dt_rho_bound:1}.  For the first term, the Cauchy-Schwarz inequality implies that,
\begin{align}\label{eqn:dt_rho_bound:2}
    (J_h^\e,\grad_x q_h)_{\W_x} \lss \|\grad_x  q_h\|_{L^2(\W_x)}\|J_h^\e\|_{L^2(\W_x)}.
\end{align}
For the second and third terms, the Cauchy-Schwarz inequality and the trace inequality \eqref{eqn:trace} imply that
\begin{align}\label{eqn:dt_rho_bound:3}
    \left<\lavg J_h^\e\ravg,\ljmp q_h\rjmp\right>_{\EIx} \lss \frac{1}{\sqrt{h_x}}\|\ljmp q_h\rjmp\|_{L^2(\EIx)}\|J_h^\e\|_{L^2(\W)}
\end{align}
and
\begin{align}\label{eqn:dt_rho_bound:4}
    \left<\gamma_I\ljmp \rho_h^\e\rjmp,\ljmp q_h\rjmp\right>_{\EIx} \lss \frac{1}{\sqrt{h_x}}\|\ljmp q_h\rjmp\|_{L^2(\EIx)}\|\rho_h^\e\|_{L^2(\W)}.
\end{align}
Substituting \eqref{eqn:dt_rho_bound:2} through \eqref{eqn:dt_rho_bound:4} and the bound on $\Theta_1$ in \eqref{eqn:rho_eps_dde:4} into \eqref{eqn:dt_rho_bound:1} gives
\begin{align}\label{eqn:dt_rho_bound:6}
\begin{split}
    \left(\partial_t\rho_h^\e,q_h\right)_{\W_x} &\lss  \bigg(\|J_h^\e\|_{L^2(\W_x)} + \|\rho_h^\e\|_{L^2(\W_x)} +\frac{1}{\e}\|g_h^\e-\mhalfh\rho_h^\e\|_{L^2(\W_x)}\bigg)\|q_h\|_{H^1_h(\W_x)}.
\end{split}
\end{align}
After dividing \eqref{eqn:dt_rho_bound:6} by $\|q_h\|_{H^1_h(\W_x)}$ and taking the supremum over all $q_h \in {V^0_{x,h}}$, the result is
\begin{align}\label{eqn:dt_rho_bound:7}
\begin{split}
    \left\|\partial_t\rho_h^\e\right\|_{H_{h,\beta}^{-1}(\W_x)} &\lss  \|J_h^\e\|_{L^2(\W_x)} + \|\rho_h^\e\|_{L^2(\W_x)} + \frac{1}{\e}\|g_h^\e-\mhalfh\rho_h^\e\|_{L^2(\W_x)}.
\end{split}
\end{align}
Integrating \eqref{eqn:dt_rho_bound:7} over $t\in[0,T]$ and applying \eqref{eqn:rho-and-J-stab-g}--\eqref{eqn:rho-and-J-stab-J} yields \eqref{eqn:dt_rho_bound}.  

When $\beta=0$, the bound on \eqref{eqn:dt_rho_bound} is simpler to show.  Indeed, if $q_h \in {S^0_{x,h}}$ is continuous, only the first term on the right-hand side of \eqref{eqn:dt_rho_bound:1} remains.  This term is bounded using \eqref{eqn:dt_rho_bound:2} so that the argument follows when replacing $\|q_h\|_{H^1_h(\W_x)}$ by $\|\grad_x  q_h\|_{L^2(\W_x)}$.

To show \eqref{eqn:dt_rho_proj_bound}, we first use the trace inequality \eqref{eqn:trace} and inverse inequality \eqref{eqn:inverse} to obtain
\begin{equation}\label{eqn:dt_rho_bound:8}
    \|q_h\|_{H^1_h(\W_x)} \lss \frac{1}{h_x}\|q_h\|_{L^2(\W_x)}    
\end{equation}
Choosing $q_h=\partial_t\mSh\rho_h^\e$ in \eqref{eqn:dt_rho_bound:6} and using \eqref{eqn:dt_rho_bound:8} and the identity
\begin{equation}
\left(\partial_t\rho_h^\e,\partial_t\mSh\rho_h^\e\right)_{\W_x} = \left(\partial_t\mSh\rho_h^\e,\partial_t\mSh\rho_h^\e\right)_{\W_x}
\end{equation}
we see that the desired estimate holds.  The proof is complete. 
\end{proof}

We next turn to the evolution equation for $J_h^\e$.
\begin{lemma}\label{lem:J_eps_dde}
For any $\e>0$, $\rho_h^\e$ and $J_h^\e$ satisfy
\begin{align}\label{eqn:J_eps_dde}
\begin{split}
    \e^2\left(\frac{\partial}{\partial t}J_h^\e,\tau_h\right)_{\W_x} &+ (\omega J_h^\e,\tau_h)_{\W_x} + \theta(\grad_x\rho_h^\e,\tau_h)_{\W_x} - \theta\left<\ljmp\rho_h^\e\rjmp,\lavg \tau_h\ravg\right>_{\EIx} \\
    &\quad - (E\rho_h^\e,\tau_h)_{\W_x} = \e\Theta_3(\tilde{g}_h^\e,\tau_h) + \sqrt{\e^{\beta+1}}\Theta_4(\rho_h^\e,\tau_h) + \sqrt{\e}\Theta_5(\rho_h^\e,\tau_h),
\end{split}
\end{align}
for all $\tau_h\in [V_{x,h}]^3$.  Here $\Theta_3$ is a remainder that includes several terms that depend on $\tilde{g}_h^\e=g_h^\e-\mhalfh\rho_h^\e$; it satisfies the bound
\begin{align}
|\Theta_3(\tilde{g}_h^\e,\tau_h)| &\lss \frac{1}{\e h_x}\|g_h^\e-\mhalfh\rho_h^\e\|_{L^2(\W)}\|\tau_h\|_{L^2(\W_x)}. \label{eqn:theta3_bnd} \end{align}
The terms
\begin{align}
    \Theta_4(\rho_h^\e,\tau_h) &=  -\frac{1}{\sqrt{\e^{1-\b}}}\left<\tfrac{|v_h\cdot{n_x}|}{2}\mhalfh\ljmp\rho_h^\e\rjmp,\ljmp \mhalfh v_h\cdot\tau_h\rjmp\right>_{\EIx\times \W_v}, \label{eqn:theta4_def} \\
    \Theta_5(\rho_h^\e,\tau_h) &= \frac{1}{\sqrt{\e}}\left<v_h\cdot n_x \mhalfh\rho_h^\e,\mhalfh v_h\cdot\tau_h\right>_{\partial\W_-}, \label{eqn:theta5_def} 
\end{align}
are also remainder terms satisfying the bounds
\begin{alignat}{2}
|\Theta_4(\rho_h^\e,\tau_h)| 
&\lss \frac{1}{\sqrt{\e^{1-\b}h_x}}\|\ljmp\rho_h^\e\rjmp\|_{L^2(\EIx)}\|\tau_h\|_{L^2(\W_x)} &\lss \frac{1}{\sqrt{\e^{1-\b}}h_x}\|\rho_h^\e\|_{L^2(\W_x)}\|\tau_h\|_{L^2(\W_x)},
\label{eqn:theta4_bnd}
\\
|\Theta_5(\rho_h^\e,\tau_h)| &\lss \frac{1}{\sqrt{\e h_x}}\|\rho_h^\e\|_{L^2(\partial\W_x)}\|\tau_h\|_{L^2(\W_x)} &\lss \frac{1}{\sqrt{\e}h_x}\|\rho_h^\e\|_{L^2(\W_x)}\|\tau_h\|_{L^2(\W_x)}. \label{eqn:theta5_bnd}
\end{alignat}
\end{lemma}

\begin{proof}
Let $v_h^i=v_h\cdot e_i$, where $e_i$ is the standard unit basis vector in the $i$-th coordinate direction.  From the definition of $v_h$ in \eqref{eqn:v_h}, it follows that $v_h^i\mhalfh\in V_{v,h}$.  To derive \eqref{eqn:J_eps_dde}, we let $z_h=v_h^i \mhalfh \tau_h \in V_h$ in \eqref{eqn:alt_forms_eq}, where $\tau_h\in V_{x,h}$ is arbitrary,
and evaluate the result term by term.  
For brevity, we will not gather all of the terms of $\Theta_3$ together, but rather identify each piece as we go and show that it satisfies the bound in \eqref{eqn:theta3_bnd}.  

For the time derivative, the definition of $J_h^\e$ in \eqref{eqn:J_h} gives
\begin{align}\label{eqn:J_eps_dde:4}
\begin{split}
    \e\left(\frac{\partial}{\partial t}g_h^\e,v_h^i\mhalfh\tau_h\right)_\W = \e\left(\frac{\partial}{\partial t}(g_h^\e,v_h\mhalfh)_{\W_v},\tau_h e_i\right)_{\W_x} = \e^2\left(\frac{\partial}{\partial t}J_h^\e,\tau_h e_i\right)_{\W_x}.
\end{split}
\end{align}

To evaluate $\mA$, we add and subtract $\mhalfh\rho_h^\e$ from the first argument and write
\begin{align}\label{eqn:J_eps_dde:1}
\begin{split}
    \mA(g_h^\e,v_h^i\mhalfh\tau_h) &= \mA(\mhalfh\rho_h^\e,v_h^i\mhalfh\tau_h) - \e\left\{\frac{1}{\e}\mA(g_h^\e-\mhalfh\rho_h^\e,v_h^i\mhalfh\tau_h)\right\}\\
    &=: I_1 - \e\{I_2\}.
\end{split}
\end{align}
The term $I_2$ belongs to $\Theta_3$.  Since $I_2$ contains $g_h^\e-\mhalfh\rho_h^\e$, following a similar treatment to $\Theta_1$ in \Cref{lem:rho_eps_dde}, we can show $I_2$ satisfies \eqref{eqn:theta3_bnd}.  For $I_1$, the definition of $\mathcal{A}$ in \eqref{eqn:A} implies that,
\begin{equation}
    I_1 =I_3 - \sqrt{\e^{\b+1}}\{I_4\} - \sqrt{\e}\{I_5\},
\end{equation}
where 
\begin{subequations}
\label{eqn:J_eps_dde:2}
\begin{align}
    \begin{split}
    I_3 &= -(\rho_h^\e,\grad_x \tau_h)_{\W_x}\cdot (v_h\mhalfh,v_h^i\mhalfh)_{\W_v}
    + \left<\lavg\rho_h^\e\ravg,\ljmp\tau_h\rjmp\right>_{\EIx}\cdot (v_h\mhalfh,v_h^i\mhalfh)_{\W_v} \\
        &\quad+ \left<\rho_h,n_x\tau_h\right>_{\partial\W_x}\cdot (v_h\mhalfh,v_h^i\mhalfh)_{\W_v}
    \end{split}
    \label{eqn:J_eps_dde:2:I_3}\\
    I_4 & = -\frac{1}{\sqrt{\e^{1-\beta}}}\left<\tfrac{|v_h\cdot{n_x}|}{2}\mhalfh\ljmp \rho_h^\e\rjmp,\ljmp \mhalfh v_h \cdot e_i\tau_h\rjmp\right>_{\EIx\times {\W_v}} \\
    I_5 &= \frac{1}{\sqrt{\e}}\left<v_h\cdot n_x \mhalfh\rho_h^\e,\mhalfh v_h\cdot e_i\tau_h\right>_{\partial\W_-}, 
\end{align}
\end{subequations}
The definitions of $\mhalfh$ and $v_h$ in \Cref{defn:root-Maxwellian-and-discrete-velocity}, combined with Assumptions \ref{ass-discrete-root-Max}.\ref{ass-momen} and \ref{ass-discrete-root-Max}.\ref{ass-energy}, imply that
\begin{align}
\label{eqn:I3_max_ints}
    (v_h\mhalfh,v_h^i\mhalfh)_{\W_v} 
        = (-2 \theta \nabla_v \mhalfh,-2 \theta \partial_{v_i} \mhalfh)_{\W_v} 
        = \theta e_i.
\end{align}
Substituting \eqref{eqn:I3_max_ints} into \eqref{eqn:J_eps_dde:2:I_3} and then applying the discrete integration-by-parts formula from \eqref{eqn:int_by_parts} gives
\begin{align}
\begin{split}\label{eqn:J_eps_dde:3}
    I_3 &= -\theta(\rho_h^\e e_i,\grad_x \tau_h)_{\W_x} + \theta\left<\lavg\rho_h^\e e_i\ravg ,\ljmp\tau_h\rjmp\right>_{\EIx}+\theta\left<\rho_h^\e e_i,{n_x}\tau_h\right>_{\partial\W_x} \\
    &= \theta(\grad_x\rho_h^\e,\tau_he_i) - \theta\left<\ljmp\rho_h^\e\rjmp,\lavg\tau_h e_i\ravg\right>_{\EIx}.
\end{split}
\end{align}
Meanwhile $I_4$ is the only component of $\Theta_4$ and can be bounded using the trace inequality \eqref{eqn:trace} to obtain
\begin{align}
    |I_4| \lss \frac{1}{\sqrt{h_x\e^{1-\b}}}\|v_h\|_{L^\infty(\W_v)}^2\|\mhalfh\|_{L^\infty(\W_v)}^2\|\ljmp\rho_h^\e\rjmp\|_{L^2(\EIx)}\|\tau_h\|_{L^2(\W_x)}.
\end{align}
Likewise, $I_5$ is the only component of $\Theta_5$ and can be bounded in a similar fashion.

To evaluate $\mB$, we add and subtract $\mhalfh\rho_h^\e$ from the first argument and write
\begin{align}\label{eqn:J_eps_dde:5}
\begin{split}
    \mB(g_h^\e,v_h^i\mhalfh\tau_h)&=\mB(\mhalfh\rho_h^\e,v_h^i\mhalfh\tau_h)
    -\e\left\{\frac{1}{\e}\mB(g_h^\e-\mhalfh\rho_h^\e,v_h^i\mhalfh\tau_h)\right\} \\
    &= I_5 + \e\{I_6\}.
\end{split}
\end{align}
Here $I_6$ is a remainder term that belongs to $\Theta_3$ and satisfies the bound in \eqref{eqn:theta3_bnd} due to the trace estimate \eqref{eqn:trace_vx} and inverse estimate \eqref{eqn:inverse}.  Meanwhile, the upwind penalty term in $\mB(\mhalfh\rho_h^\e,v_h^i\mhalfh\tau_h)$ vanishes because the first argument is continuous in $v$; this leaves
\begin{equation}
\label{eqn:I5_def}
   I_5
   = -\left (E\rho_h^\e ,\tau_h \right)_{\W_x} \cdot
   \left [ \left (\mhalfh,\grad_v(v_h^i\mhalfh)\right)_{\W_v} -
   \left< \mhalfh,\ljmp v_h^i\mhalfh\rjmp \right>_{\EIv} \right].
\end{equation}
From the integration-by-parts identity \eqref{eqn:int_by_parts} (applied to functions in $V_{v,h}$) and continuity of $\mhalfh$,
\begin{equation}\label{eqn:I5_int_by_parts}
\left (\mhalfh,\grad_v(v_h^i\mhalfh)\right)_{\W_v} -
   \left< \mhalfh,\ljmp v_h^i\mhalfh\rjmp \right>_{\EIv} 
   = 
   -\left(\grad_v\mhalfh, v_h^i\mhalfh \right)_{\W_v} 
   + \left< \mhalfh,n_v v_h^i\mhalfh\right>_{\partial\W_v}
\end{equation}
The definition of $v_h$ in \eqref{eqn:v_h}, along with \eqref{eqn:I3_max_ints}, implies that for the first term above,
\begin{equation}
\label{eqn:I3_max_ints:2}
    -\left(\grad_v\mhalfh, v_h^i\mhalfh \right)_{\W_v} 
    =  \frac{1}{2 \theta} \left(v_h\mhalfh, v_h^i\mhalfh \right)_{\W_v} 
    = \frac12 e_i.
\end{equation}
Substituting \eqref{eqn:I5_int_by_parts} with \eqref{eqn:I3_max_ints:2} into \eqref{eqn:I5_def} and recalling the definition of $\mD$ from \eqref{eqn:D} gives
\begin{equation}
    \label{eqn:J_eps_dde:6}
\begin{split}
I_5 
    &= -\frac12 (E\rho_h^\e,\tau_h e_i)_{\W_x} 
    + \frac12 (E\rho_h^\e,\tau_h e_i)_{\W_x} \left< \mhalfh,n_v v_h^i\mhalfh\right>_{\partial\W_v} \\
    &= -\frac12 (E\rho_h^\e,\tau_h e_i)_{\W_x} - \mD(g_h^\e,v_h^i\mhalfh\tau_h)
\end{split}
\end{equation}
To evaluate $\mQ$, we use Assumption \ref{ass-discrete-root-Max}.\ref{ass-momen}:
\begin{align}\label{eqn:J_eps_dde:7}
\begin{split}
    \frac{1}{\e}\mQ(g_h^\e,v_h^i\mhalfh\tau_h) &= -\frac{1}{\e}(\omega(\mhalfh\rho_h^\e -g_h^\e ),v_h^i\mhalfh\tau_h)_{\W} \\
    &= -\frac{1}{\e}(\omega\rho_h^\e ,\tau_h)_{\W_x}(\mhalfh,v_h^i\mhalfh)_{\W_v} + \left( \frac{\omega}{\e}(g_h^\e ,v_h\mhalfh)_{\W_v},\tau_he_i\right)_{\W_x} \\
    &= (\omega J_h^\e,\tau_he_i)_{\W_x}.
\end{split}
\end{align}
Lastly, to evaluate $\mC$, we add and subtract $\mhalfh\rho_h^\e$ from the first argument and write
\begin{align}\label{eqn:J_eps_dde:8}
\begin{split}
    \mC(g_h^\e,v_h^i\mhalfh\tau_h) &= \frac{1}{2\theta}(E\cdot v_h \mhalfh\rho_h^\e,v_h^i\mhalfh\tau_h)_\W 
    +\e\Bigg\{\frac{1}{2\theta\e}(E\cdot v_h (g_h^\e-\mhalfh\rho_h^\e),v_h^i\mhalfh\tau_h)_\W\Bigg\} \\
    &= I_7+\e\{I_8\}.
\end{split}
\end{align}
Here $I_8$ is a remainder term of $\Theta_3$ which satisfies the bound in \eqref{eqn:theta3_bnd} and,  because of \eqref{eqn:I3_max_ints},
\begin{align}\label{eqn:J_eps_dde:9}
\begin{split}
    I_7 
    &= \frac{1}{2\theta} \left(E\rho_h^\e,\tau_h \right)_{\W_x}\cdot 
    (v_h \mhalfh,v_h^i\mhalfh)_{\W_v}
    = \frac{1}{2}(E\rho_h^\e,\tau_h e_i)_{\W_x}.
\end{split}
\end{align}

We have shown \eqref{eqn:J_eps_dde} for all $\tau_h e_i$ where $\tau_h\in V_{x,h}$, and thus \eqref{eqn:J_eps_dde} holds for all $\tau_h\in [V_{x,h}]^3$.  The proof is complete.
\end{proof}

We now use \eqref{eqn:J_eps_dde} to get a space-time bound on $\partial_t J_h^\e$.  Recall the definition of $\e_{h_v}$ in \eqref{eqn:eps_h}.

\begin{lemma}\label{lem:dt_J_eps}
Assume $\e\leq\e_{h_v} \lss {h_x} \leq 1$.  Then 
\begin{align} \label{eqn:J_eps_inf}
    \|J_h^\e\|_{L_T^\infty(L^2(\W_x))}& \lss\frac{1}{h_x}\|g_{0,h}\|_{L^2(\W)}, \\
    \e^{3/2}\left\|\partial_t J_h^\e\right\|_{L_T^2(L^2(\W_x))}&\lss \frac{1}{\sqrt{h_x}} \|g_{0,h}\|_{L^2(\W)}. \label{eqn:dt_J_eps}
\end{align}
\end{lemma}

\begin{proof}
We will first prove \eqref{eqn:J_eps_inf} which is an estimate needed to obtain \eqref{eqn:dt_J_eps}.   Setting $\tau_h=J_h^\e$ in  \eqref{eqn:J_eps_dde} gives
\begin{align}\label{eqn:J_eps_inf:1}
\begin{split}
    \frac{\e^2}{2}\frac{\dx[]}{\dx[t]}\|J_h^\e\|_{L^2(\W_x)}^2 &+ \omega_{\min}\|J_h^\e\|_{L^2(\W_x)}^2 \leq -\theta(\grad_x\rho_h^\e,J_h^\e)_{\W_x} + \theta\left<\ljmp\rho_h^\e\rjmp,\lavg J_h^\e\ravg\right>_{\EIx} \\
    &\quad + (E\rho_h^\e,J_h^\e)_{\W_x} + \e\Theta_3(\tilde{g}_h^\e,J_h^\e) + \e^\b\Theta_4(\rho_h^\e,J_h^\e) + \sqrt{\e}\Theta_5(\rho_h^\e,J_h^\e).
\end{split}
\end{align}
The first three terms on the right-hand side of \eqref{eqn:J_eps_inf:1} can be bounded using Cauchy-Schwarz together with the inverse inequality \eqref{eqn:inverse} for the first, a trace inequality \eqref{eqn:trace} for the second, and the $L^\infty$ bound on $E$ for the third.  After absorbing $\e$-independent constants,
\begin{equation}
    -\theta(\grad_x\rho_h^\e,J_h^\e)_{\W_x} 
    + \theta\left<\ljmp\rho_h^\e\rjmp,\lavg J_h^\e\ravg\right>_{\EIx} \\
    +(E\rho_h^\e,J_h^\e)_{\W_x} 
    \lss \frac{1}{h_x} \|\rho_h^\e\|_{L^2(\W_x)}\|J_h^\e\|_{L^2(\W_x)}
\end{equation}
Meanwhile, the bounds on $\Theta_3$, $\Theta_4$, and $\Theta_5$ from \Cref{lem:rho_eps_dde} and \Cref{lem:J_eps_dde} imply that
\begin{align}
\begin{split}
    \e\Theta_3(\tilde{g}_h^\e,J_h^\e) + \sqrt{\e^{\b+1}}\Theta_4(\rho_h^\e,J_h^\e) + \sqrt{\e}\Theta_5(\rho_h^\e,J_h^\e)
    &\lss \frac{1}{h_x}  \Big(\|\rho_h^\e\|_{L^2(\W_x)} \\
    &\quad+ \|g_h^\e-\mhalfh\rho_h^\e\|_{L^2(\W)}\Big)\|J_h^\e\|_{L^2(\W_x)}.
\end{split}
\end{align}
Substituting these bounds into \eqref{eqn:J_eps_inf:1} and dividing by $\e^2\|J_h\|_{L^2(\W_x)}$ yields
\begin{align}\label{eqn:J_eps_inf:2}
    \frac{\dx[]}{\dx[t]}\|J_h^\e\|_{L^2(\W_x)} + \frac{\omega_{\min}}{\e^2}\|J_h^\e\|_{L^2(\W_x)} \lss \frac{1}{\e^2 h_x}\left(\|\rho_h^\e\|_{L^2(\W_x)} + \|g_h^\e-\mhalfh\rho_h^\e\|_{L^2(\W)}\right).
\end{align}
According to \eqref{eqn:init_J}, $J_h^\e |_{t=0} =0$.  Thus Gr\"onwall's Lemma applied to \eqref{eqn:J_eps_inf:2}, along with the $L^\infty$ in time bound in \eqref{eqn:rho-and-J-stab-g}, recovers \eqref{eqn:J_eps_inf}.

We now prove \eqref{eqn:dt_J_eps}.  As the proof below is quite technical, we first briefly summarize the process.  The idea is to pass the time derivative from $J_h^\e$ to $\rho_h^\e$ for the terms that are not sufficiently small in $\e$ to bound with the usual techniques.  However, since $\partial_t\rho_h^\e$ is not uniformly bounded w.r.t $\e$ in $L_T^2(L^2(\W_x))$, we will first add and subtract its projection $
\mSh\rho_h^\e$, whose time derivative is uniformly bounded in $\e$ by \Cref{lem:dt_rho_bound}, before passing the time derivative over.  Having an explicit bound on the size of $\rho_h^\e-\mSh\rho_h^\e$ by \eqref{eqn:rho-and-J-stab-proj}, we can obtain the $\e^{3/2}$ scale in \eqref{eqn:dt_J_eps}.

Setting $\tau_h=\e\partial_t J_h^\e$ in  \eqref{eqn:J_eps_dde} gives 
\begin{align}\label{eqn:dt_J_eps:1}
\begin{split}
    \e^3\left\|\partial_t J_h^\e\right\|_{L^2(\W_x)}^2 + \frac{\e}{2}\frac{\dx[]}{\dx[t]}\|\sqrt{\omega}J_h^\e\|_{L^2(\W_x)}^2 &= -\e\theta(\grad_x\rho_h^\e,\partial_t J_h^\e)_{\W_x} + \e\theta\left<\ljmp\rho_h^\e\rjmp,\lavg \partial_t J_h^\e\ravg\right>_{\EIx} \\
    &\quad + \e(E\rho_h^\e,\partial_t J_h^\e)_{\W_x} + \e^2\Theta_3(\tilde{g}_h^\e,\partial_t J_h^\e) \\
    &\quad+ \e^{\b+1}\Theta_4(\rho_h^\e,\partial_t J_h^\e)+\e^{3/2}\Theta_5(\rho_h^\e,\partial_t J_h^\e).
\end{split}
\end{align}
We then integrate \eqref{eqn:dt_J_eps:1} over $t \in [0,T]$, use the zero initial condition in \eqref{eqn:init_J}, and drop the positive term $\|\sqrt{\omega}J_h^\e\|_{L^2(\W_x)}^2|_{t=T}$.  This gives 
\begin{multline}
\label{eqn:dt_J_eps:2}
    \e^3\left\|\partial_t J_h^\e\right\|_{L_T^2(L^2(\W_x))}^2  
    \leq \int_0^T \Big[-\e\theta(\grad_x\rho_h^\e,\partial_t J_h^\e)_{\W_x} + \e\theta\left<\ljmp\rho_h^\e\rjmp,\lavg \partial_t J_h^\e\ravg\right>_{\EIx} \\
    + \e(E\rho_h^\e,\partial_t J_h^\e)_{\W_x} + \e^2\Theta_3(\tilde{g}_h^\e,\partial_t J_h^\e)
      + \e^{(\b+3)/2}\Theta_4(\rho_h^\e,\partial_t J_h^\e)+\e^{3/2}\Theta_5(\rho_h^\e,\partial_t J_h^\e) \Big] \dx[t].
\end{multline}
We now add and subtract $\mSh\rho_h^\e$ (recall $\mSh$ is an $L^2$ projection) to several terms of \eqref{eqn:dt_J_eps:2}:
\begin{align}\label{eqn:dt_J_eps:3}
\begin{split}
    \e^3\left\|\partial_t J_h^\e\right\|_{L_T^2(L^2(\W_x))}^2  
    &\leq \int_0^T \Big[ \Big\{-\e\theta(\grad_x(\rho_h^\e-\mSh\rho_h^\e),\partial_t J_h^\e)_{\W_x} \\
        & \qquad \qquad + \e\theta\left<\ljmp\rho_h^\e-\mSh\rho_h^\e\rjmp,\lavg \partial_t J_h^\e\ravg\right>_{\EIx}
        + \e(E(\rho_h^\e-\mSh\rho_h^\e),\partial_t J_h^\e)_{\W_x} \Big\}\\
    &\quad + \Big\{-\e\theta(\grad_x\mSh\rho_h^\e,\partial_t J_h^\e)_{\W_x} + \e\theta\left<\ljmp\mSh\rho_h^\e\rjmp,\lavg \partial_t J_h^\e\ravg\right>_{\EIx}\Big\} \\
    &\quad + \Big\{\e(E\mSh\rho_h^\e,\partial_t J_h^\e)_{\W_x} \Big\} \\
    &\quad + \Big\{\e^2\Theta_3(\tilde{g}_h^\e,\partial_t J_h^\e) 
    + \e^{(\b+3)/2}\Theta_4(\rho_h^\e,\partial_t J_h^\e)
        +\e^{3/2}\Theta_5(\rho_h^\e,\partial_t J_h^\e)\Big\} \Big] \dx[t] \\
    &\leq \int_0^T \Big[ \{I_1\} + \{I_2\} + \{I_3\} + \{I_4\}\Big] \dx[t].
\end{split}
\end{align}
We will bound $I_1$, $I_2$, $I_3$, and $I_4$ independently.  For $I_1$, after applying Cauchy-Schwarz, the trace inequality \eqref{eqn:trace}, the projection bound \eqref{eqn:rho-and-J-stab-proj}, \Cref{ass:eps_h_relation}, and Young's inequality, we have
\begin{align}\label{eqn:dt_J_eps:4}
\begin{split}
    \int_0^T I_1 \dx[t] &\lss \e\left(\|\rho_h^\e-\mSh\rho_h^\e\|_{L_T^2(L^2(\W_x))} + \|\rho_h^\e-\mSh\rho_h^\e\|_{L^2(H_h^1(\W_x))}\right)\|\partial_t J_h^\e\|_{L_T^2(L^2(\W_x))} \\
    &\lss \frac{\e^{3/2}}{\sqrt{h_x}}\left(\sqrt{\frac{\e}{h_x}}+1\right)\|g_{0,h}\|_{L^2(\W)}\|\partial_t J_h^\e\|_{L_T^2(L^2(\W_x))} \\ 
    &\lss \frac{1}{h_x\nu}\|g_{0,h}\|_{L^2(\W)}^2 + \nu\e^3\|\partial_t J_h^\e\|_{L_T^2(L^2(\W_x))}^2.
\end{split}
\end{align}
for all $\nu>0$.

For $I_2$, we integrate by parts in time to obtain
\begin{align}\label{eqn:dt_J_eps:5}
\begin{split}
    \int_0^T I_2 \dx[t] 
    &= \Bigg\{ \e\int_0^T \theta \left(\grad_x\partial_t \mSh\rho_h^\e,J_h^\e \right)_{\W_x} - \theta\left<\ljmp\partial_t \mSh\rho_h^\e\rjmp,\lavg J_h^\e\ravg\right>_{\EIx} \dx[t]\Bigg\} \\
    &\quad+ \Bigg\{-\e\theta \left (\grad_x\mSh\rho_h^\e,J_h^\e \right)_{\W_x}\Bigg|_0^T + \e\theta\left<\ljmp\mSh\rho_h^\e\rjmp,\lavg J_h^\e\ravg\right>_{\EIx}\Bigg|_0^T \Bigg\} \\
    &= \{K_1\} + \{K_2\}.
\end{split}
\end{align}
We first bound $K_1$.  Using Cauchy-Schwarz, the inverse inequality \eqref{eqn:inverse}, and the trace inequality \eqref{eqn:trace}, as well as the bound on $\partial_t \mSh\rho_h^\e$ in \eqref{eqn:dt_rho_proj_bound} and the bound on $J_h^\e$ in \eqref{eqn:rho-and-J-stab-J}, we have
\begin{align}\label{eqn:dt_J_eps:6}
\begin{split}
    K_1 \lss \frac{\e}{h_x}\|\partial_t \mSh\rho_h^\e\|_{L_T^2(L^2(\W_x))}\|J_h^\e\|_{L_T^2(L^2(\W_x))} \lss \frac{\e}{h_x^2}\|g_{0,h}\|_{L^2(\W)}^2.
\end{split}
\end{align}
For $K_2$, terms evaluated at $t=0$ vanish due to \Cref{ass:bc}.  Following a similar treatment as for $K_1$, but instead using the $L_T^\infty$ estimates \eqref{eqn:rho-and-J-stab-J} and \eqref{eqn:rho-and-J-stab-g}, we obtain
\begin{align}\label{eqn:dt_J_eps:7}
\begin{split}
    K_2 \lss \frac{\e}{h_x}\|\mSh\rho_h^\e\|_{L^\infty(L^2(\W_x))}\|J_h^\e\|_{L^\infty(L^2(\W_x))} \lss \frac{\e}{h_x^2}\|g_{0,h}\|_{L^2(\W)}^2.
\end{split}
\end{align}

For $I_3$, the treatment is similar to that of $I_2$. Integrating by parts in time and applying bounds similar to those used for $K_1$ and $K_2$, we find
\begin{align}\label{eqn:dt_J_eps:8}
\begin{split}
    \int_0^T I_3 \dx[t] &=  -\e\int_0^T \left(\partial_t (E\mSh\rho_h^\e ),J_h^\e \right)_{\W_x} \dx[t] + \e \left( E\mSh\rho_h^\e,J_h^\e \right)_{\W_x}\Bigg|_0^T \\
    &= -\e\int_0^T \left[ \left(\mSh\rho_h^\e\partial_t E,J_h^\e\right)_{\W_x} 
        + \left(E\partial_t \mSh\rho_h^\e,J_h^\e \right)_{\W_x} \right] \dx[t] 
        + \e \left(E\mSh\rho_h^\e,J_h^\e \right)_{\W_x}\Bigg|_0^T \\
    &\lss \e \|\rho_h^\e\|_{L_T^2(L^2(\W_x))}\|J_h^\e\|_{L_T^2(L^2(\W_x))} + \e\|\partial_t \mSh\rho_h^\e\|_{L_T^2(L^2(\W_x))}\|J_h^\e\|_{L_T^2(L^2(\W_x))} \\
    & \qquad+ \e\|\rho_h^\e\|_{L^\infty(L^2(\W_x))}\|J_h^\e\|_{L^\infty(L^2(\W_x))} \\
    &\lss \frac{\e}{h_x}\|g_{0,h}\|_{L^2(\W)}^2.
\end{split}
\end{align}
We now focus on each term of $I_4$.  To bound $\Theta_3$, we use \eqref{eqn:theta3_bnd}, \eqref{eqn:rho-and-J-stab-mp-g}, and Young's inequality:
\begin{align}\label{eqn:dt_J_eps:9}
\begin{split}
    \int_0^T \e^2\Theta_3(\tilde{g}_h^\e,\partial_t J_h^\e) \dx[t] &\lss \frac{\sqrt{\e}}{h_x}\left(\frac{1}{\e}\|g_h^\e-\mhalfh\rho_h^\e\|_{L_T^2(L^2(\W))}\right)\left(\e^{3/2}\|\partial_t J_h^\e\|_{L_T^2(L^2(\W_x))}\right) \\
    &\lss \frac{\e}{\nu h_x^2}\|g_{0,h}\|_{L^2(\W)}^2 + \nu\e^3\|\partial_t J_h^\e\|_{L_T^2(L^2(\W_x))}^2
\end{split}
\end{align}
for any $\nu>0$.  We treat $\Theta_4$ in a similar manner.  Using \eqref{eqn:theta4_bnd} and \eqref{eqn:rho-and-J-stab-jmps} with \Cref{ass:eps_h_relation}, we have
\begin{align}\label{eqn:dt_J_eps:10}
\begin{split}
    \int_0^T \e^{(\beta+3)/2}\Theta_4(\rho_h^\e,\partial_t J_h^\e) \dx[t] &\lss \frac{\e^{\beta/2}}{\sqrt{h_x}}\left(\frac{1}{\e^{(1-\beta)/2}}\|\ljmp\rho_h^\e\rjmp\|_{L_T^2(L^2(\EIx))}\right)\left(\e^{3/2}\|\partial_t J_h^\e\|_{L_T^2(L^2(\W_x))}\right) \\
    &\lss \frac{\e^{\beta}}{\nu h_x}\left(\frac{1}{\e^{(1-\beta)}}\|\ljmp\rho_h^\e\rjmp\|_{L_T^2(L^2(\EIx))}^2\right) + \nu\e^3\| \partial_tJ_h^\e\|_{L_T^2(L^2(\W_x))}^2 \\
    &\lss \frac{\e^\beta}{\nu h_x}\|g_{0,h}\|_{L^2(\W)}^2 + \nu\e^3\|\partial_t J_h^\e\|_{L_T^2(L^2(\W_x))}^2
\end{split}
\end{align}
for any $\nu>0$.   We treat $\Theta_5$ similar to the $\beta=0$ case of $\Theta_4$; cf. \eqref{eqn:dt_J_eps:10}:
\begin{align}\label{eqn:dt_J_eps:11}
\begin{split}
    \int_0^T \e^{3/2}\Theta_5(\rho_h^\e,\partial_t J_h^\e) \dx[t] 
    &\lss \frac{1}{\nu h_x}\|g_{0,h}\|_{L^2(\W)}^2 + \nu\e^3\|\partial_t J_h^\e\|_{L_T^2(L^2(\W_x))}^2
\end{split}
\end{align}
for any $\nu>0$.  Combining \eqref{eqn:dt_J_eps:9}-\eqref{eqn:dt_J_eps:11} yields the following bound for $I_4$:
\begin{align}\label{eqn:dt_J_eps:12}
\begin{split}
    \int_0^T I_4 \dx[t] 
    &\lss \frac{1}{\nu}\left(\frac{1}{h_x}+\frac{\e}{h_x^2}\right)\|g_{0,h}\|_{L^2(\W)}^2 + \nu\e^3\|\partial_t J_h^\e\|_{L_T^2(L^2(\W_x))}^2.
\end{split}
\end{align}
for any $\nu>0$.  Combining \eqref{eqn:dt_J_eps:3}-\eqref{eqn:dt_J_eps:7} and \eqref{eqn:dt_J_eps:12} we obtain

\begin{align}\label{eqn:dt_J_eps:13}
\begin{split}
   \e^3\left\|\partial_t  J_h^\e\right\|_{L_T^2(L^2(\W_x))}^2 &\lss 
   \frac{1}{\nu}\left(\frac{1}{h_x}+\frac{\e}{h_x^2}\right)\|g_{0,h}\|_{L^2(\W)}^2 
   + \frac{\e}{h_x^2}\|g_{0,h}\|_{L^2(\W)}^2 + \nu\e^3\|\partial_t J_h^\e\|_{L_T^2(L^2(\W_x))}^2.
\end{split}
\end{align}
Choosing $\nu$, independent of $\e$ and $h_x$, sufficiently small to move $\e^3\|\partial_t J_h^\e\|_{L_T^2(L^2(\W_x))}^2$ from the right-hand side of \eqref{eqn:dt_J_eps:13} and applying \Cref{ass:eps_h_relation} to the first two terms on the right-hand side of \eqref{eqn:dt_J_eps:13} gives us \eqref{eqn:dt_J_eps}.  The proof is complete. 
\end{proof}
\section{The Drift Diffusion Limit}\label{sect:dde-limit}

The bounds in \Cref{sect:a_priori_estimates} allow us to take the limit of $\rho_h^\e$ and $J_h^\e$ as $\e\to 0$.  In this section, we show these limits satisfy \eqref{eqn:dde_system-b1}, a discrete version of the drift-diffusion equations \eqref{eqn:dde-cont-system}.  Recall the definition of $\e_{h_v}$ from \eqref{eqn:eps_h}.

\begin{theorem}\label{thm:dde_limit}
Let $h_x,h_v>0$ be fixed.  Then for all $\e\leq\e_{h_v}$, we have $\rho_h^\e\in L_T^2(L^2(\W_x))$, $\frac{\partial}{\partial t}\rho_h^\e\in L_T^2(H_{h,\beta}^{-1}(\W_x))$ and $J_h^\e\in L_T^2(L^2(\W_x))$ with bound
\begin{align}\label{eqn:dde_limit_bound}
\|\rho_h^\e\|_{L_T^2(L^2(\W_x))} + \|\partial_t\rho_h^\e\|_{L^2(H_{h,\beta}^{-1}(\W_x))} + \|J_h^\e\|_{L_T^2(L^2(\W_x))} \lss \|g_{0,h}\|_{L^2(\W)}. 
\end{align}
Moreover, there exist functions $\rho_h^0 \in H^1([0,T];V_{x,h}^0)$, $J_h^0\in L^2([0,T];[V_{x,h}]^3)$, and subsequences of $\{\rho_h^\e\}_\e$ and $\{J_h^\e\}_\e$, not relabeled, such that $\rho_h^\e\wto\rho_h^0$ in $L_T^2(L^2(\W_x))$  and $J_h^\e\wto J_h^0$ in $L_T^2([L^2(\W_x)]^3)$.  If $\beta=0$, then we have $\rho_h^0\in H^1([0,T];S_{x,h}^0)$ and $\partial_t\rho_h^\e\wto \partial_t\rho_h^0$ in $L_T^2((S_{x,h}^0)^*)$.  If $\beta=1$, then $\partial_t\rho_h^\e\wto \partial_t\rho_h^0$ in $L_T^2((V_{x,h}^0)^*)$.

Additionally, $\rho_h^0$ and $J_h^0$ satisfy the following drift-diffusion system:
\begin{subequations}
\label{eqn:dde_system-b1}
\begin{align}
    \left(\frac{\partial}{\partial t}\rho_h^0,q_h\right)_{\W_x} - (J_h^0,\grad_x q_h)_{\W_x} + \left<\lavg J_h^0\ravg,\ljmp q_h\rjmp\right>_{\EIx} +  \left<\gamma_I \ljmp\rho_h^0\rjmp,\ljmp q_h\rjmp\right>_{\EIx} &= 0, \label{eqn:dde_system-b1-1}\\
    (\omega J_h^0,\tau_h)_{\W_x} + \theta(\grad_x\rho_h^0,\tau_h)_{\W_x} - \theta\left<\ljmp\rho_h^0\rjmp,\lavg \tau_h\ravg\right>_{\EIx}
    - (E\rho_h^0,\tau_h)_{\W_x} &= 0, \label{eqn:dde_system-b1-2}\\
    (\rho_h^0(0),q_h)_{\W_x} = (\rho_{0,h},q_h)_{\W_x}&. \label{eqn:dde_system-b1-3}
\end{align}
\end{subequations}
for all $\tau_h\in[V_{x,h}]^3$ and $q_h\in V_{x,h}^0$ if $\beta=1$ and for all $\tau_h\in[V_{x,h}]^3$ and $q_h\in S_{x,h}^0$ if $\beta=0$.
\end{theorem}

\begin{remark}
When $\beta=0$ all of the interior edge terms in \eqref{eqn:dde_system-b1} vanish due to the continuity of the $\rho_h^\e$ and $q_h$.
\end{remark}

\begin{proof}
The proof proceeds in two steps.

\noindent\textit{\underline{Step 1}: Existence of Limits.} 
It follows from \eqref{eqn:rho-and-J-stab-g}, \eqref{eqn:rho-and-J-stab-J}, and \eqref{eqn:dt_rho_bound} that $\|\rho_h^\e\|_{L_T^2(L^2(\W_x))}$, $\|J_h^\e\|_{L_T^2(L^2(\W_x))}$, and $\|\partial_t\rho_h^\e\|_{L^2_T(H_{h,\beta}^{-1}(\W_x))}$, respectively, \eqref{eqn:dde_limit_bound} holds and each term is uniformly bounded in $\e$.  Since each of these spaces are Hilbert spaces, we can extract a subsequence $\rho_h^\e$ and $J_h^\e$, not relabeled, and limiting functions $\rho_h^0\in L_T^2(V_{x,h})$ and $J_h^0\in L_T^2([V_{x,h}]^3)$ such that $\rho_h^\e\wto\rho_h^0$ in $L_T^2(L_2(\W_x))$ and $J_h^\e\wto J_h^0$ in $L_T^2([L^2(\W_x)]^3)$.  We now show $\rho_h^0\in L_T^2(V_{x,h}^0)$.  By \eqref{eqn:rho-and-J-stab-jmps}, 
\begin{align}\label{eqn:dde_limit_bound:1}
    \|\rho_h^\e\|_{L_T^2(L^2(\partial\W_x))}^2\lss \e\|g_{0,h}\|_{L^2(\W)}^2.
\end{align}
Since $V_h$ is finite dimensional, then $\rho_h^\e\wto \rho_h^\e$ in $L_T^2(L^2(\partial\W_x))$.  Because the norm is weakly lower semi-continuous, \eqref{eqn:dde_limit_bound:1} implies that
\begin{align*}
    \|\rho_h^0\|_{L_T^2(L^2(\partial\W_x))}^2 =0 .
\end{align*}
Therefore $\rho_h^0\in L_T^2(V_{x,h}^0)$.  If $\beta=0$, then \eqref{eqn:rho-and-J-stab-jmps}  additionally gives us
\begin{align}\label{eqn:dde_limit_bound:2}
    \|\ljmp\rho_h^\e\rjmp\|_{L_T^2(L^2(\EIx))}^2\lss \e\|g_{0,h}\|_{L^2(\W)}^2,
\end{align}
and thus passing the limit as $\e\to 0$ in \eqref{eqn:dde_limit_bound:2} we obtains
\begin{align*}
    \|\ljmp\rho_h^0\rjmp\|_{L_T^2(L^2(\EIx))}^2 = 0.
\end{align*}
Hence $\rho_h^0(t)$ is continuous in $x$ for a.e.\ time $t$ and $\rho_h^0\in L_T^2(S_{x,h}^0)$ if $\beta=0$.  

For this paragraph we will consider the case $\beta=0$ and put the respective $\beta=1$ result in parentheses.  By \eqref{eqn:dt_rho_bound}, $\partial_t\rho_h^0$ is uniformly bounded in $\e$ in $L_T^2((S_{x,h}^0)^*)$ $(\text{resp }L_T^2((V_{x,h}^0)^*))$  where $(S_{x,h}^0)^*$ $(\text{resp } (V_{x,h}^0)^*)$ is the dual space of $S_{x,h}^0$ $(\text{resp } V_{x,h}^0)$.  Hence there is a subsequence of $\partial_t\rho_h^\e$ such that $\partial_t\rho_h^\e\wto \zeta$ for some $\zeta\in L_T^2((S_{x,h}^0)^*)$ $(\text{resp } L_T^2((V_{x,h}^0)^*))$.  Since $S_{x,h}^0$ (resp. $V_{x,h}^0$) is finite-dimensional and thus a Hilbert space with respect to the $L^2$ inner product on $\W_x$, we can apply the Riesz representation theorem to show there is $\zeta_h\in L_T^2(S_{x,h}^0)$ $(\text{resp } L_T^2(V_{x,h}^0))$ such that
 \[
\zeta(t;q_h) =  (\zeta_h(t),q_h(t))_{\W_x}
 \]
 for all $q_h\in S_{x,h}^0$ $(\text{resp } V_{x,h}^0)$ and a.e.\ $t$ where $\zeta(t;\cdot)\in (S_{x,h}^0)^*$ $(\text{resp }(V_{x,h}^0)^*)$.  By a standard density argument (see \cite[Chapter 7, Problem 5]{Eva2010}), we have $\zeta_h=\partial_t\rho_h^0$.  Therefore $\rho_h^0\in H^1([0,T];S_{x,h}^0)\hookrightarrow C^0([0,T];S_{x,h}^0)$ $(\text{resp } H^1([0,T];V_{x,h}^0)\hookrightarrow C^0([0,T];V_{x,h}^0) )$.

\noindent\textit{\underline{Step 2}: The Limiting System}. We first recover \eqref{eqn:dde_system-b1-1}.  We choose $q_h\in L_T^2(S_{x,h}^0)$ if $\beta=0$ and  $q_h\in L_T^2(V_{x,h}^0)$ if $\beta=1$ in \eqref{eqn:rho_eps_dde} and integrate in time to obtain
\begin{align}\label{eqn:dde_sys:1}
\begin{split}
    \int_0^T \left[
    \left(\partial_t \rho_h^\e,q_h\right)_{\W_x} - (J_h^\e,\grad_x q_h)_{\W_x} + \left<\lavg J_h^\e\ravg,\ljmp q_h\rjmp\right>_{\EIx} +  \left<\gamma_I \ljmp\rho_h^\e\rjmp,\ljmp q_h\rjmp\right>_{\EIx} \right]\dx[t]
    = \int_0^T \e^\b\Theta_1(\tilde{g}_h^\e,q_h) \dx[t]. 
\end{split}
\end{align}
Note that $\int_0^T\Theta_2(\tilde{g}_h^\e,q_h)\dx[t] = 0$ since for a.e.\ $t$, $q_h(t)=0$ on $\partial\W_x$.  If $\beta=0$, then $\int_0^T\Theta_1(\tilde{g}_h^\e,q_h)\dx[t]=0$ since $q_h(t)$ is continuous in $x$ for a.e.\ $t$; additionally, the interior penalty term is also zero so we drop its $\e^{\beta-1}$ multiplier found in \eqref{eqn:rho_eps_dde}.  If $\beta=1$, then $\int_0^T\Theta_1(\tilde{g}_h^\e,q_h)\dx[t]$ is uniformly bounded with respect to $\e$ by \eqref{eqn:theta1_bnd} and \eqref{eqn:rho-and-J-stab-mp-g}; thus the right hand side will vanish as $\e\to 0$.  Therefore we can pass the weak limit as $\e\to 0$ in \eqref{eqn:dde_sys:1} and obtain 
\[
\int_0^T \left(\frac{\partial}{\partial t}\rho_h^0,q_h\right)_{\W_x} - (J_h^0,\grad_x q_h)_{\W_x} + \left<\lavg J_h^0\ravg,\ljmp q_h\rjmp\right>_{\EIx} +  \left<\gamma_I \ljmp\rho_h^0\rjmp,\ljmp q_h\rjmp\right>_{\EIx} \dx[t]= 0, 
\]
which implies 
\[
\left(\frac{\partial}{\partial t}\rho_h^0,q_h\right)_{\W_x} - (J_h^0,\grad_x q_h)_{\W_x} + \left<\lavg J_h^0\ravg,\ljmp q_h\rjmp\right>_{\EIx} +  \left<\gamma_I \ljmp\rho_h^0\rjmp,\ljmp q_h\rjmp\right>_{\EIx} = 0.
\]
for all $q_h\in V_{x,h}^0$ if $\beta=1$ and all $q_h\in S_{x,h}^0$ if $\beta=0$ and a.e. $0<t\leq T$.  Thus we arrive at \eqref{eqn:dde_system-b1-1}. 

For \eqref{eqn:dde_system-b1-2}, choose $\tau_h\in L_T^2([V_{x,h}]^3)$ in \eqref{eqn:J_eps_dde} and integrate in time to obtain
\begin{align}\label{eqn:dde_sys:2}
\begin{split}
      \int_0^T (\omega J_h^\e,\tau_h)_{\W_x} &+ \theta(\grad_x\rho_h^\e,\tau_h)_{\W_x} - \theta\left<\ljmp\rho_h^\e\rjmp,\lavg \tau_h\ravg\right>_{\EIx}  - (E\rho_h^\e,\tau_h)_{\W_x} \dx[t]\\
    & = \int_0^T\e\Theta_3(\tilde{g}_h^\e,\tau_h) + \sqrt{\e^{\beta+1}}\Theta_4(\rho_h^\e,\tau_h) \\
    &\qquad+ \sqrt{\e}\Theta_5(\rho_h^\e,\tau_h) - \e^2\left(\partial_t  J_h^\e,\tau_h\right)_{\W_x} \dx[t].
\end{split}
\end{align}
Since $\e^{3/2}\|\partial_t J_h^\e\|_{L_T^2(L^2(\W_x))}$ is bounded in $\e$ by \eqref{eqn:dt_J_eps}, the time derivative term in \eqref{eqn:dde_sys:2} vanishes as $\e\to 0$.  Additionally since $\int_0^T\Theta_3(\tilde{g}_h^\e,\tau_h)\dx[t]$, $\int_0^T\Theta_4(\rho_h^\e,\tau_h)\dx[t]$, and $\int_0^T\Theta_5(\rho_h^\e,\tau_h)\dx[t]$ are bounded in $\e$ by \eqref{eqn:theta3_bnd}, \eqref{eqn:theta4_bnd}, and \eqref{eqn:theta5_bnd} respectively, then the entire right hand side of \eqref{eqn:dde_sys:2} vanishes as $\e$ vanishes.  Therefore the limiting equation of \eqref{eqn:dde_sys:2} as $\e\to 0$ is
\[
\int_0^T (\omega J_h^0,\tau_h)_{\W_x} + \theta(\grad_x\rho_h^0,\tau_h)_{\W_x} - \theta\left<\ljmp\rho_h^0\rjmp,\lavg \tau_h\ravg\right>_{\EIx}
    - (E\rho_h^0,\tau_h)_{\W_x} \dx[t]= 0 ,
\]
which implies
\[
 (\omega J_h^0,\tau_h)_{\W_x} + \theta(\grad_x\rho_h^0,\tau_h)_{\W_x} - \theta\left<\ljmp\rho_h^0\rjmp,\lavg \tau_h\ravg\right>_{\EIx}
    - (E\rho_h^0,\tau_h)_{\W_x} = 0
\]
for all $\tau_h\in [V_{x,h}]^3$ and a.e. $0<t\leq T$. Therefore we recover \eqref{eqn:dde_system-b1-2}. 

We now derive the projected initial condition \eqref{eqn:dde_system-b1-3} which is similarly shown in \cite[Page 379]{Eva2010} but is listed here for completeness.  Let $\beta=1$ and choose $q_h\in H^1([0,T],V_{x,h}^0)$ with $q_h(T)=0$ in \eqref{eqn:dde_sys:1}.  Integration by parts on the first term of \eqref{eqn:dde_sys:1} recovers
\begin{align}\label{eqn:dde_sys:6}
\begin{split}
    -\int_0^T\left(\rho_h^\e,\frac{\partial}{\partial t}q_h\right)_{\W_x} &- (J_h^\e,\grad_x q_h)_{\W_x} + \left<\lavg J_h^\e\ravg,\ljmp q_h\rjmp\right>_{\EIx} \\
    &\quad+  \left<\gamma_I \ljmp\rho_h^\e\rjmp,\ljmp q_h\rjmp\right>_{\EIx} \dx[t]= \int_0^T \e\Theta_1(g_h^\e,q_h) \dx[t] + (\rho_{0,h},q_h(0))_{\W_x}. 
\end{split}
\end{align}
Sending $\e\to 0$ in \eqref{eqn:dde_sys:6} yields
\begin{align}\label{eqn:dde_sys:7}
\begin{split}
    -\int_0^T\left(\rho_h^0,\frac{\partial}{\partial t}q_h\right)_{\W_x} &- (J_h^0,\grad_x q_h)_{\W_x} + \left<\lavg J_h^0\ravg,\ljmp q_h\rjmp\right>_{\EIx} + \left<\gamma_I \ljmp\rho_h^0\rjmp,\ljmp q_h\rjmp\right>_{\EIx} \dx[t] =  (\rho_{0,h},q_h(0))_{\W_x}. 
\end{split}
\end{align}
Choosing the same test function in \eqref{eqn:dde_system-b1-1} and integrating by parts in time also yields
\begin{align}\label{eqn:dde_sys:8}
\begin{split}
    -\int_0^T\left(\rho_h^0,\frac{\partial}{\partial t}q_h\right)_{\W_x} &- (J_h^0,\grad_x q_h)_{\W_x} + \left<\lavg J_h^0\ravg,\ljmp q_h\rjmp\right>_{\EIx} + \left<\gamma_I \ljmp\rho_h^0\rjmp,\ljmp q_h\rjmp\right>_{\EIx} \dx[t] =  (\rho_h^0(0),q_h(0))_{\W_x}. 
\end{split}
\end{align}
Subtracting \eqref{eqn:dde_sys:7} from \eqref{eqn:dde_sys:8} implies \eqref{eqn:dde_system-b1-3}.  The case for $\beta=0$ is similar.  The proof is complete. 
\end{proof}

Now we show that the whole sequence $\{\rho_h^\e\}_\e$ and $\{J_h^\e\}_\e$ must converge to $\rho_h^0$ and $J_h^0$ respectively.  We do this by showing uniqueness of solutions to the drift-diffusion system \eqref{eqn:dde_system-b1} with the following lemma.

\begin{lemma}\label{lem:r_J_0_bound}
Suppose $\rho_h$ and $J_h$ satisfy \eqref{eqn:dde_system-b1}, then we have the following bound for any $h_x>0$:
\begin{align}\label{eqn:r_J_0_bound}
    \|\rho_h\|_{L_T^\infty(L^2(\W_x))}^2 + \frac{\omega_{\min}}{\theta}\|J_h\|_{L_T^2(L^2(\W_x))}^2 &\leq \exp\left(\frac{\|E\|_{ L^\infty([0,T]\times\W_x)}}{\theta\omega_{\min}}T\right)\|\rho_{0,h}\|_{L^2(\W_x)}^2 
\end{align}

Moreover, the solution pair $\{\rho_h,J_h\}$ to \eqref{eqn:dde_system-b1} is unique.  
\end{lemma}
        
\begin{proof}
We first focus on \eqref{eqn:r_J_0_bound}.  Choose $q_h=\theta\rho_h$ and $\tau_h=J_h$ in \eqref{eqn:dde_system-b1}.  Adding both equations in \eqref{eqn:dde_system-b1} gives us
\begin{align}\label{eqn:r_J_0_bound:1}
    \frac{\theta}{2}\frac{\dx[]}{\dx[t]}\|\rho_h\|_{L^2(\W_x)}^2 + (\omega J_h,J_h)_{\W_x} + \theta\left<\gamma_I \ljmp\rho_h\rjmp,\ljmp \rho_h\rjmp\right>_{\EIx} = (E\rho_h,J_h)_{\W_x}
\end{align}
for each system.  We note if $\beta=0$, then the interior penalty term is zero since $\rho_h$ is continuous.  Thus \eqref{eqn:r_J_0_bound:1} is true independent of $\beta$.  
Dropping the interior penalty term, bounding the right hand side of \eqref{eqn:r_J_0_bound:1} by H\"older's and Young's inequality, and dividing by $\theta/2$ we arrive at
\begin{align}\label{eqn:r_J_0_bound:2}
    \frac{\dx[]}{\dx[t]}\|\rho_h\|_{L^2(\W_x)}^2 + \frac{\omega_{\min}}{\theta}\| J_h\|_{L^2(\W_x)}^2 \leq \frac{\|E\|_{L^\infty([0,T]\times\W_x)}}{\theta\omega_{\min}}\|\rho_h\|_{L^2(\W_x)}^2
\end{align}
Applying Gr\"onwall's to \eqref{eqn:r_J_0_bound:2} gives us \eqref{eqn:r_J_0_bound}.  

Uniqueness of the solution pair follows from applying \eqref{eqn:r_J_0_bound} to the case where $\rho_{0,h}=0$.  The proof is complete.
\end{proof}

Due to the uniqueness result from Lemma \ref{lem:r_J_0_bound}, we attach the additional corollary.

\begin{corollary}
The full sequences $\{\rho_h^\e\}_\e$ and $\{J_h^\e\}_{\e}$ weakly converge to $\rho_h^0$ and $J_h^0$ respectively in the topologies given in Theorem \ref{thm:dde_limit}. 
\end{corollary}

\section{Error Estimates}\label{sect:error_est}

In this section we develop error estimates for $\rho_h^\e$ against the true drift-diffusion limit $\rho^0$ which solves \eqref{eqn:dde-cont}.  The error estimates are created by comparing both against the discrete drift-diffusion $\rho_h^0$ which solves \eqref{eqn:dde_system-b1}.  This is summarized in the following theorem whose proof we delay until the end of the section.  

\begin{theorem}\label{thm:error_est}
Suppose $\rho^0\in L_T^2(H^s(\W))$ and $J^0\in L_T^2([L^2(\W_x)]^3)$ satisfy \eqref{eqn:dde-cont-system} for some $s\geq 2$, $\omega\in W^{r,\infty}(\W_x)$, and $E\in L_T^\infty(W^{r,\infty}(\W_x))$ for some $r\geq s-1$.  Define
\begin{equation}
\Cwr = \|\omega\|_{W^{r,\infty}(\W_x)}\|E\|_{L_T^\infty(W^{r,\infty}(\W_x))}.
\end{equation}
Then for any $\e\leq\e_{h_v}$ where $\e_{h_v}$ is defined in \eqref{eqn:eps_h} we have the following error estimate:
\begin{align}\label{eqn:err-eps-h}
\begin{split}
    \|\rho_h^\e-\rho^0\|_{L_T^2(L^2(\W_x))} &\lss \sqrt{\frac{\e}{h_x}}\|g_{0,h}\|_{L^2(\W)} + h_x^{\min\{k+1,s\}-1}\|\rho^0\|_{L_T^2(H^s(\W_x))} \\
   &\quad+ \Cwr h_x^{\min\{k+1,s-1\}-\beta/2}\|\rho^0\|_{L_T^2(H^s(\W_x))}.
\end{split}
\end{align}
\end{theorem}

\subsection{Error Estimates in \texorpdfstring{$\e$}{epsilon}}

Here we build estimates comparing $\rho_h^\e$ to $\rho_{h}^0$.  Define $e_\rho^\e=\rho_h^\e-\rho_h^0$ and $e_J^\e=J_h^\e-J_h^0$.  Subtracting \eqref{eqn:dde_system-b1} from the system \eqref{eqn:rho_eps_dde} and \eqref{eqn:J_eps_dde} gives us the following error equations:
\begin{align}\label{eqn:error_eqns}
\begin{split}
    \left(\partial_t e_\rho^\e,q_h\right)_{\W_x} &- (e_J^\e,\grad_x q_h)_{\W_x} + \left<\lavg e_J^\e\ravg,\ljmp q_h\rjmp\right>_{\EIx} +  \e^{\b-1}\left<\gamma_I \ljmp e_\rho^\e\rjmp,\ljmp q_h\rjmp\right>_{\EIx} \\
    & = \e^{\b}\Theta_1(\tilde{g}_h^\e,q_h) \\
    (\omega e_J^\e,\tau_h)_{\W_x} &+ \theta(\grad_x e_\rho^\e,\tau_h)_{\W_x} - \theta\left<\ljmp e_\rho^\e\rjmp,\lavg \tau_h\ravg\right>_{\EIx}
    - (E e_\rho^\e,\tau_h)_{\W_x} \\ 
    &= \e\Theta_3(\tilde{g}_h^\e,\tau_h) + \sqrt{\e^{\b+1}}\Theta_4(\rho_h^\e,\tau_h) + \sqrt{\e}\Theta_5(\rho_h^\e,\tau_h) +\sqrt{\e}\left(\e^{3/2}\partial_t J_h^\e,\tau_h\right)_{\W_x}
\end{split}
\end{align}
for all $\tau_h\in[V_{x,h}]^3$ and $q_h\in V_{x,h}^0$ if $\b=1$ and $q_h\in S_{x,h}^0$ if $\b=0$.

In order to bound the error of $e_\rho^\e$, we decompose the $\rho$ error as $e_\rho^\e=\eta_\rho^\e - \xi_\rho^\e:=(\rho_h^\e-\mSh{\rho_h^\e}) - (\rho_h^0-\mSh{\rho_h^\e})$. Thus $\eta_\rho^\e\in V_{x,h}$ and $\xi_\rho^\e\in V_{x,h}^0$ if $\b=1$ and $\xi_\rho^\e\in S_{x,h}^0$ if $\b=0$.

\begin{lemma}\label{lem:err_bnd}
For any $h_x>0$, $h_v>0$, and $\e\leq\e_{h_v}$ where $\e_{h_v}$ is defined in Lemma \ref{lem:g-l2-stab}, $e_\rho^\e$ and $e_J^\e$ satisfy the following error bound:
\begin{align}\label{eqn:err_bnd}
\begin{split}
    \|\xi_\rho^\e(T)\|_{L^2(\W)}^2 +& \frac{\omega_{\min}}{2}\|e_J^\e\|_{L_T^2(L^2(\W_x))}^2 + \frac{\theta\gamma_*}{2}\|\ljmp \xi_\rho^\e\rjmp\|_{L_T^2(L^2(\EIx))}^2 \\
    &\lss \frac{\e}{h_x}\|g_{0,h}\|_{L^2(\W)}^2.
\end{split}
\end{align}
\end{lemma}

\begin{proof}
Choose $q_h=-\theta\xi_\rho^\e$ and $\tau_h = e_J^\e$ in \eqref{eqn:error_eqns}.  Adding the two equations in \eqref{eqn:error_eqns} we arrive at
\begin{align}\label{eqn:error_bnd:1}
\begin{split}
    \theta\left(\partial_t \xi_\rho^\e,\xi_\rho^\e\right)_{\W_x} &+ (\omega e_J^\e,e_J^\e)_{\W_x} +  \theta\left<\gamma_I\ljmp\xi_\rho^\e\rjmp,\ljmp\xi_\rho^\e\rjmp\right>_{\EIx} \\
    &= -\theta(\grad_x\eta_\rho^\e,e_J^\e)_{\W_x} + \theta\left<\lavg e_J^\e\ravg,\ljmp\eta_\rho^\e\rjmp\right>_{\EIx} + (E\eta_\rho^\e,e_J^\e)_{\W_x} \\
    &\quad- (E\xi_\rho^\e,e_J^\e)_{\W_x} - \e\theta\Theta_1(\tilde{g}_h^\e,\xi_\rho^\e) + \e\Theta_3(\tilde{g}_h^\e,e_J^\e) + \sqrt{\e^{\b+1}}\Theta_4(\rho_h^\e,e_J^\e) \\
    &\quad+ \sqrt{\e}\Theta_5(\rho_h^\e,e_J^\e) + \sqrt{\e}\left(\e^{3/2}\partial_t J_h^\e,e_J^\e\right)_{\W_x} + \left(\partial_t \eta_\rho^\e,\xi_\rho^\e\right)_{\W_x}.
\end{split}
\end{align}
We seek to bound each of the terms on the right hand side of \eqref{eqn:error_bnd:1}.  Let 
\begin{equation}
    I_1 := -\theta(\grad_x\eta_\rho^\e,e_J^\e)_{\W_x} + \theta\left<\lavg e_J^\e\ravg,\ljmp\eta_\rho^\e\rjmp\right>_{\EIx} + (E\eta_\rho^\e,e_J^\e)_{\W_x}.
\end{equation}
By use of inverse inequalities, trace inequalities, and Young's inequality, we can bound $I_1$ for any $\nu>0$ by
\begin{align}\label{eqn:error_bnd:2}
    |I_1| \lss \frac{1}{\nu}\|\eta_\rho^\e\|_{L^2(\W_x)}^2 + \frac{1}{\nu}\|\eta_\rho^\e\|_{H_h^1(\W_x)}^2 + \nu\|e_J^\e\|_{L^2(\W_x)}^2.
\end{align}
We can bound $(E\xi_\rho^\e,e_J^\e)_{\W_x}$ with Young's inequality to obtain
\begin{align}\label{eqn:error_bnd:3}
    |(E\xi_\rho^\e,e_J^\e)_{\W_x}| \lss \frac{1}{\nu}\|\xi_\rho^\e\|_{L^2(\W_x)}^2 + \nu\|e_J^\e\|_{L^2(\W_x)}^2
\end{align}
for any $\nu>0$.  For $\Theta_1$, we note if $\beta=0$, then $\xi_\rho^\e$ is continuous in $x$, so $\Theta_1(\tilde{g}_h^\e,\xi_\rho^\e)=0$.  Thus we focus on the case $\beta=1$ for which we use \eqref{eqn:theta1_bnd} and Young's inequality to bound $\Theta_1$ as
\begin{align}\label{eqn:error_bnd:4}
\begin{split}
    |\e\Theta_1(\tilde{g}_h^\e,\xi_\rho^\e)| &\lss \frac{1}{\nu h_x}\|\mhalfh\rho_h^\e-g_h^\e\|_{L^2(\W)}^2 +  \nu\|\ljmp\xi_\rho^\e\rjmp\|_{L^2(\EIx)}^2 
\end{split}
\end{align}
for any $\nu>0$.  We use \eqref{eqn:theta3_bnd} and Young's inequality to obtain
\begin{align}\label{eqn:error_bnd:5}
    |\e\Theta_3(\tilde{g}_h^\e,e_J^\e)| \lss \frac{1}{\nu h_x^2}\|\mhalfh\rho_h^\e-g_h^\e\|_{L^2(\W)}^2 +  \nu\|e_J^\e\|_{L^2(\W_x)}^2
\end{align}
for any $\nu>0$.  Similarly, we can bound $\Theta_4$ and $\Theta_5$ by \eqref{eqn:theta4_bnd}, \eqref{eqn:theta5_bnd}, and Young's inequality to find
\begin{align}\label{eqn:error_bnd:6}
    |\sqrt{\e^{\b+1}}\Theta_4(\rho_h^\e,e_J^\e) + \sqrt{\e}\Theta_5(\rho_h^\e,e_J^\e)| \lss \frac{\e^{2\beta}}{\nu h_x}\|\ljmp\rho_h^\e\rjmp\|_{L^2(\EIx)}^2 + \frac{1}{\nu h_x}\|\rho_h^\e\|_{L^2(\partial\W_x)}^2 +  \nu\|e_J^\e\|_{L^2(\W_x)}^2
\end{align}
for any $\nu>0$.  We bound the $\partial_t J_h^\e$ term as
\begin{align}\label{eqn:error_bnd:7}
    \big|\sqrt{\e}\left(\e^{3/2}\partial_t J_h^\e,e_J^\e\right)_{\W_x}\big| \lss \frac{\e}{\nu}\|\e^3\partial_t J_h^\e\|_{L^2(\W_x)}^2 +  \nu\|e_J^\e\|_{L^2(\W_x)}^2
\end{align}
for any $\nu>0$.  The term $\left(\partial_t \eta_\rho^\e,\xi_\rho^\e\right)_{\W_x}$ is zero since by definition $\eta_\rho^\e$ is orthogonal to $\xi_\rho^\e$ in $L^2(\W_x)$.   Injecting \eqref{eqn:error_bnd:2} through \eqref{eqn:error_bnd:7} into \eqref{eqn:error_bnd:1} gives us
\begin{align}\label{eqn:error_bnd:8}
\begin{split}
    \frac{1}{2}\frac{\dx[]}{\dx[t]}\|\xi_\rho^\e\|_{L^2(\W_x)}^2 &+ \omega_{\min}\|e_J^\e\|_{L^2(\W)}^2 + \frac{\theta\gamma_*}{2}\|\ljmp\xi_\rho^\e\rjmp\|_{L^2(\EIx)}^2 \\
    &\lss \frac{1}{\nu}\|\xi_\rho^\e\|_{L^2(\W_x)}^2 + \frac{1}{\nu}\|\eta_\rho^\e\|_{L^2(\W_x)}^2 + \tfrac{1}{\nu}\|\eta_\rho^\e\|_{H_h^1(\W_x)}^2 \\
    &\quad+\frac{1}{\nu h_x}\left(1 + \frac{1}{h_x}\right)\|\mhalfh\rho_h^\e-g_h^\e\|_{L^2(\W)}^2 + \frac{\e}{\nu}\|\e^3\partial_t J_h^\e\|_{L^2(\W_x)}^2 \\
    &\quad + \frac{\e^{2\beta}}{\nu h_x}\|\ljmp\rho_h^\e\rjmp\|_{L^2(\EIx)}^2 + \frac{1}{\nu h_x}\|\rho_h^\e\|_{L^2(\partial\W_x)}^2 \\
    &\quad + \nu\|e_J^\e\|_{L^2(\W)}^2 + \nu \|\ljmp\xi_\rho^\e\rjmp\|_{L^2(\EIx)}^2
\end{split}
\end{align}
for all $\nu>0$. Choosing $\nu$, independent of $\e$ and $h_x$, sufficiently small we can move the last two terms on the right of \eqref{eqn:error_bnd:8} over to the left. Since $\xi_\rho^\e(0) = 0$, we can then apply Gr\"onwall's to obtain (assuming $h_x \lss 1$) 
\begin{align}
\begin{split}\label{eqn:error_bnd:9}
    \|\xi_{\rho}^\e\|_{L_T^\infty(L^2(\W_x))}^2 &+ \|e_J^\e\|_{L_T^2(L^2(\W_x))}^2 + \|\ljmp\xi_\rho^\e\rjmp\|_{L_T^2(L^2(\EIx))}^2 \\
    &\lss \|\eta_\rho^\e\|_{L_T^2(L^2(\W_x))}^2  +\|\eta_\rho^\e\|_{L_T^2(H_h^1(\W_x))}^2 \\
    &\quad+ \frac{1}{\nu h_x^2}\|\mhalfh\rho_h^\e-g_h^\e\|_{L_T^2(L^2(\W))}^2 + \e\|\e^3\partial_t J_h^\e\|_{L_T^2(L^2(\W_x))}^2 \\ 
    &\quad + \frac{\e^{2\beta}}{ h_x}\|\ljmp\rho_h^\e\rjmp\|_{L_T^2(L^2(\EIx))}^2 + \frac{1}{ h_x}\|\rho_h^\e\|_{L_T^2(L^2(\partial\W_x))}^2.
\end{split}
\end{align}
Note that the appropriate norms of $\eta_\rho^\e=\rho_h^\e-\mSh\rho_h^\e$, $\rho_h^\e$, and $\mhalfh\rho_h^\e-g_h^\e$ can all be bounded from \Cref{lem:rho-and-J-stab} while $\frac{\partial}{\partial t}J_h^\e$ can be bounded in $L^2$ by \Cref{lem:dt_J_eps}.  Applying these bounds to \eqref{eqn:error_bnd:9}, recalling \Cref{ass:eps_h_relation}, and noticing $\e^2\lss \e$ gives us \eqref{eqn:err_bnd}.  The proof is complete.
\end{proof}

We can now show the $\e$-error estimate.

\begin{theorem}\label{thm:err-eps}
Let $\{\rho_h^0,J_h^0\}$ satisfy \eqref{eqn:dde_system-b1}.  Then for any $\e\leq\e_{h_v}$, where $\e_{h_v}$ is defined in \eqref{eqn:eps_h}, we have the following error estimate:
\begin{align}\label{eqn:eps_err_est}
    \|\rho_h^\e-\rho_h^0\|_{L_T^2(L^2(\W_x))} + \|J_h^\e-J_h^0\|_{L_T^2(L^2(\W_x))} \lss \sqrt{\frac{\e}{h_x}}\|g_{0,h}\|_{L^2(\W)}.
\end{align}
\end{theorem}

\begin{proof}
Using the triangle inequality and H\"older's inequality, we have
\begin{align*}
\|\rho_h^\e-\rho_h^0\|_{L_T^2(L^2(\W_x))} &\leq \|\eta_\rho^\e\|_{L_T^2(L^2(\W_x))} + \|\xi_\rho^\e\|_{L_T^2(L^2(\W_x))} \\
&\lss \|\eta_\rho^\e\|_{L_T^2(L^2(\W_x))} + \|\xi_\rho^\e\|_{L_T^\infty(L^2(\W_x))}
\end{align*}
which we can bounded by \eqref{eqn:rho-and-J-stab-proj} and \eqref{eqn:err_bnd}.  The $J_h^\e$ error estimate follows from \eqref{eqn:err_bnd}.  The proof is complete.
\end{proof}

\subsection{Error Estimate in \texorpdfstring{$h$}{h}}

We now focus on the error estimates of the limiting drift diffusion system \eqref{eqn:dde_system-b1}.  Here we guarantee a positive rate of convergence for $k_x\geq 1$.  The polynomial degree restriction is not surprising as traditionally an interior penalty term on $J_{h}^0$ is required is order to obtain a positive rate of convergence for piecewise-constant polynomial approximations (see \cite[Table 2.6]{Riv2008}).  We note that DG discretizations of \eqref{eqn:dde-cont-system} have been studied for the one-dimensional case in \cite{liu2016analysis}; however, their discretization defines the auxiliary variable in the system as a scalar multiple of $\grad_x\rho_h^0$ rather than $J_h^\e$.  Additionally,  their error estimates rely on $\rho_h^0$ being discontinuous in space and the fluxes for $\rho_h^0$ and $J_h^0$ to be alternating so that the Gauss-Radau projection can be utilized.  Since both of these properties do not hold for \eqref{eqn:dde_system-b1}, we include our own error estimates in $h$.

\begin{lemma}\label{lem:err-h}
Suppose $\rho^0\in L_T^2(H^s(\W))$ and $J^0\in L_T^2([L^2(\W_x)]^3)$ satisfy \eqref{eqn:dde-cont-system} for some $s\geq 2$, $\omega\in W^{r,\infty}(\W_x)$, and $E\in L_T^\infty(W^{r,\infty}(\W_x))$ for some $r\geq 1$.  Define $\mu = \min\{r,s-1\}$ and recall $\Cwr$ from \Cref{thm:error_est}.
Then
\begin{align}\label{eqn:h-err-est}
\begin{split}
   &\|\rho_h^0-\rho^0\|_{L_T^2(L^2(\W_x))} + 
   \|J_h^0-J^0\|_{L_T^2(L^2(\W_x))} \\
   &\qquad\lss h_x^{\min\{k+1,s\}-1}\|\rho^0\|_{L_T^2(H^s(\W_x))} 
   + \Cwr h_x^{\min\{k+1,\mu\}-\beta/2}\|\rho^0\|_{L_T^2(H^s(\W_x))}.
\end{split}
\end{align}
\end{lemma}

\begin{proof}
We decompose $e_{h,\rho}:=\rho_h^0-\rho_0$ and $e_{h,J}:=J_h^0-J^0$ as
\begin{subequations}
\begin{align}
        e_{h,\rho}
        &=\xi_{h,\rho}-\eta_{h,\rho}:=(\rho_h^0-\mSh\rho^0)-(\rho^0-\mSh\rho^0) ,\\
        e_{h,J} &= \xi_{h,J}-\eta_{h,J}:=(J_h^0-\mPh J^0)-(J^0-\mPh J^0),
\end{align}
\end{subequations}
where $\mPh$ is the $L^2$-orthogonal  projection onto $V_h$.  Because $\rho^0,J^0$ also solve \eqref{eqn:dde_system-b1}, the differences $e_{h,\rho}$ and $e_{h,J}$ satisfy
\begin{align}\label{eqn:h-err-eqn}
\begin{split}
    \left(\partial_t e_{h,\rho},q_h\right)_{\W_x} - (e_{h,J},\grad_x q_h)_{\W_x} + \left<\lavg e_{h,J}\ravg,\ljmp q_h\rjmp\right>_{\EIx} +  \left<\gamma_I \ljmp e_{h,\rho}\rjmp,\ljmp q_h\rjmp\right>_{\EIx} &= 0 \\
    (\omega e_{h,J},\tau_h)_{\W_x} + \theta(\grad_x e_{h,\rho},\tau_h)_{\W_x} - \theta\left<\ljmp e_{h,\rho}\rjmp,\lavg \tau_h\ravg\right>_{\EIx}
    - (E e_{h,\rho},\tau_h)_{\W_x} &= 0.
\end{split} 
\end{align}
for all $\tau_h\in [V_{x,h}]^3$ and $q_h\in S_{x,h}^0$ if $\beta=0$ and $q_h\in V_{x,h}^0$ if $\beta=1$.  Choosing $q_h=\theta\xi_{h,\rho}$ and $\tau_h=\xi_{h,J}$ in \eqref{eqn:h-err-eqn}  gives 
\begin{align}\label{eqn:sub-err-eqn}
\begin{split}
   &\theta \left(\partial_t \xi_{h,\rho},\xi_{h,\rho}\right)_{\W_x} - \theta(\xi_{h,J},\grad_x\xi_{h,\rho})_{\W_x} + \theta\left<\lavg \xi_{h,J}\ravg,\ljmp\xi_{h,\rho}\rjmp\right>_{\EIx} + 
   \theta\left<\gamma_I \ljmp \xi_{h,\rho}\rjmp,\ljmp \xi_{h,\rho}\rjmp\right>_{\EIx} \\
    & \quad = \theta\left(\partial_t \eta_{h,\rho},\xi_{h,\rho}\right)_{\W_x} 
    - \theta(\eta_{h,J},\grad_x\xi_{h,\rho})_{\W_x} 
    + \theta\left<\lavg\eta_{h,J}\ravg,\ljmp\xi_{h,\rho}\rjmp\right>_{\EIx} + \theta\left<\gamma_I \ljmp \eta_{h,\rho}\rjmp,\ljmp \xi_{h,\rho}\rjmp\right>_{\EIx}, 
\end{split} \\ \label{eqn:sub-err-eqn:2}
\begin{split}
    &(\omega \xi_{h,J},\xi_{h,J})_{\W_x} + \theta(\grad_x\xi_{h,\rho},\xi_{h,J})_{\W_x} - \theta\left<\ljmp\xi_{h,\rho}\rjmp,\lavg \xi_{h,J}\ravg\right>_{\EIx} - (E\xi_{h,\rho},\xi_{h,J})_{\W_x}  \\
    & \quad =(\omega\eta_{h,J},\xi_{h,J})_{\W_x} 
    + \theta(\grad_x\eta_{h,\rho},\xi_{h,J})_{\W_x} 
   - \theta\left<\ljmp\eta_{h,\rho}\rjmp,\lavg \xi_{h,J}\ravg\right>_{\EIx} - (E\eta_{h,\rho},\xi_{h,J})_{\W_x}. 
\end{split}
\end{align}
By properties of the $L^2$-projection. $\left(\partial_t \eta_{h,\rho},\xi_{h,\rho}\right)_{\W_x}=0$ and $(\eta_{h,J},\grad_x\xi_{h,\rho})_{\W_x}=0$.  Adding \eqref{eqn:sub-err-eqn} and \eqref{eqn:sub-err-eqn:2} gives 
\begin{align}\label{eqn:h-err-eqn:1}
\begin{split}
   \tfrac{\theta}{2}\tfrac{\dx[]}{\dx[t]}\|\xi_{h,\rho}\|_{L^2(\W_x)}^2 + 
   \omega_{\min}\|\xi_{h,J}\|_{L^2(\W_x)}^2 &+ 
   \gamma_*\theta\|\left\ljmp\xi_{h,\rho}\rjmp\right\|_{L^2(\W_x)}^2 \leq \theta\left<\lavg \eta_{h,J}\ravg,\ljmp\xi_{h,\rho}\rjmp\right>_{\EIx} \\
   &\quad+ \theta\left<\gamma_I \ljmp \eta_{h,\rho}\rjmp,\ljmp \xi_{h,\rho}\rjmp\right>_{\EIx} + (E\xi_{h,\rho},\xi_{h,J})_{\W_x}  \\
    &\quad+(\omega\eta_{h,J},\xi_{h,J})_{\W_x} + \theta(\grad_x\eta_{h,\rho},\xi_{h,J})_{\W_x} \\
    &\quad- \theta\left<\ljmp\eta_{h,\rho}\rjmp,\lavg \xi_{h,J}\ravg\right>_{\EIx} - (E\eta_{h,\rho},\xi_{h,J})_{\W_x}. 
\end{split}
\end{align}
We now bound the terms on the right hand side of \eqref{eqn:h-err-eqn} using a combination of  Cauchy-Schwarz, trace \eqref{eqn:trace}, inverse \eqref{eqn:inverse}, and Young's inequalities.  Note any integral on $\EIx$ vanishes if $\beta=0$ so we will multiply the edge contributions by $\beta$ to compensate.  Doing this gives us
\begin{align}\label{eqn:h-err-eqn:2}
\begin{split}
   \tfrac{\dx[]}{\dx[t]}\|\xi_{h,\rho}\|_{L^2(\W_x)}^2 + 
   \|\xi_{h,J}\|_{L^2(\W_x)}^2 &+ 
   \beta\|\left\ljmp\xi_{h,\rho}\rjmp\right\|_{L^2(\W_x)}^2 \lss \|\xi_{h,\rho}\|_{L^2(\W_x)}^2\\
   &\quad+\tfrac{\beta}{h_x}\| \eta_{h,J}\|_{L^2(\W_x)}^2 + \tfrac{1}{h_x^2}\| \eta_{h,\rho}\|_{L^2(\W_x)}^2 + \|\eta_{h,J}\|_{L^2(\W_x)}^2.
\end{split}
\end{align}
Since $\omega,E(t)\in W^{r,\infty}(\W)$, then $J^0(t)\in H^{\mu}(\W_x)$ where $\mu=\min\{r,s-1\}$.  From standard finite element interpolation theory (cf \cite{Bre2008}) we have the projection estimates
\begin{align}\label{eqn:h-err-eqn:3}
\begin{split} 
\|\eta_{h,\rho}\|_{L_T^2(L^2(\W_x))} &\lss h_x^{\min\{k+1,s\}}\|\rho^0\|_{L_T^2(H^{s}(\W_x))}, \\
\|\eta_{h,J}\|_{L_T^2(L^2(\W_x))} &\lss \Cwr h_x^{\min\{k+1,\mu\}}\|\rho^0\|_{L_T^2(H^{s}(\W))}.  
\end{split}
\end{align}
Integrating \eqref{eqn:h-err-eqn:2} from $0$ to $T$, applying the bounds \eqref{eqn:h-err-eqn:3}, and invoking Gr\"onwall's lemma we have
\begin{align}\label{eqn:h-err-eqn:4}
\begin{split}
   \|\xi_{h,\rho}\|_{L_T^\infty(L^2(\W_x))}^2 &+ 
   \|\xi_{h,J}\|_{L_T^2(L^2(\W_x))}^2 \lss 
   h_x^{2\min\{k+1,s\}-2}\|\rho^0\|_{L_T^2(H^s(\W_x))}^2 \\
   &\quad+ \Cwr^2h_x^{2\min\{k+1,\mu\}-\beta}\|\rho^0\|_{L_T^2(H^s(\W_x))}^2.
\end{split}
\end{align}
Thus by \eqref{eqn:h-err-eqn:4} and a triangle inequality we have \eqref{eqn:h-err-est}.  The proof is complete.
\end{proof}

We can now prove \Cref{thm:error_est} as a consequence of \Cref{thm:err-eps} and \Cref{lem:err-h}.

\begin{proof}[Proof of \Cref{thm:error_est}]
Using a triangle inequality and applying \Cref{thm:err-eps} and \Cref{lem:err-h} immediately implies the result.  The proof is complete.
\end{proof}

\begin{remark}
Consider an $H^2$ solution $\rho^0$, $k_x=1$, and full upwinding, that is, $\beta=0$.  Then the error in \eqref{eqn:err-eps-h} is $\mathcal{O}(\sqrt{\e/h_x}+h_x)$ which is optimal if we set $h_x=\e^{1/3}$.
\end{remark}
\section{Conclusion}

We have developed two stable discontinuous Galerkin methods for the a linear Boltzmann semiconductor problem, rigorously showed that they are asymptotically preserving, and explicitly showed their limiting discrete drift-diffusion systems as $\e\to 0$.  

Future work includes extending the results presented in this paper to a self-consistent electric field $E$, non-homogeneous inflow boundary data $f_{-}$, and non-isotropic initial data.

\appendix

\section{Technical Lemmas}








\subsection{Proof of \texorpdfstring{\Cref{lem:tech_lb}}{Lemma \ref{lem:tech_lb}}}

\Cref{lem:tech_lb} is a result of the following lemma: 
\begin{lemma} \label{lem:tech_eta}
There exists $\gamma_*>0$ independent of $\e$ and $h_x$ such that
\begin{align}\label{ass-pos}
    \inf_{\substack{\xi\in \R^d \\ \|\xi\|_2 = 1}}\left(v_h\mhalfh,\xi\mhalfh\right)_{\{v:v_h(v)\cdot\xi > 0\}} &> \gamma_*,\\
    \inf_{\substack{\xi\in \R^d \\ \|\xi\|_2 = 1}}\left(\frac{|v_h\cdot\xi|}{2}\mhalfh,\mhalfh\right)_{\W_v} &> \gamma_*. \label{ass-pos-2}
\end{align}
\end{lemma}

\begin{proof}
We first focus on \eqref{ass-pos}.  We will show the function $\gamma:\R^d\to\R$ defined by
\[
 \gamma(\xi) = \left(v_h\mhalfh,\xi\mhalfh\right)_{\{v:v_h(v)\cdot\xi > 0\}}
\]
is lower semi-continuous.  Let $\xi_n\to\xi$.  Define $\gamma_n:\W_v\to\R$ by
\[
    \gamma_n(v) = v_h\cdot\xi_n\mhalfh\mhalfh\chi_{\{v_h(v)\cdot\xi_n > 0\}}
\]
where $\chi_A$ is the indicator function for the set $A$.  By Fatou's Lemma we have
\begin{align}\label{eqn:fatou}
\liminf_{n\to\infty}\gamma(\xi_n) = \liminf_{n\to\infty}\int_{\W_v} \gamma_n(v)\dx[v] \geq \int_{\W_v} \liminf_{n\to\infty} \gamma_n(v) \dx[v].
\end{align}
We claim
\begin{align}\label{eqn:liminf_lem}
    \lim_{n\to\infty} \gamma_n(v) = v_h\cdot\xi\mhalfh\mhalfh\chi_{\{v_h(v)\cdot\xi > 0\}}
\end{align}
for all a.e.\ $v\in\W_v$.  Let $v\in\W_v$ with $v_h(v)\cdot\xi < 0$.  Then eventually we have $v_h(v)\cdot\xi_k < 0$ for all $k$ sufficiently large.  Thus the indicator function evaluates to zero and $\gamma_k(v)=0$; thus \eqref{eqn:liminf_lem} holds.  Let  $v\in\W_v$ with $v_h(v)\cdot\xi < 0$.  Then similarly $v_h(v)\cdot\xi_k > 0$ for all $k$ sufficiently large.  Hence the indication function evaluates to 1 and we can pass the limit to show \eqref{eqn:liminf_lem} holds.  Since  the set $\{v:v_h(v)\cdot\xi=0\}$ is a set of measure zero,  \eqref{eqn:liminf_lem} holds for all a.e.\ $v\in\W_v$.  Using \eqref{eqn:liminf_lem} we continue \eqref{eqn:fatou} to obtain
\[
\liminf_{n\to\infty}\gamma(\xi_n) \geq \int_{\W_v} \liminf_{n\to\infty} \gamma_n(v) \dx[v] = \int_{\W_v} v_h\cdot\xi\mhalfh\mhalfh\chi_{\{v_h(v)\cdot\xi > 0\}}\dx[v] = \gamma(\xi). 
\]
Therefore $\gamma$ is lower semi-continuous.  Since $\gamma>0$ on the compact unit sphere, it obtains a positive minimum.  Thus the first equality of \eqref{ass-pos} holds.

For \eqref{ass-pos-2}, we note that the function $\xi\to \big(\tfrac{|v_h\cdot\xi|}{2}\mhalfh,\mhalfh\big)_{\W_v}$ is Lipschitz continuous.  Since it is also positive on the compact unit sphere, it obtains a positive minimum.  The proof is complete.  
\end{proof}

\section{Maxwellian Approximation}\label{sect:maxwellian_discussion}

\Cref{lem:M_bnd} gives precise bounds for the discrete 1D root Maxwellian constructed in \Cref{rmk:interp_maxwell}.

\begin{lemma}\label{lem:M_bnd}
Let $L>0$ with 
\begin{equation}\label{eqn:L_ass}
L\geq\sqrt{\theta}, 
\end{equation}
and suppose $\W_v=[-L,L]$.  Furthermore, assume
\begin{equation}\label{eqn:M_bnd_ass}
    h_v^2\leq \frac{4}{\sqrt{3}}\theta.
\end{equation}
Let $Q_{h_v}:C^0(\overline{\W_v})\to S_{v,h_v}$ with $k_v=1$ be the piecewise linear nodal Lagrange interpolant.  Define 
\begin{equation}\label{eqn:scaled_lag_def}
\widetilde{Q}_{h_v}:=\frac{Q_{h_v}u}{\|Q_{h_v}u\|_{L^2(\W_v)}}.
\end{equation}
For $i=1,\ldots,3$, let $M_{h,i}^{\frac{1}{2}}=\widetilde{Q}_h(M_{i}^{\frac{1}{2}})$ where $M_{i}^{\frac{1}{2}}(v_i)$ is defined in \eqref{eqn:maxwell_approx}.  Then $M_{h,i}^{\frac{1}{2}}$ is positive, continuous, and satisfies \Cref{ass-discrete-root-Max}.\ref{ass-mass}, \Cref{ass-discrete-root-Max}.\ref{ass-symm}, \Cref{ass-discrete-root-Max}.\ref{ass-momen}.  Moreover, we have the following approximation results:
\begin{align}\label{eqn:approx_maxwell_bnd_l2}
    \|M_{i}^{\frac{1}{2}}-\widetilde{Q}_{h_v}M_{i}^{\frac{1}{2}}\|_{L^2(\W_v)} &\leq \frac{5}{2}\left(1-\erf(\tfrac{L}{\sqrt{2\theta}})^{1/2}\right) + 5h_v^2\frac{\sqrt{3}}{8\theta}, \\
    \|\partial_v(M_{i}^{\frac{1}{2}}-\widetilde{Q}_{h_v}M_{i}^{\frac{1}{2}})\|_{L^2(\W_v)} &\leq \frac{5}{2}\left(1-\erf(\tfrac{L}{\sqrt{2\theta}})^{1/2}\right) + \frac{5}{2}h_v^2\frac{\sqrt{3}}{16\theta^{3/2}} +\frac{5}{2}\frac{\sqrt{3}}{\sqrt{2}}\frac{1}{4\theta}h_v.
    \label{eqn:approx_maxwell_bnd_h1}
\end{align}
\end{lemma}

\begin{proof}
    For ease of notation set $\mathcal{M}(v)=M_{i}^{\frac{1}{2}}(v)$.  Equation \eqref{eqn:L_ass} gives us the estimate
    \begin{equation}\label{eqn:M_bnd:1}
        \frac{16}{25} \leq \erf(\tfrac{1}{\sqrt{2}}) \leq \erf(\tfrac{L}{\sqrt{2\theta}}) < 1. 
    \end{equation}
    Direction calculation and \eqref{eqn:L_ass} yields
    \begin{subequations}  \label{eqn:M_bnd:2}
    \begin{align}
        \|\mM\|_{L^2(\W_v)}^2 &= \erf(\tfrac{L}{\sqrt{2\theta}}) \leq 1. \label{eqn:M_bnd:2a} \\
        \|\partial_v\mM\|_{L^2(\W_v)}^2 &= \frac{1}{4\theta}\erf(\tfrac{L}{\sqrt{2\theta}}) - \frac{L}{2\sqrt{2\pi\theta^3}}\exp\left(\tfrac{-L^2}{2\theta}\right) \leq \frac{1}{4\theta}\erf(\tfrac{L}{\sqrt{2\theta}}). \label{eqn:M_bnd:2b} \\
        \|\partial_v^2\mM\|_{L^2(\W_v)}^2 &= \frac{3}{16\theta^2}\erf(\tfrac{L}{\sqrt{2\theta}}) - \frac{L(\theta-L^2)}{8\sqrt{2\pi\theta^7}}\exp\left(\tfrac{-L^2}{2\theta}\right) \leq \frac{3}{16\theta^2}\erf(\tfrac{L}{\sqrt{2\theta}}) \label{eqn:M_bnd:2c}
    \end{align}
    \end{subequations}
    We can give precise bounds on the interpolation error $\|\mM-\widetilde{Q}_{h_v}\mM\|$ from the proof in \cite[Theorem (0.4.5)]{Bre2008}, \eqref{eqn:M_bnd:2c}, and \eqref{eqn:M_bnd:2a}:
    \begin{align}\label{eqn:M_bnd:3}
    \begin{split}
        \|\mM-{Q}_{h_v}\mM\|_{L^2(\W_v)} +& \frac{1}{\sqrt{2}}h_v\|\partial_v(\mM-{Q}_{h_v}\mM)\|_{L^2(\W_x)} \\
        &\leq \frac{1}{2}h_v^2\|\partial_v^2\mM\|_{L^2(\W_v)} \leq \frac{\sqrt{3}}{8\theta}h_v^2\|\mM\|_{L^2(\W_v)} \leq \frac{\sqrt{3}}{8\theta}h_v^2.
    \end{split}
    \end{align}
    Using the reverse triangle inequality, \eqref{eqn:M_bnd:3}, and \eqref{eqn:M_bnd:1}, we obtain
    \begin{align}\label{eqn:M_bnd:4}
    \begin{split}
        \|Q_h\mM\|_{L^2(\W_v)} \geq (1-\tfrac{\sqrt{3}h_v^2}{8\theta})\|\mM\|_{L^2(\W_v)} \geq \tfrac{4}{5}(1-h_v^2\tfrac{\sqrt{3}}{8\theta})
    \end{split}
    \end{align}
    From \eqref{eqn:M_bnd_ass}, we have
    \begin{equation}
        1-\frac{\sqrt{3}}{8\theta}h_v^2 \geq \frac{1}{2},
    \end{equation}
    and hence
    \begin{align}\label{eqn:M_bnd:5}
    \begin{split}
        \|Q_h\mM\|_{L^2(\W_v)} \geq \frac{2}{5},\quad\text{ and }\quad \frac{1}{\|Q_h\mM\|_{L^2(\W_v)}} \leq \frac{5}{2}.
    \end{split}
    \end{align}
    We now show \eqref{eqn:approx_maxwell_bnd_l2}. Define 
    \begin{equation}\label{eqn:alpha_M}
        \alpha_\mM =  \|Q_{h_v}\mM\|_{L^2(\W_v)}.
    \end{equation} 
    Using in definition of $\widetilde{Q}_{h_v}$ we obtain
    \begin{equation}\label{eqn:M_bnd:6}
        \|\mM-\widetilde{Q}_{h_v\mM}\|_{L^2(\W_v)} = \frac{1}{\alpha_\mM}\|\alpha_\mM\mM-Q_{h_v}\mM\|_{L^2(\W_v)}.
    \end{equation}
    Adding and subtracting key quantities and several uses of the standard and reverse triangle inequalities both yield
    \begin{align}\label{eqn:M_bnd:7}
    \begin{split}
        \|\alpha_\mM\mM-Q_{h_v}\mM\|_{L^2(\W_v)} 
        &\leq |1-\alpha_\mM|\|\mM\|_{L^2(\W_v)} + \|\mM-Q_{h_v}\mM\|_{L^2(\W_v)} \\
        &\leq \big( |1-\|\mM\|_{L^2(\W_v)}| + |\|\mM\|_{L^2(\W_v)}-\alpha_\mM| \big)\|\mM\|_{L^2(\W_v)} \\
        &\quad+ \|\mM-Q_{h_v}\mM\|_{L^2(\W_v)} \\
        &\leq |1-\|\mM\|_{L^2(\W_v)}|\|\mM\|_{L^2(\W_v)} \\
        &\quad+ \|\mM-Q_{h_v}\mM\|_{L^2(\W_v)} \|\mM\|_{L^2(\W_v)} \\
        &\quad+ \|\mM-Q_{h_v}\mM\|_{L^2(\W_v)}.
    \end{split}
    \end{align}
    The terms on the right hand side of \eqref{eqn:M_bnd:7} can be bounded using \eqref{eqn:M_bnd:2} and \eqref{eqn:M_bnd:3}.  These estimates along with \eqref{eqn:M_bnd:6} and \eqref{eqn:M_bnd:5} yield \eqref{eqn:approx_maxwell_bnd_l2}.  For \eqref{eqn:approx_maxwell_bnd_h1}, a similar $H^1$ estimate to \eqref{eqn:M_bnd:7} can be formed, namely:
    \begin{align}\label{eqn:M_bnd:8}
    \begin{split}
         \|\partial_v(\alpha_\mM\mM-Q_{h_v}\mM)\|_{L^2(\W_v)} &\leq |1-\|\mM\|_{L^2(\W_v)}| \|\partial_v\mM\|_{L^2(\W_v)} \\
         &\quad + \|\mM-Q_{h_v}\mM\|_{L^2(\W_v)} \|\partial_v\mM\|_{L^2(\W_v)} \\
         &\quad+ \|\partial_v(\mM-Q_{h_v}\mM)\|_{L^2(\W_v)}. 
    \end{split}
    \end{align}
    Estimate \eqref{eqn:M_bnd:8} along with the estimates above yield \eqref{eqn:approx_maxwell_bnd_h1}.  The proof is complete.
\end{proof}

\bibliographystyle{abbrv} 
\bibliography{references} 
\end{document}